\documentclass[journal,10pt,draftclsnofoot,onecolumn,onehalfspacing]{IEEEtran}

\usepackage{cite}
\usepackage{amsmath,amssymb,amsfonts}

\usepackage{psfrag}

\usepackage{graphicx}

\usepackage{textcomp}
\usepackage{xcolor}
\usepackage{cite}
\usepackage[square,numbers]{natbib}
\usepackage{amsmath}
\usepackage{amssymb}
\usepackage{amsfonts}
\usepackage{amsthm}
\usepackage{algorithm,algcompatible,amsmath}
\usepackage{color}
\usepackage{comment}
\usepackage{amssymb}
\usepackage{graphicx}
\usepackage{lipsum}
\usepackage[colorinlistoftodos]{todonotes}
\usepackage{textcomp}
\usepackage{xcolor}
\usepackage[numbers]{natbib}

\newtheorem{thm}{Theorem}
\newtheorem{prop}{Proposition}

\newtheorem{remark}{Remark}

\newtheorem{lem}{Lemma}

\newtheorem{define}{Definition}
\newtheorem{assumption}{Assumption}

\usepackage{soul}
\usepackage{hyperref}

\algnewcommand\INPUT{\item[\textbf{Input:}]}
\algnewcommand\OUTPUT{\item[\textbf{Output:}]}

\begin{document}

\title{Straggler-Robust Distributed Optimization with the Parameter Server Utilizing Coded Gradient
\thanks{}
}

\author{Elie Atallah, Nazanin Rahnavard~\IEEEmembership{Senior Member,~IEEE}, and Chinwendu Enyioha
        \\
        \IEEEauthorblockA{{Department of Electrical and Computer Engineering}\\
{University of Central Florida, Orlando, FL}\\
Emails: \{elieatallah@knights., nazanin@eecs., cenyioha@\}ucf.edu}
\thanks{* This material is based upon work supported by the National Science Foundation under grants CCF-1718195.}
}

\maketitle

\begin{abstract}
Optimization in distributed networks plays a central role in almost all distributed machine learning problems. In principle, the use of distributed task allocation has reduced the computational time, allowing better response rates and higher data reliability. However, for these computational algorithms to run effectively in complex distributed systems, the algorithms ought to compensate for communication asynchrony, and network node failures and delays known as stragglers. These issues can change the effective connection topology of the network, which may vary through time, thus hindering the optimization process. In this paper, we propose a new distributed unconstrained optimization algorithm for minimizing a strongly convex function which is adaptable to a parameter server network. In particular, the network worker nodes solve their local optimization problems, allowing the computation of their local coded gradients, and send them to different server nodes. Then each server node aggregates its communicated local gradients, allowing convergence to the desired optimizer. This algorithm is robust to network worker node failures or disconnection, or delays known as stragglers. One way to overcome the straggler problem is to allow coding over the network. We further extend this coding framework to enhance the convergence of the proposed algorithm under such varying network topologies. Finally, we implement the proposed scheme in MATLAB and provide comparative results demonstrating the effectiveness of the proposed framework.

\end{abstract}


\section{Introduction}
{M}{any} problems in distributed systems over the cloud, or in wireless ad hoc networks \citep{johansson2008distributed,rabbat2004distributed,ram2009distributed}, are formulated as convex optimization programs in a parallel computing scheme. Depending on the structure of these networks, i.e., centralized, decentralized or fully distributed, the optimization techniques are adapted to accommodate such structures. However, the malfunctioning of processors directly impacts the overall performance of parallel computing. Dealing with this malfunctioning is referred to as the \emph{straggling problem}. Many applications, whether over the cloud or in local distributed networks, have experienced considerable time delays, due in part to this straggling problem. Asynchronous \cite{li2014communication}, \cite{ho2013more} and synchronous algorithms \cite{zaharia2008improving}, \cite{chen2016revisiting} have been proposed to overcome this problem.  
While Lee et al. \cite{lee2016speeding} and Dutta et al. \cite{dutta2016short} describe techniques for mitigating stragglers in different applications, a recent work by Tandon et al. \cite{tandon17a} focused on codes for recovering the batch gradient of a loss function (i.e., synchronous gradient descent). 
Specifically, a coding scheme in \cite{tandon17a} was proposed, enabling a distributed division of tasks into uncoded (naive) and coded parts. This partition alleviates the effect of straggling servers in a trade-off among computational, communication complexity and time delay.
This novel coding scheme solves this problem by providing robustness to partial failure or delay of nodes.

\section{Problem Setup}\label{sec:problem}

We consider a network of $ n $ server nodes indexed by $ V=\{1,2,\ldots,n\} $ and $ m  $ worker nodes on a parameter server platform using a multi-bus multiprocessor system with shared memory. The objective is to solve a minimization problem where the solution set $ \mathcal{X}^{*} $ belongs to a convex set $ \mathcal{X} $where the gradient is bounded. To that end we require that the global function $ f $ is divided into $ p $ partitions with arbitrary number of replication for each. Thus, we require arbitrary interleaved connections according to availability. Meanwhile, in Figure \ref{Fig.1} due to the complexity of the system schematic we show a restrictive setup where each server is connected to its own partition all the time where the redundancy of each partition is unity. We show in \ref{Fig.2} a more elaborate schematic for the general case when discussing the push/pull steps for one server.

%
The optimization problem is the unconstrained optimization given by  
\begin{equation}\label{eq1}
\begin{split}
\hat{{\bf x}} = arg & \min_{{\bf x} \in  \mathbb{R}^{N}}  {f({\bf x})} = \sum_{i=1}^{n} f^{(i)}(x) \\
\end{split}
\end{equation} 
\\

\begin{figure}
\includegraphics[scale=0.5,bb=0 0 500 500]{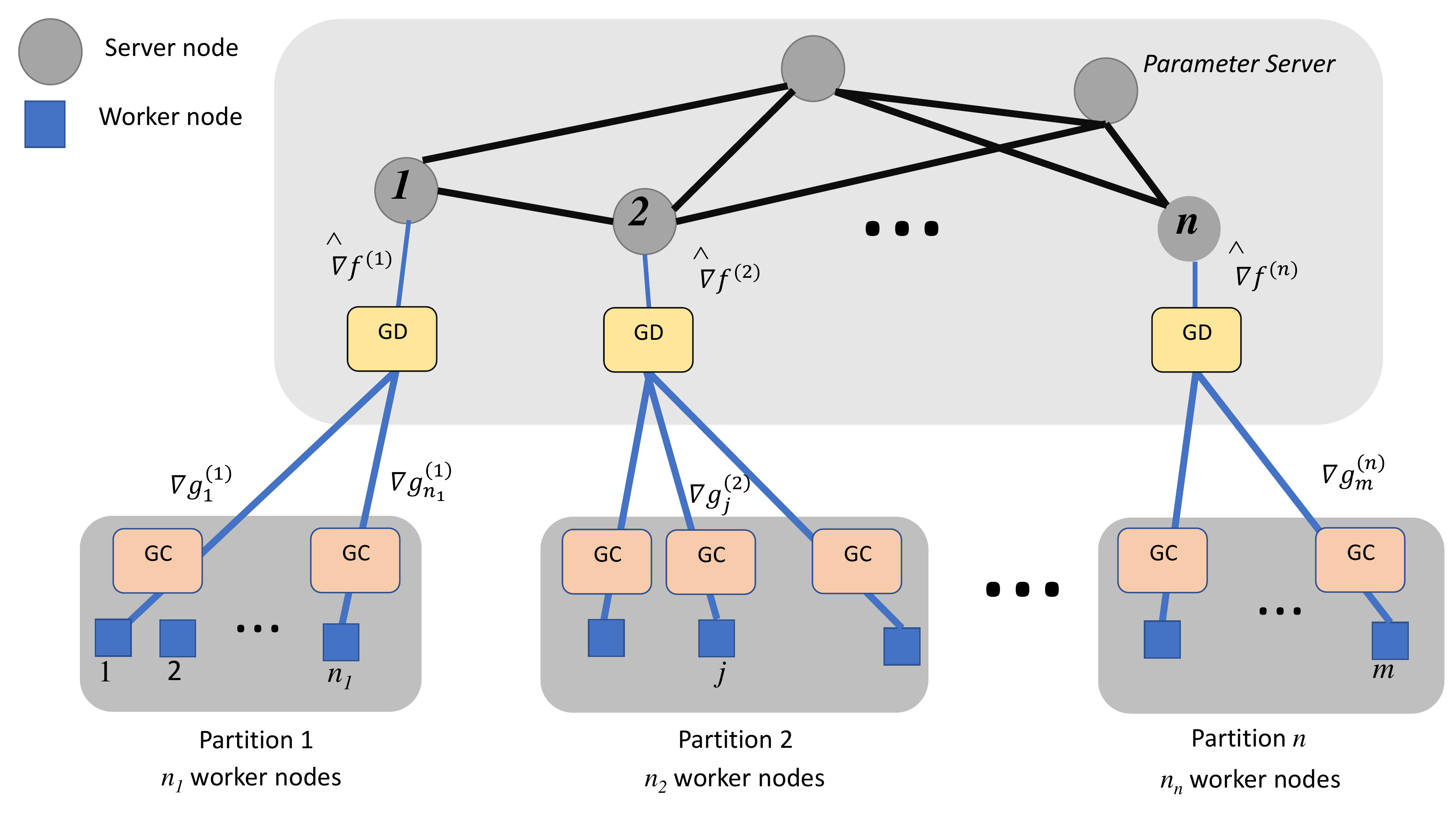}
\label{Parameter Server}
\caption{Parameter server network in the special case with $ n $ server nodes and $ m $ workers nodes. The worker nodes are  divided into $ p = n  $ partitions where partition $ i $ has $ n_{i} $ workers and each server is connected to its unique partition at all iterations.}
\label{Fig.1}
\end{figure}


Due to the random behavior of this distributed system we adapt a similar approach to analyze the unconstrained Problem (\ref{eq1})
as in \citep{lee2013distributed}, Lee and Nedi\'c. In their approach the authors solve a constrained distributed optimization problem based on stochastic gradient descent (i.e., partial gradients) and random projections
\vspace{-0.3cm}
\begin{equation}
\begin{split}
\min_{{\bf x} \in \underset{i}{\cap} \mathcal{X}_{i}}f({\bf x}) = \sum_{i=1}^{n} f^{(i)}(x)      
\end{split}
\end{equation}
on $ n $ nodes, where the local optimization problems are carried on the nodes themselves and $ \mathcal{X}_{i} $ are convex sets such that $ \underset{i}{\cap} \mathcal{X}_{i} $ is the constraint set.
We note that we use a similar approach as their random approach which employs the supermartingale theorem due to the structure of our distributed parameter servers and the randomness of connections along the multi-bus between servers and workers. This analysis also utilizes the supermartingale theorem, however the projection is on the set $ \mathcal{X} $ of bounded gradients which has nothing in common with any projection step in the algorithm and is their only for the mathematical analysis.

In this work to solve the unconstrained distributed optimization problem 
\eqref{eq1}
we use a gradient descent method without projections through utilizing \emph{Straggler-Robust Distributed Optimization} (SRDO) Algorithm.
As its name infers our algorithm has the extra feature of being robust to stragglers. To this end, coded local gradients $ \nabla{g}_{j}^{(i)} $ (i.e., local coded gradient of worker $ j $ in partition $ i $) are carried on the different partitions of worker nodes and decoded to the partition's gradient $ \nabla{f}^{(i)} $ (i.e., the gradient of the function of the partition connected to server $ i $, cf. Remark 1) on the connected server nodes (cf. Fig.1).


Meanwhile, after dividing the load into different partitions each partition $ i \in \{1,\ldots,p\} $ is distributed with an arbitrary redundancy among the workers. Thus, each partition replica utilizes a gradient coding similar to that of \citep{tandon17a} to enable robustness to an allowed number of stragglers per each partition. 
More specifically, under a global clock each server sends its estimate $ {\bf v}_{i}(k) $ to possibly any worker under the interleaved shared memory multi-bus system.
Each worker then calculates its coded gradient relative to the partition replica it belongs to and using the estimate it received from a server. Then under a global clock, the servers synchronously call for different partitions and decode their local functions gradients $ \nabla{f}^{(i)} $ to compute their estimates $ {\bf x}_{i}(k) $.

Thus, different servers are working in a synchronous manner to compute their solution $ {\bf x}_{i}(k) $ through a gradient descent step decoded from partial coded gradients of the connected partition replica. However, each partition worker calculates its coded gradient by evaluating the coded load on the weighted averages $ {\bf v}_{q}(k-k^{'}_{q}) $ which can be from different servers and different time step with bounded delay, (i.e., $ q \in \{1,\dots,n\} $ and $ 0 \leq  k^{'}_{q}  \leq H $.


\begin{algorithm}
\begin{algorithmic}
\State{\textbf{Given: $ f({\bf x})$} = $ \sum_{i=1}^{p}f^{(i)}(x)$}\\
\State {e.g. Least Squares Prototype: ${\bf A} \in \mathbb{R}^{M \times N}$, $ {\bf y} \in \mathbb{R}^{M} $ where $ f({\bf x})= \frac{1}{2}\|{\bf y} - {\bf A}{\bf x} \|^{2}_{2}$.}
\State{\textbf{Find:} $ {\bf x}^{*} \in \mathbb{R}^{N} $. }
\State \textbf{Initialization:} Each server $ i $ sets $ v_{i}(0) $ to an arbitrary random vector. \;
\State \textbf{Set:} $ tol $, $ \epsilon_{i}= \emph{large number} $. \\
\textbf{while $\epsilon_{i} > tol $} \\ {
\textbf{At each server $ i $:} \\
\textbf{Push step of iteration $ k $:} \& \\
\text{Send $ v_{i}(k) $ to \& arbitrary workers} \\
\begin{subequations}\label{eq2}
\begin{equation}
\begin{split}
\textbf{Pull step at iteration $ k $:} \ \ \ \ \ & \\
\textit{Decoding the partition gradient} \ \widehat{\nabla{{f}}}^{(i)}(k) \\
\textit{from the sent local coded gradients} \ \nabla{g^{(i)}_{j}(k)}
\end{split}
\end{equation}
\begin{equation}
\begin{split}
\textbf{Iteration $ k+ 1 $:} & \\
{\bf x}_{i}(k+1)= & {\bf v}_{i}(k)-\alpha_{k}\widehat{\nabla{{f}}}^{(i)}(k)
\end{split}
\end{equation}
\begin{equation}
\begin{split}
{\bf v}_{i}(k+1)= & \sum_{j=1}^{n}w_{ij}(k+1){\bf x}_{j}(k+1).
\end{split}
\end{equation}
\end{subequations}
$ \epsilon_{i}= \| {\bf v}_{i}(k+1)-  {\bf v}_{i}(k) \|_{2} $} \\
\textbf{end}
\OUTPUT { \  \ \ $ {\bf x}^{*} = {\bf x}_{i}(k) $.}
\caption{Algorithm Updating at Each Server SRDO Algorithm}
\end{algorithmic}
\end{algorithm}

Next, we briefly outline the steps of our algorithm implemented at the server and worker nodes; and elaborate on these steps in Section \ref{sec:algorithm}.
After the distribution of the load in the distribution step accordingly. 

On the servers side: \\
{\bf Push step:} Under a global clock each server $ i $ sends a message containing the weighted average $ {\bf v}_{i}(k) $ to an arbitrary number of workers (i.e., that could be in different partitions). Then each worker starts computing its local coded gradient in the worker computation of coded gradient step.


{\bf Pull step:} Under a global clock each server gets activated and calls for coded gradients from an arbitrary partition. The workers from that partition $ \iota $ will send the coded gradients $ \nabla{g}_{j}^{(\iota)} $ to the connected server $ i $. (i.e., some workers are stragglers thus don't send their coded gradients). Server $ i $ after receiving the coded gradients from the connected partition $ \iota $ decodes the partition gradient $ \nabla{f}^{(i)} $.
\begin{remark}
Here, we identify the partition $ \iota $ at this connection instant with server $ i $.
\end{remark}
Then the server calculates the estimate $ {\bf x}_{i}(k+1) $.

{\bf Consensus step:} Under a global clock each server gets activated again and computes its weighted average $ {\bf v}_{i}(k+1) $ from its connected servers $ x_{j}(k+1)$ according to $ (\ref{eq2}c)$. 
Then the algorithm state goes back to the push step and continues henceforth until convergence. 

On the workers side: \\
{\bf Worker Computation of Coded Gradient:} When a worker $ i $ of a partition $ \iota $ receives a weighted average $ {\bf v}_{j}(k) $ from a server $ j $ it gets activated and starts calculating its coded gradient $ \nabla{g}_{i}^{(\iota)} $ relative to the coding scheme used on partition $ \iota $. Workers can get delayed in their computation of coded gradients and need to send their computed gradients to a connected server to their partition at a subsequent time instant. Here, there is one aspect of asynchronous behavior in the algorithm that influence the computed partition gradient used at the connected server. That is, we don't require that the coded gradients of which the partition's (or connected server) partial gradient is decoded to be of the consecutive previous instant weighted averages evaluations but rather of possibly older weighted averages evaluations.

This approach is tolerant to the allowed number of stragglers, and it is also robust to more than the allowed number of stragglers.

We prove convergence in the general case where $ n $, $ m $ and $ p $ are arbitrary.
\subsection{Assumptions on the Convex Functions}

\begin{assumption}\label{A3.1}
Let the following conditions be satisfied:

\begin{itemize}
\item[\textbf{(a)}]{ Each function  $\mathit{f^{(i)}} : \mathbb{R} ^{N} \rightarrow \mathbb{R}  $  is convex. }
\item[\textbf{(b)}]{ The functions $\mathit{f^{(i)}}$, $i \in {1,2,\dots,n}$, are differentiable and have Lipschitz gradients with a constant $ L $ over $\mathbb{R}^{N}$,
\begin{equation*}
\begin{split}
\|\nabla{f^{(i)}}({\bf x})-\nabla{f^{(i)}}({\bf y})\| \leq L\|{\bf x}- {\bf y}\|
\end{split}
\end{equation*}
for all ${\bf x}$, ${\bf y} \in \mathbb{R}^{N}$. }
\item[\textbf{(c)}]{ The gradients $\nabla{f^{(i)}}({\bf x})$, where $i \in V $ are bounded over the set $\mathcal{X} $ where $ \mathcal{X}^{*} =\{ {\bf x} | {\bf x} = \arg \min f({\bf x}) \} \subset \mathcal{X} $; i.e., there exists a constant $G_{f}$ such that $\| \nabla{f^{(i)}({\bf x})} \| \leq G_{f} $ for all $ {\bf x} \in \mathcal{X}$ and all $i \in V $.}

\end{itemize}
\end{assumption}

The assumed structure on $f$ is typical for problems of this kind and enables a detailed convergence analysis.

Next, we make the following assumptions about the server-server edge weights in the consensus step.

\begin{assumption}\label{A3.3}[\textbf{Row Stochastic}]
For all $k \geq 0$, we have: 

The matrices $ w_{ij}(k) $ in (\ref{eq2}c) are chosen such that $ w_{ij} = {\bf W}(k) $ depending on the network server-connection topology in a way that allows consensus.\\
$ \textbf{ (b) }  $ $\sum_{j=1}^{n}[{\bf W}(k)]_{ij}=1$ for all $i \in V$. \\
$ \textbf{ (c) } $ There exists a scalar $\nu \in (0,1)$ such $[{\bf W}(k)]_{ij} \geq \nu$ if $[{\bf W}(k)]_{ij} > 0$. \\
$\textbf{ (d) } $ $\sum_{i=1}^{n}[{\bf W}(k)]_{j} \leq 1 - \mu $ for all $j \in V$ and $ 0 < \mu < 1 $. \\
$\textbf{ (e) } $ If server $ i $ is disconnected from server $ j $ at instant $ k $, then $[{\bf W}(k)]_{ij} = 0$.
\end{assumption}

\begin{remark}
Notice that for the matrices $ {\bf W}(k) $ we have for $ b_{i} = \sum_{j=1}^{n}[{\bf W}(k)]_{ij} a_{j} $ that
\begin{equation}
\begin{split}
\sum_{i=1}^{n} \| b_{i} \| \leq \sum_{i=1}^{n} \sum_{j=1}^{n}[{\bf W}(k)]_{ij} \| a_{j} \| \leq ( 1-\mu ) \sum_{j=1}^{n} \| a_{j} \| \leq \sum_{j=1}^{n} \| a_{j} \|
\end{split}
\end{equation}
and we are going to use either inequality as needed in our analysis.
\end{remark}

\begin{assumption} Bounded Delayed Evaluation and Gradient Computation:

We assume the decoding of gradient $ \nabla{f}^{(i)}(k) $ at time $ k $ decoded from coded gradients evaluated of weighted averages $ {\bf v}_{i} (k - k^{'}) $ where $  0 \leq k^{'} \leq H $.
We assume the use of stale gradients in gradient computation scenario 3 of weighted average evaluations of instants $ k - k ^{''} $ where $  0 \leq k^{''} \leq \kappa $ to be more explicit, i.e., that are of the global iteration $ k - \bar{k} $ where $ \bar{k} \leq k^{''} \leq \bar{k} + H $. Without a loss of generality, we assume $ H = \kappa $.
\end{assumption}

\begin{assumption} Choice of Partition by Server in Pull Step

Each server $ i $ gets connected to a partition $ (i) $ out of the $ p $ partitions with a probability $ \gamma_{(i)} $ and gets no connection with any partiyion with a probability $ \gamma_{(0)} $.
\end{assumption}

\begin{assumption} Diminishing Coordinated Synchronized Stepsizes

The stepsizes $ \alpha_{i,k} $ of server $ i $ are coordinated and synchronized where $ \alpha_{i,k}=\alpha_{k} > 0 $ and , $  \alpha_{k} \rightarrow 0 $. $ \sum_{k=0}^{\infty} \alpha_{k} = \infty $, $ \sum_{k=0}^{\infty} \alpha_{k}^{2} < \infty $. For example, a unanimous stepsize $ \alpha_{i,k} = \alpha_{k} = \frac{1}{k+1} $ among all servers $ i $ per iteration $ k $.
\end{assumption}

\section{Main Algorithm: Stragglers Robust Distributed Optimization Algorithm (SRDO)}\label{sec:algorithm}

In solving problem \eqref{eq1}, we propose a synchronous iterative gradient descent method. This method is robust to an allowed number of stragglers, and is also valid for a varying network topology with more than the allowed number of stragglers as we are going to show. \\
An appropriate implementable platform for the algorithm is a multi-bus distributed parameter server shared memory network.
The network is equipped with a universal clock that synchronizes the actions of its nodes.

$ \textbf{Initialization:} $ Each server node $ i $ at global iteration $ k = 0 $ begins with random weighted average $ {\bf v}_{i}(0) \in \mathbb{R}^{N} $ for $ i \in \{1,2,\dots,n\} $, and sends ${\bf v}_{i}(0)$ to arbitrary number of worker nodes in the \underline{push step}, as shown in Fig. 2. Each worker node $ j $ of a partition $ u $ then solves for its partition gradient in a coded manner using its received ${\bf v}_{i}(0)$, finding a coded local gradient $(\nabla{g_{j}}^{(u)})^{T}={\bf B_{j}}^{(u)} \overline{\nabla{{f}}^{(u)}} $ evaluated of $ v_{i}(0) $ where $ f^{(u)} $ corresponds to the partition $ u $ function such that $ f^{(u)} = \sum_{l=1}^{n_{u}}f^{(u)}_{l} $ 
and  
$ \overline{\nabla{{f}}^{(u)}}=[(\nabla{f}^{(u)}_{1})^{T} \; (\nabla{f}^{(u)}_{2})^{T} \; \ldots \; (\nabla{f}^{(u)}_{n_{u}})^{T}]^{T} $. \\
Then in the pull step of iteration $ k $, each server node $ i $ according to the global clock gets activated and aggregates those local coded gradients received from its partition $ \iota $ worker nodes and finds the partition gradient $ \nabla{f^{(\iota)}} $ (i.e., notice that we identified partition $ \iota $ with server $ i $) in the \underline{pull step}, as shown in Fig. 2, (i.e., decoding the partition's gradient in the respective network topology scenario, i.e., related to the number of stragglers, and according to the used gradient computation scenario).
Afterward, the algorithm adapts the values ${\bf v}_{i}(0)$ at the server nodes by that connected partition computed gradient if it exists, and finds the estimates ${\bf x}_{i}(1)$ according (\ref{eq2}b). It then forms the weighted averages ${\bf v}_{i}(1)$ according to (\ref{eq2}c), and sends them  to the worker nodes, as the cycle continues henceforth until the algorithm converges to the optimizer. We prove the convergence in Section VII.

\begin{remark}\label{remark2}
Each worker node $ j $ of partition $ \iota $ uses the estimate ${\bf v}_{i}(k)$ and finds the coded local gradient $\nabla{g}_{j}^{(\iota)}(k)({\bf v}_{i}(k))$ corresponding to its local optimization at ${\bf v}_{i}(k)$; i.e.,  here, $g_{j}^{(\iota)}(k)$ is a function employed due to the coding scheme, and $\nabla{g_{j}^{(\iota)}(k)}$ corresponds to $ g_{j}^{(\iota)}(k)
=\sum_{q=1}^{n_{\iota}}{\bf B}^{(\iota)}_{j,q}f^{(\iota)}_{q}(k) $. The function $f^{(\iota)}_{q}$ corresponds to $ q \in \{1,2,\dots,n_{\iota}\}$, where $ \sum_{q=1}^{n_{\iota}}f^{(\iota)}_{q}=f^{(\iota)} $, and $ n_{\iota} $ is the number of worker nodes in partition $ \iota $.
In the pull step, each server node uses the received coded local gradients that are employed at probably different estimates ${\bf v}_{i}(k-k')$ of different time instants $ k' \leq \kappa $, and tries to decode the partition gradient of the function $f^{(\iota)}$ by using a different coded scheme according to a specific partition sub-partition, (i.e., which need not to be unanimous to all partitions). 
We define the set of connected nodes at iteration $k$ to server node $ i $ as $\Gamma_{i}(k)$, and thus, the set of stragglers to node $ i $ as $\Gamma^{\complement}_{i}(k) \triangleq \{ 1,2,\dots,n_{i} \} \setminus \Gamma_{i}(k)$. 
\end{remark}

\begin{remark}
It is worth noting the following about the synchronous behavior of SRDO:
\begin{itemize}
\item If it happens that a server node $ j $ still didn't receive the coded gradients of any of the partitions and is unable to decode the partition's partial gradient and calculate its estimate $ {\bf x}_{j}(k+1) $ then $ {\bf x}_{j}(k+1)={\bf v}_{j} $.
\item If at the push step of iteration $ k $ a worker is unable to receive any of the weighted averages $ {\bf v}_{q}(k-k^{'}_{q}) $ for all $ 0 \leq k^{'}_{q} \leq H $, where $ q \in \{1,\ldots,n\} $, that allows it to compute its local coded gradient in time  before the pull step for the same instant $ k $, then that worker is considered a straggler. 
\item Each worker has to finish its computation before interacting with another server.
In that respect, gradients evaluated at previous weighted averages can still be used by a server as long as the worker would send its coded gradient when it finishes computation at the time of the synchronous update. Moreover, a similar scenario is when these coded gradients from prior are kept in memory of a server and are used in the subsequent updates. The benefit of the latter scenario on the prior one is that it mitigates the effect of stragglers at the instant of update.
\end{itemize}
\end{remark}

\subsection{Remark on the Computation of the Gradient under Different Scenarios}
We distinguish three different scenarios for different partition's gradient computation scenarios:
\\

$ \textbf{Scheme \ 1:} $ \\In this scheme, the number of all partition's worker nodes disconnected to their respective server node $ i $ at the pull step is less than or equal to the maximum allowed number of stragglers (i.e.. $ | \Gamma_{i}^{\complement} | \leq s $).. And the server decodes the partition's inexact gradient $ \widehat{\nabla{f}^{(i)}} $ by the brute application of the described coding scheme. 

$ \textbf{Scheme \ 2:} $ \\
In this scheme, the number of all partition's worker nodes disconnected (i.e., fail or get delayed) from their respective server node $ i $ at the pull step is greater than the maximum allowed number of stragglers, 
(i.e., $ | \Gamma_{i}^{\complement} | > s $).
Server node $ i $ uses only the received coded local gradients from its connection set $ \Gamma_{i}(k) $ to compute the partition's inexact gradient $  \widehat{\nabla{f}^{(i)}}  $.

$ \textbf{Scheme  \ 3:} $ \\
In this scheme, the number of all worker nodes disconnected (i.e., fail or get delayed) from their respective server node $ i $ at the pull step is greater than the maximum allowed number of stragglers, 
(i.e., $ | \Gamma_{i}^{\complement} | > s $).
Server node $ i $ uses the received local gradients at instant $ k $ from  its connection set $ \Gamma_{i}(k) $ of the connected worker nodes of partition $ \iota $ identified with server $ i $ 
or stale delayed coded local gradients to compute the partition's inexact gradient  $ \widehat{\nabla{f}^{(i)}} $,  (see Remark 3 for different scenarios of this scheme). \\


Hence, in SRDO step (\ref{eq2}b), the iterate ${\bf x}_{i}(k+1)$ employing the partition's decoded gradient is calculated by the server node. i.e.,
{\small \begin{equation}\label{eq5}
\begin{split}
{\bf x}_{i}(k+1)={\bf v}_{i}(k)-\alpha_{k}\widehat{\nabla{f}^{(i)}}(k)
\end{split}
\end{equation}}
where $\widehat{\nabla{f}^{(i)}}(k)=\sum_{j \in \Gamma(k) }{\bf A}^{(i)}_{fit,j}{\nabla{g}^{(i)}_{j}}({\bf v}_{q}(k-k_{q}^{'}))$ where $ 0 \leq k_{q}^{'} \leq H  $ and $ q \in V=\{1,\dots,n\} $.

The server node aggregates the weighted local coded gradients at different estimates ${\bf v}_{q}(k)$ available on each connected partition's worker node, to decode the gradient of the partition. \\

This decoded partition's gradient utilized in the updating step of the partition's connected server in its more general form compatible with the aforementioned three schemes can be further written in a more reduced form corresponding to each index $i \in \{1,\ldots,n\} $ as:
{ \begin{equation}\label{eq5.1}
\begin{split}
\widehat{\nabla{f}^{(i)}} & (k)=  \nabla{f^{(i)}({\bf v}_{i}(k))} \\
& + \sum_{j \in \Gamma_{i}(k)}{\bf A}^{(i)}_{fit,j}\sum_{l=1}^{n_{i}}{\bf B}^{(i)}_{j,l}(\nabla{f}^{(i)}_{l}({\bf v}_{q}(k-k_{q}^{'}))-\nabla{f}^{(i)}_{l}({\bf v}_{i}(k))).
\end{split}
\end{equation}}

Where here the server node $ i $ receives coded gradients in the pull step from a worker node $ l $ which previously received an arbitrary estimate $ {\bf v}_{q} $ in the prior consecutive push step. Here, we seek arbitrary server to worker connections in both the pull and push steps. That is, servers are not only connected to the same workers partition in both the pull and the push steps as then in that case an exact but different partition's gradients are decoded on different servers.

\section{Background Material}

\label{sec:review}

\subsection{Gradient Coding Scheme}

As previously discussed, when solving Problem \eqref{eq1} using a distributed synchronous gradient descent worker nodes may be stragglers \cite{li2014communication,ho2013more,dean2012large}; i.e., the nodes fail or get delayed significantly in computing or communicating the gradient vector to the server node. Tandon et al. \cite{tandon17a} proposed replicating some data across machines in a defined coding scheme to address this problem. This scheme allows the recovery of the overall gradient by the aggregation of the  computed local gradients from the connected nodes active in the network at iteration $k$.


Specifically, in \cite{tandon17a}, the authors find a lower bound on the structure of the coding partition scheme that allows the computation of the overall gradient in the presence of any $s$ or fewer stragglers, (i.e., if we have fewer stragglers than the maximum allowed, we can use any $n-s$ combination of the connected nodes). This bound is on the minimum number of replicas for each partition $J_{i}$ forming the local function $f^{(i)}$ such that $f=\sum_{i=1}^{d}f^{(i)}$ which should be at least replicated $s+1$ times across all machines.
A coding scheme robust to any $s$ stragglers or less, corresponding to $n$ nodes and $m$ data partitions.
Without loss of generality, we assume the number of nodes $ n $ is equal to the number of partitions $ m $ in the algorithm.

Therefore, to employ coding and decoding of the overall gradient in the case of the allowed number of stragglers not to exceed $s$, we require a scheme in which \cite{tandon17a}: 
\begin{equation}
{\bf B} \in \mathbb{R}^{n \times d}, \ \  {\bf A} \in \mathbb{R}^{{{n}\choose{s}} \times n}  \ \ \text{and}  \quad {\bf A}{\bf B}=\boldmath{1}_{{{n}\choose{s}} \times d},
\end{equation}
where the decoding matrix ${\bf A}$ and the encoding matrix ${\bf B}$ can be calculated from Algorithm 1 and  Algorithm 2 in \cite{tandon17a}, respectively.

We exploit the above coded scheme to compute $ \nabla{f}^{(i)} $, where $ \nabla{f}= \sum_{i=1}^{n}\nabla{f}^{(i)} $. Thus, we apply this coding scheme to compute $ \nabla{f}^{(i)} = \sum_{l=1}^{n_{i}}f^{(i)}_{l} $ (i.e., in the coding scheme, the number of data partitions $ d $ is without loss of generality equal to the number of nodes $ n $; therefore, in our partitions $ d = n_{i} $, the number of nodes of partition $ i $). See Remark~\ref{remark2} in Section~\ref{sec:algorithm}.

\begin{figure}
\hspace{-1cm}
\includegraphics[bb=0 0 800 800, scale=0.5]{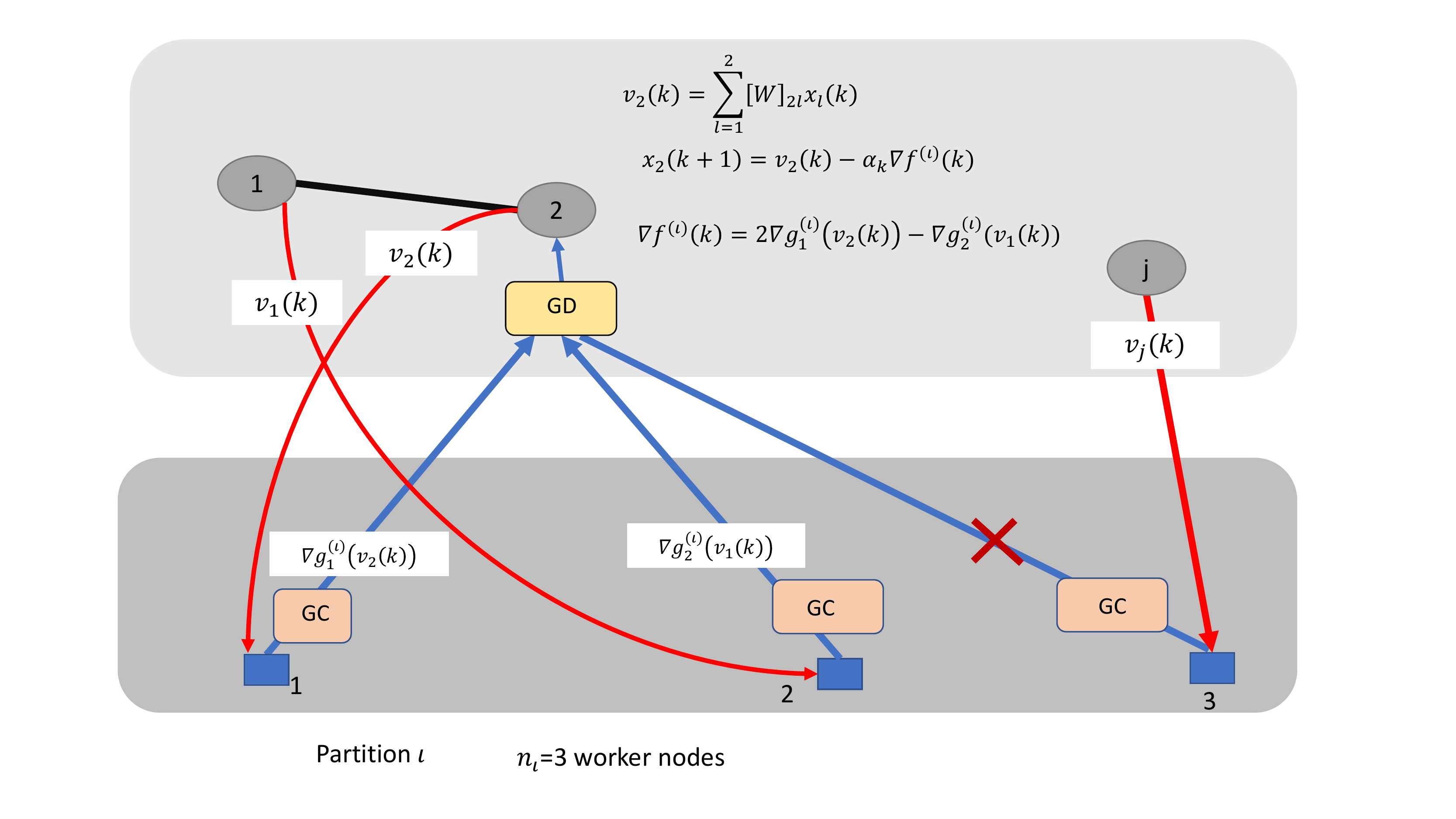}
\caption{Parameter server schematic with coding scheme $ 1 $ of allowed number of stragglers on partition $ 1 $ of $ 3 $ worker nodes connecting to sever $ 2 $. Notice that in this general scenario, the $ \widehat{\nabla}{f}_{2} $ that is calculated at server node $ 2 $ is identified with partition $ 1 $ gradient $ \widehat{\nabla}{f}^{(1)} $.}
\label{Fig.2}
\end{figure}

\section{Theoretical Analysis Background}\label{sec:pre-convergence}

\subsection{Specific Relations and Theorems}

\subsubsection{Convexity of functions}

We refer to \citep{bertsekas2003convex} for the definition of convex functions and the properties concerning nonexpansiveness of projections and the definition of minimum distance between a point and a closed set.

\subsubsection{Supermartingale Convergence Result}
In our analysis as in paper \cite{lee2013distributed}, we also make use of the supermartingale convergence result due to Robbins and Siegmund  (see  Lemma 10-11, p. 49-50 \cite{polyak1987introduction} or original paper \cite{robbins2012selected}, p. 111-135)
\begin{thm}
Let ${v_{k}}$, ${u_{k}}$, ${a_{k}}$ and ${b_{k}}$ be sequences of non-negative random variables such that
\begin{equation*}
\begin{split}
\end{split}
\mathbb{E}[v_{k+1}]|\mathcal{F}_{k}] \leq (1+a_{k})v_{k}-u_{k}+b_{k} \ \text{ for all } k \geq 0 \ a.s.
\end{equation*}
where $\mathcal{F}_{k}$ denotes the collection $v_{0},...,v_{k}$, $u_{0},...,u_{k}$, $a_{0},...,a_{k}$ and $b_{0},...,b_{k}$. Also, $\sum_{k=0}^{\infty}a_{k} < \infty$ and $\sum_{k=0}^{\infty}b_{k} < \infty$ a.s.\\
Then, we have $\lim_{k \to \infty}v_{k}=v$ for a random variable $v \geq 0$ a.s. and $\sum_{k=0}^{\infty}u_{k} < \infty$ a.s. 
\end{thm}


\section{Convergence Analysis for the Main Algorithm SRDO} \label{sec:convergence}

\begin{define}\label{Rk}
Let
{\small \begin{equation}\label{eq6}
\begin{split}
    R_{i}(k) & = -\alpha_{k}\sum_{j \in \Gamma_{i}(k)}{\bf A}^{(i)}_{fit,j}{\bf B}^{(i)}_{j,l}(\nabla{f}^{(i)}_{l}({\bf v}_{q}(k-k_{q}^{'}))-\nabla{f}^{(i)}_{l}({\bf v}_{i}(k))) \\
    & = \alpha_{k} {\bf \epsilon}_{i}(k)
\end{split}    
\end{equation}}

Hence, from \eqref{eq5} and \eqref{eq5.1} we have
\begin{equation}\label{eq7}
\begin{split}
{\bf x}_{i}(k+1)={\bf v}_{i}(k)-\alpha_{k}\nabla{f^{(i)}({\bf v}_{i}(k))}+R_{i}(k)    
\end{split}
\end{equation}

Or 
{ \begin{equation}\label{3.43}
\begin{split}
& \hspace{3cm} {\bf x}_{i}(k+1)=\overline{{\bf x}_{i}}(k+1) + R_{i}(k), \\
& \text{where}  \ \ \overline{{\bf x}_{i}}(k+1)={\bf v}_{i}(k)-\alpha_{k}\nabla{f^{(i)}({\bf v}_{i}(k))}
\end{split}
\end{equation}}

\end{define}


\begin{lem}\label{general_error_bound}
Let Assumptions 1, 2, 4 and 5 hold. Let the sequences $ \{{\bf x}_{i}(k)\} $ and $ \{ {\bf v}_{i}(k) \} $ be generated by method \eqref{eq2}. Then we have {\small \begin{equation}\label{ineq_general_error_bound}
\begin{split}
    & \sum_{l=1}^{n}\|  {\bf v}_{l}(k+1)-  {\bf x}^{*}\|^{2}  \leq \sum_{j=1}^{n}\| {\bf v}_{j}(k)- {\bf x}^{*}\|^{2}  \\
    & + 2\alpha_{k}\sum_{j=1}^{n}\sum_{(i)=1}^{p}\gamma_{(i)}\langle \nabla{f}^{(i)}( {\bf v}_{j}(k)), {\bf x}^{*}- {\bf v}_{j}(k)\rangle \\
    & + 2 \alpha_{k}\sum_{j=1}^{n} \sum_{(i)=1}^{p}\gamma_{(i)} \langle {\bf \epsilon}_{j,(i)}(k), {\bf v}_{j}(k)- {\bf x}^{*}\rangle \\
    & + 2 \alpha_{k}^{2} \frac{\sum_{j=1}^{n}\sum_{(i)=1}^{p}\gamma_{(i)}}{(\sum_{(i)=1}^{p}\gamma_{(i)})^{2}}[\|\nabla{f}^{(i)}( {\bf v}_{j}(k))\|^{2} + \| {\bf \epsilon}_{j,(i)}(k)\|^{2}]
\end{split}
\end{equation}}
\end{lem}

\begin{lem}\label{type1_error_bound}
Let Assumptions 1, 2, 4 and 5 hold. Let the functions in Assumption 1 also satisfy  $ f^{(i)}({\bf x}^{*}) = f^{(i)}({\bf x}^{(i)}) $ for all $ (i) $. Let the sequences $ \{{\bf x}_{i}(k)\} $ and $ \{ {\bf v}_{i}(k) \} $, $ i \in V $ be generated by method \eqref{eq2}. Then we have 
{ \begin{equation}\label{ineq_type1_error_bound}
\begin{split}
    & ( 1 -\mu ) \sum_{l=1}^{n}\|  {\bf v}_{l}(k+1) - {\bf x}^{*}\|^{2}  \leq (1-\mu) \sum_{j=1}^{n}\|{\bf v}_{j}(k)- {\bf x}^{*}\|^{2}  \\
   & + 2 \alpha_{k}\sum_{j=1}^{n}\sum_{(i)=1}^{p}\gamma_{(i)} \langle {\bf \epsilon}_{j,(i)}(k), {\bf v}_{j} - {\bf x}^{*} \rangle  + 2 \alpha_{k}^{2} \frac{\sum_{j=1}^{n}\sum_{(i)=1}^{p}\gamma_{(i)}}{(\sum_{(i)=1}^{p}\gamma_{(i)})^{2}}\|{\bf \epsilon}_{j,(i)}(k)\|^{2}\\
    & - 2 a \alpha_{k} \sum_{j=1}^{n}\sum_{(i)=1}^{p}\gamma_{(i)} (b - \frac{2 \alpha_{k} a}{(\sum_{(i)=1}^{p}\gamma_{(i)})^{2}} ) \| {\bf v}_{j}(k) - {\bf x}^{(i)} \| ^{2}  
\end{split}
\end{equation}}
\end{lem}

\begin{prop}\label{Convergence_type1}
Let Assumptions 1-5 hold. Let the functions in Assumption 1 be strongly convex and satisfy  $ f^{(i)}( {\bf x}^{*}) = f^{(i)}( {\bf x}^{(i)}) $ for all $ (i) $. Let the sequences $ \{ {\bf x}_{i}(k) \} $ and $ \{ {\bf v}_{i}(k) \} $, $ i \in V  $ be generated by method \eqref{eq2} with stepsizes and errors as given in Assumptions 3 and 5. Assume that problem \eqref{eq1} has a non-empty optimal solution set $ \mathcal{X}^{*} $ as given in Assumption 1. Then, the sequences $ \{ {\bf x}_{i}(k) \} $ and $ \{ {\bf v}_{i}(k) \} $, $ i \in V  $ converge to the same random point in $ \mathcal{X}^{*} $ with probability 1.
\end{prop}

\begin{proof}
With errors as in Assumption 3 we have $ \|R_{i}(k)\|$ and $\|\epsilon_{j,(i)}(k)\|$ as given in Appendix.
Then having $ f^{(i)}( {\bf x}^{*}) = f^{(i)}( {\bf x}^{(i)}) $ for all $ (i) $ also satisfied we have Lemma~\ref{type1_error_bound} satisfied. Then we can use  the resulting inequality \eqref{ineq_type1_error_bound} with the substitution of $ \| {\bf \epsilon}_{j,(i)}(k) \| $ from Appendix to get 
{ \begin{equation}\label{ineq_type1_error_bound_subst}
\begin{split}
    & \sum_{l=1}^{n}\|  {\bf v}_{l}(k+1) -  {\bf x}^{*}\|^{2}  \leq (1-\mu)\sum_{j=1}^{n}\| {\bf v}_{j}(k) - {\bf x}^{*}\|^{2}  \\
   & + 4 L \alpha_{k} \sum_{j=1}^{n} \max_{(i)}
\| {\bf A}^{(i)} \| _{\infty} \|{\bf B}^{(i)}\| _{2,\infty}  \max_{k-H \leq \hat{k} \leq k ; q \in V} \| {\bf v}_{q}(\hat{k}) -  {\bf x}^{*}\|^{2}   \\
& + 8 L^{2} \alpha_{k}^{2} \frac{\sum_{j=1}^{n} \max_{(i)} \| {\bf A}^{(i)} \| ^{2} _{\infty} \|{\bf B}^{(i)}\|^{2} _{2,\infty} \max_{k-H \leq \hat{k} \leq k ; q \in V } \| {\bf v}_{q}(\hat{k}) -  {\bf x}^{*}\|^{2} }{(\sum_{(i)=1}^{p}\gamma_{(i)})^{2}} \\
& - 2 a \alpha_{k} \sum_{j=1}^{n}\sum_{(i)=1}^{p}\gamma_{(i)} (b - \frac{2 \alpha_{k} a}{(\sum_{(i)=1}^{p}\gamma_{(i)})^{2}} ) \| {\bf v}_{j}(k) - {\bf x}^{(i)} \| ^{2}   
\end{split}
\end{equation}}
But in order to be able to use Lemma~\ref{Lemma6a} the last term in \eqref{ineq_type1_error_bound_subst} should be negative.Which means $ b - \frac{2 \alpha_{k} a}{(\sum_{(i)=1}^{p}\gamma_{(i)})^{2}}  \geq 0 $ where $ b = \langle \overrightarrow{u},\overrightarrow{v} \rangle $. And $ \nabla{f}^{(i)}({\bf v}_{j}(k)) - \nabla{f}^{(i)}({\bf x}^{(i)}) = a \| {\bf v}_{j}(k) - {\bf x}^{(i)} \| \overrightarrow{u} $ where $ a \leq L $.
However, $ f^{(i)} $ is strongly convex for every $ (i) $, then $ \langle \nabla{f}^{(i)}({\bf v}_{j}(k)) - \nabla{f}^{(i)}({\bf x}^{(i)}), {\bf v}_{j}(k) - {\bf x}^{(i)} \rangle  = a \| {\bf v}_{j}(k) - {\bf x}^{(i)} \| ^{2} \langle \overrightarrow{u},\overrightarrow{v} \rangle  \geq \sigma_{(i)} ^{2} \| {\bf v}_{j}(k) - {\bf x}^{(i)} \| ^{2} $, that is $ \langle \overrightarrow{u},\overrightarrow{v} \rangle \geq \frac{\sigma_{(i)}}{a} $. 
Therefore, a sufficient condition is 
\begin{equation}\label{suffcond1}
\begin{split}
\frac{2 \alpha_{k} a}{(\sum_{(i)=1}^{p}\gamma_{(i)})^{2}}  \leq  \frac{2 \alpha_{k} L}{(\sum_{(i)=1}^{p}\gamma_{(i)})^{2}}  \leq \frac{\sigma_{(i)}}{a} = \langle \overrightarrow{u},\overrightarrow{v} \rangle 
\end{split}
\end{equation}
The sufficient condition in \eqref{suffcond1} is satisfied for $ k \geq k_{0} $ since $ \alpha_{k} \rightarrow 0 $.
Then \eqref{ineq_type1_error_bound_subst} is similar to the martingale inequality \eqref{eqw} of Lemma~\ref{Lemma6a} for $ k \geq k_{0} $. By the result of Lemma~\ref{Lemma6a} we have for $ v_{k}= \sum_{i=1}^{n} \| {\bf v}_{i}(k) - {\bf x}^{*} \|^{2} $ that
\begin{equation}
\begin{split}
    \sum_{i=1}^{n} \| {\bf v}_{i}(k) - {\bf x}^{*} \|^{2} \leq \rho ^{k}  V_{0} 
\end{split}    
\end{equation}
for $ k \geq \bar{k}_{1} =\bar{k}= k_{0}+B $ where $ \rho $, $ V_{0} $ and $ \bar{k} $ are as in the lemma. Therefore, as $ k \rightarrow \infty $ we have $ \sum_{i=1}^{n} \| {\bf v}_{i}(k) - {\bf x}^{*} \|^{2} \rightarrow 0 $. That is, $  \| {\bf v}_{i}(k) - {\bf x}^{*} \| \rightarrow 0 $ for all $ i \in V $. 
Then in view of (\ref{eq2}b) where $ {\bf x}_{i}(k+1) = {\bf v}_{i}(k) -\alpha_{k}\nabla{f}^{(i)}({\bf v}_{i}(k)) +  R_{i}(k) $ and since $ R_{i}(k) \rightarrow 0 $ because  $  \| {\bf v}_{i}(k) - {\bf x}^{*} \| \rightarrow 0 $  or $ \alpha_{k} \rightarrow 0 $ and $ \alpha_{k} \nabla{f}^{(i)}({\bf v}_{i}(k))\rightarrow 0 $ since $ \alpha_{k} \rightarrow 0 $ and $ {\bf v}_{i}(k) \rightarrow {\bf x}^{*} $, thus $ {\bf v}_{i}(k) \in \mathcal{X} $ where $ \nabla{f}^{(i)}({\bf v}_{i}(k)) \leq G_{f} $. Therefore, $ {\bf x}_{i}(k+1) = {\bf v}_{i}(k) = {\bf x}^{*} $ as $ k \rightarrow \infty $.

\end{proof}

\begin{prop}\label{Convergence_type1a}
Let Assumptions 1-5 hold. Let the functions in Assumption 1  satisfy  $ f^{(i)}( {\bf x}^{*}) = f^{(i)}( {\bf x}^{(i)}) $ for all $ (i) $. Let the sequences $ \{ {\bf x}_{i}(k) \} $ and $ \{ {\bf v}_{i}(k) \} $, $ i \in V  $ be generated by method \eqref{eq2} with stepsizes and errors as given in Assumptions 3 and 5. Assume that problem \eqref{eq1} has a non-empty optimal solution set $ \mathcal{X}^{*} $ as given in Assumption 1. Then, the sequences $ \{ {\bf x}_{i}(k) \} $ and $ \{ {\bf v}_{i}(k) \} $, $ i \in V  $ converge to the same random point in $ \mathcal{X}^{*} $ with probability 1.
\end{prop}

\begin{proof}
With errors as in Assumption 3 we have $ \|R_{i}(k)\|$ and $\|\epsilon_{j,(i)}(k)\|$ as given in Appendix.
Then having $ f^{(i)}( {\bf x}^{*}) = f^{(i)}( {\bf x}^{(i)}) $ for all $ (i) $ also satisfied we have Lemma~\ref{type1_error_bound} satisfied. Then we can use  the resulting inequality \eqref{ineq_type1_error_bound} with the substitution of $ \| {\bf \epsilon}_{j,(i)}(k) \| $ from Appendix to get 
{ \begin{equation}\label{ineq_type1_error_bound_substa}
\begin{split}
    & \sum_{l=1}^{n}\|  {\bf v}_{l}(k+1) -  {\bf x}^{*}\|^{2}  \leq (1-\mu)\sum_{j=1}^{n}\| {\bf v}_{j}(k) - {\bf x}^{*}\|^{2}  \\
   & + 4 L \alpha_{k} \sum_{j=1}^{n} \max_{(i)}
\| {\bf A}^{(i)} \| _{\infty} \|{\bf B}^{(i)}\| _{2,\infty}  \max_{k-H \leq \hat{k} \leq k ; q \in V} \| {\bf v}_{q}(\hat{k}) -  {\bf x}^{*}\|^{2}   \\
& + 8 L^{2} \alpha_{k}^{2} \frac{\sum_{j=1}^{n} \max_{(i)} \| {\bf A}^{(i)} \| ^{2} _{\infty} \|{\bf B}^{(i)}\|^{2} _{2,\infty} \max_{k-H \leq \hat{k} \leq k ; q \in V } \| {\bf v}_{q}(\hat{k}) -  {\bf x}^{*}\|^{2} }{(\sum_{(i)=1}^{p}\gamma_{(i)})^{2}} \\
& - 2 a \alpha_{k} \sum_{j=1}^{n}\sum_{(i)=1}^{p}\gamma_{(i)} (b - \frac{2 \alpha_{k} a}{(\sum_{(i)=1}^{p}\gamma_{(i)})^{2}} ) \| {\bf v}_{j}(k) - {\bf x}^{(i)} \| ^{2}   
\end{split}
\end{equation}}
But in order to be able to use Lemma~\ref{Lemma6a} on \eqref{ineq_type1_error_bound_substa} we must have
{ \begin{equation}\label{cond1a}
\begin{split}
  & - \mu \sum_{j=1}^{n}\| {\bf v}_{j}(k) - {\bf x}^{*}\|^{2}   - 2 a \alpha_{k} \sum_{j=1}^{n}\sum_{(i)=1}^{p}\gamma_{(i)} (b - \\
   & \ \ \ \ \ \frac{2 \alpha_{k} a}{(\sum_{(i)=1}^{p}\gamma_{(i)})^{2}} ) \| {\bf v}_{j}(k) - {\bf x}^{(i)} \| ^{2} \leq 0  
\end{split}
\end{equation}}
That is
{ \begin{equation}
\begin{split}
   & -\mu \sum_{j=1}^{n}\| {\bf v}_{j}(k) - {\bf x}^{*}\|^{2}   +  \sum_{j=1}^{n}\sum_{(i)=1}^{p}\gamma_{(i)} \frac{4 \alpha_{k}^{2} a^{2}}{(\sum_{(i)=1}^{p}\gamma_{(i)})^{2}} ) \| {\bf v}_{j}(k) - {\bf x}^{(i)} \| ^{2} \\
   & \leq 2 a \alpha_{k} \sum_{j=1}^{n}\sum_{(i)=1}^{p}\gamma_{(i)} b \| {\bf v}_{j}(k) - {\bf x}^{(i)} \| ^{2}
\end{split}
\end{equation}}
which reduces to the sufficient condition
{ \begin{equation}\label{suffcond1a}
\begin{split}
  & -\mu \sum_{j=1}^{n}\| {\bf v}_{j}(k) - {\bf x}^{*}\|^{2}   +  \sum_{j=1}^{n}\sum_{(i)=1}^{p}\gamma_{(i)} \frac{8 \alpha_{k}^{2} a^{2}}{(\sum_{(i)=1}^{p}\gamma_{(i)})^{2}} ) \| {\bf v}_{j}(k) - {\bf x}^{*}\|^{2} \\
  & + \sum_{j=1}^{n}\sum_{(i)=1}^{p}\gamma_{(i)} \frac{8 \alpha_{k}^{2} a^{2}}{(\sum_{(i)=1}^{p}\gamma_{(i)})^{2}} ) \| {\bf x}^{*} - {\bf x}^{(i)} \| ^{2} \\
  & \leq 2 a \alpha_{k} \sum_{j=1}^{n}\sum_{(i)=1}^{p}\gamma_{(i)} b \| {\bf v}_{j}(k) - {\bf x}^{(i)} \| ^{2}
\end{split}
\end{equation}}
But $ 0 \geq b = \langle \overrightarrow{u},\overrightarrow{v} \rangle \leq 1 $. And $ \nabla{f}^{(i)}({\bf v}_{j}(k)) - \nabla{f}^{(i)}({\bf x}^{(i)}) = a \| {\bf v}_{j}(k) - {\bf x}^{(i)} \| \overrightarrow{u} $ where $ a \leq L $.
Thus, the right hand side of \eqref{suffcond1a} is nonnegative.
Therefore, the sufficient condition in \eqref{suffcond1a} reduces to
{ \begin{equation}\label{suffcond1aa}
\begin{split}
& + \sum_{j=1}^{n}\sum_{(i)=1}^{p}\gamma_{(i)} \frac{8 \alpha_{k}^{2} a^{2}}{(\sum_{(i)=1}^{p}\gamma_{(i)})^{2}} ) \| {\bf x}^{*} - {\bf x}^{(i)} \| ^{2} \\
& \leq ( \mu - \sum_{(i)=1}^{p}\gamma_{(i)} \frac{8 \alpha_{k}^{2}a^{2}}{(\sum_{(i)=1}^{p}\gamma_{(i)})^{2}} \sum_{j=1}^{n} \| {\bf v}_{j}(k) - {\bf x}^{*}\|^{2}
\end{split}
\end{equation}}
But the sufficient condition in \eqref{suffcond1aa} is satisfied for $ 0 < \mu < 1 $  and $ 0 < a <  L $ since $ \alpha_{k} \rightarrow 0 $. Thus, 
{ \begin{equation}\label{suffcond1aaa}
\begin{split}
& + \sum_{j=1}^{n}\sum_{(i)=1}^{p}\gamma_{(i)} \frac{8 \alpha_{k}^{2} a^{2}}{(\sum_{(i)=1}^{p}\gamma_{(i)})^{2}} ) \| {\bf x}^{*} - {\bf x}^{(i)} \| ^{2} \\
& + \sum_{j=1}^{n}\sum_{(i)=1}^{p}\gamma_{(i)} \frac{8 \alpha_{k}^{2} L^{2}}{(\sum_{(i)=1}^{p}\gamma_{(i)})^{2}} ) \| {\bf x}^{*} - {\bf x}^{(i)} \| ^{2} \\
& \leq ( \mu - \sum_{(i)=1}^{p}\gamma_{(i)} \frac{8 \alpha_{k}^{2}L^{2}}{(\sum_{(i)=1}^{p}\gamma_{(i)})^{2}} \sum_{j=1}^{n} \| {\bf v}_{j}(k) - {\bf x}^{*}\|^{2} \\
& \leq ( \mu - \sum_{(i)=1}^{p}\gamma_{(i)} \frac{8 \alpha_{k}^{2}a^{2}}{(\sum_{(i)=1}^{p}\gamma_{(i)})^{2}} \sum_{j=1}^{n} \| {\bf v}_{j}(k) - {\bf x}^{*}\|^{2}
\end{split}
\end{equation}}
Therefore, for $ k \geq k_{0} $ we have the right hand side of \eqref{suffcond1aaa} to be greater than $ E $ and the left hand side to be less than $ E $ since $ \alpha_{k} \rightarrow 0 $. Then \eqref{ineq_type1_error_bound_substa} is similar to the martingale inequality \eqref{eqw} of Lemma~\ref{Lemma6a} for $ k \geq k_{0} $. By the result of Lemma~\ref{Lemma6a} we have for $ v_{k}= \sum_{i=1}^{n} \| {\bf v}_{i}(k) - {\bf x}^{*} \|^{2} $ that
\begin{equation}
\begin{split}
    \sum_{i=1}^{n} \| {\bf v}_{i}(k) - {\bf x}^{*} \|^{2} \leq \rho ^{k}  V_{0} 
\end{split}    
\end{equation}
for $ k \geq \bar{k}_{1} =\bar{k}= k_{0}+B $ where $ \rho $, $ V_{0} $ and $ \bar{k} $ are as in the lemma. Therefore, as $ k \rightarrow \infty $ we have $ \sum_{i=1}^{n} \| {\bf v}_{i}(k) - {\bf x}^{*} \|^{2} \rightarrow 0 $. That is, $  \| {\bf v}_{i}(k) - {\bf x}^{*} \| \rightarrow 0 $ for all $ i \in V $. 
Then in view of (\ref{eq2}b) where $ {\bf x}_{i}(k+1) = {\bf v}_{i}(k) -\alpha_{k}\nabla{f}^{(i)}({\bf v}_{i}(k)) +  R_{i}(k) $ and since $ R_{i}(k) \rightarrow 0 $ because  $  \| {\bf v}_{i}(k) - {\bf x}^{*} \| \rightarrow 0 $  or $ \alpha_{k} \rightarrow 0 $ and $ \alpha_{k} \nabla{f}^{(i)}({\bf v}_{i}(k))\rightarrow 0 $ since $ \alpha_{k} \rightarrow 0 $ and $ {\bf v}_{i}(k) \rightarrow {\bf x}^{*} $, thus $ {\bf v}_{i}(k) \in \mathcal{X} $ where $ \nabla{f}^{(i)}({\bf v}_{i}(k)) \leq G_{f} $. Therefore, $ {\bf x}_{i}(k+1) = {\bf v}_{i}(k) = {\bf x}^{*} $ as $ k \rightarrow \infty $.

\end{proof}

\begin{lem}\label{type2_error_bound}
Let Assumptions 1, 2, 4 and 5 hold. Let the functions in Assumption 1 also satisfy  $ f^{(i)}({\bf x}^{*}) > f^{(i)}(x^{(i)}) $ for at least one $ (i) $. Let the sequences $ \{{\bf x}_{i}(k)\} $ and $ \{ {\bf v}_{i}(k) \} $, $ i \in V $ be generated by method \eqref{eq2}.
Then we have
{ \begin{equation}\label{ineq_type2_error_bound}
\begin{split}
    & ( 1- \mu) \sum_{l=1}^{n}\|  {\bf v}_{l}(k+1) - {\bf x}^{*}\|^{2}  \leq (1-\mu) \sum_{j=1}^{n}\|{\bf v}_{j}(k)- {\bf x}^{*}\|^{2}  \\
   & + 2 \alpha_{k}\sum_{j=1}^{n}\sum_{(i)=1}^{p}\gamma_{(i)} \langle {\bf \epsilon}_{j,(i)}(k), {\bf v}_{j} - {\bf x}^{*} \rangle  + 2 \alpha_{k}^{2} \frac{\sum_{j=1}^{n}\sum_{(i)=1}^{p}\gamma_{(i)}}{(\sum_{(i)=1}^{p}\gamma_{(i)})^{2}}\|{\bf \epsilon}_{j,(i)}(k)\|^{2}\\
    & + 2 \alpha_{k} \sum_{j=1}^{n}\sum_{(i)\in I }\gamma_{(i)} \langle \nabla{f}^{(i)}({\bf v}_{j}(k)), {\bf x}^{*} - {\bf x}^{(i)} \rangle  \\
      & - 2 a \alpha_{k} \sum_{j=1}^{n}\sum_{(i)=1}^{p}\gamma_{(i)} (b - \frac{2 \alpha_{k} a}{(\sum_{(i)=1}^{p}\gamma_{(i)})^{2}} ) \| {\bf v}_{j}(k) - {\bf x}^{(i)} \| ^{2}   
\end{split}
\end{equation}}
\end{lem}

\begin{prop}\label{sum_v_i-x^*_type2}
Let Assumptions 1-5 hold. Let the functions in Assumption 1 be strongly convex and satisfy  $ f^{(i)}({\bf x}^{*}) > f^{(i)}( {\bf x}^{(i)}) $ for at least one $ (i) $. Let the sequences $ \{ {\bf x}_{i}(k) \} $ and $ \{ {\bf v}_{i}(k) \} $, $ i \in V  $ be generated by method \eqref{eq2} with stepsizes and errors as given in Assumptions 3 and 5. Assume that problem \eqref{eq1} has a non-empty optimal solution set $ \mathcal{X}^{*} $ as given in Assumption 1. Then, the sequence $ \{ \sum_{i=1}^{n} \| {\bf v}_{i}(k) - {\bf x}^{*} \|^{2} \} $ converge to a nonnegative value $ D $.
\end{prop}

\begin{proof}
With errors as in Assumption 3 we have $ \|R_{i}(k)\|$ and $\|\epsilon_{i}(k)\|$ as given in Appendix.
Then having $ f^{(i)}({\bf x}^{*}) > f^{(i)}({\bf x}^{(i)}) $ for at least one $ (i) $ also satisfied we have Lemma~\ref{type2_error_bound} satisfied. Then we can use  the resulting inequality \eqref{ineq_type2_error_bound} with the substitution of $ \| {\bf \epsilon}_{j,(i)}(k) \| $ to get 
{ \begin{equation}\label{ineq_type2_error_bound_subst}
\begin{split}
& \sum_{l=1}^{n}\|  {\bf v}_{l}(k+1)-  {\bf x}^{*}\|^{2}  \leq (1-\mu)\sum_{j=1}^{n}\| {\bf v}_{j}(k)- {\bf x}^{*}\|^{2}  \\
& + 2\alpha_{k}n |I| L \max_{(i)} \| {\bf x}^{*}- {\bf x}^{(i)} \| ^{2} \\
& + 4 L \alpha_{k}\sum_{j=1}^{n}\| {\bf A}^{(i)} \| _{\infty} \|{\bf B}^{(i)}\| _{2,\infty}  \max_{k-H \leq \hat{k} \leq k ; q \in V} \| {\bf v}_{q}(\hat{k}) -  {\bf x}^{*}\|^{2} \\
& + 8 L^{2} \alpha_{k}^{2} \frac{\sum_{j=1}^{n}\| {\bf A}^{(i)} \| ^{2} _{\infty} \|{\bf B}^{(i)}\| ^{2} _{2,\infty} \max_{k-H \leq \hat{k} \leq k ; q \in V} \| {\bf v}_{q}(\hat{k}) -  {\bf x}^{*}\|^{2} }{(\sum_{(i)=1}^{p}\gamma_{(i)})^{2}} \\
& - 2 a \alpha_{k} \sum_{j=1}^{n}\sum_{(i)=1}^{p}\gamma_{(i)} (b - \frac{2 \alpha_{k} a}{(\sum_{(i)=1}^{p}\gamma_{(i)})^{2}} ) \| {\bf v}_{j}(k) - {\bf x}^{(i)} \| ^{2}   
\end{split}
\end{equation}}
But in order to be able to use Lemma~\ref{Lemma6a1} the last term in \eqref{ineq_type2_error_bound_subst} should be negative. Which means $ b - \frac{2 \alpha_{k} a}{(\sum_{(i)=1}^{p}\gamma_{(i)})^{2}}  \geq 0 $ where $ b = \langle \overrightarrow{u},\overrightarrow{v} \rangle $. And $ \nabla{f}^{(i)}({\bf v}_{j}(k)) - \nabla{f}^{(i)}({\bf x}^{(i)}) = a \| {\bf v}_{j}(k) - {\bf x}^{(i)} \| \overrightarrow{u} $ where $ a \leq L $.
However, $ f^{(i)} $ is strongly convex for every $ (i) $, then $ \langle \nabla{f}^{(i)}({\bf v}_{j}(k)) - \nabla{f}^{(i)}({\bf x}^{(i)}), {\bf v}_{j}(k) - {\bf x}^{(i)} \rangle  = a \| {\bf v}_{j}(k) - {\bf x}^{(i)} \| ^{2} \langle \overrightarrow{u},\overrightarrow{v} \rangle  \geq \sigma_{(i)} ^{2} \| {\bf v}_{j}(k) - {\bf x}^{(i)} \| ^{2} $, that is $ \langle \overrightarrow{u},\overrightarrow{v} \rangle \geq \frac{\sigma_{(i)}}{a} $. 
Therefore, a sufficient condition is 
\begin{equation}\label{suffcond2}
\begin{split}
\frac{2 \alpha_{k} a}{(\sum_{(i)=1}^{p}\gamma_{(i)})^{2}}  \leq  \frac{2 \alpha_{k} L}{(\sum_{(i)=1}^{p}\gamma_{(i)})^{2}}  \leq \frac{\sigma_{(i)}}{a} = \langle \overrightarrow{u},\overrightarrow{v} \rangle 
\end{split}
\end{equation}
The sufficient condition in \eqref{suffcond2} is satisfied for $ k \geq k ^{'}_{0} $ since $ \alpha_{k} \rightarrow 0 $.
Since $ \max_{(i)} \| {\bf x}^{*}- {\bf x}^{(i)} \| ^{2} $ is fixed independent of $ k $ then \eqref{ineq_type2_error_bound_subst} is similar to the martingale inequality \eqref{eqw1} of Lemma~\ref{Lemma6a1}. By the result of Lemma~\ref{Lemma6a1} we have for $ v_{k}= \sum_{i=1}^{n} \| {\bf v}_{i}(k) - {\bf x}^{*} \|^{2} $ that
\begin{equation}
\begin{split}
    \sum_{i=1}^{n} \| {\bf v}_{i}(k) - {\bf x}^{*} \|^{2} \leq \rho ^{k}  V_{0} + \eta
\end{split}    
\end{equation}
for $ k \geq \bar{k}_{2} =\bar{k} = k^{'}_{0} + B $ where $ \rho $, $ V_{0} $, $ \bar{k} $ are as in the lemma and where $ \eta > 0 $ as substituted from the inequality \eqref{ineq_type2_error_bound_subst} by using the lemma. Therefore, as $ k \rightarrow \infty $ we have $ \sum_{i=1}^{n} \| {\bf v}_{i}(k) - {\bf x}^{*} \|^{2} \rightarrow \eta $. That is, $  \sum_{i=1}^{n} \| {\bf v}_{i}(k) - {\bf x}^{*} \|^{2} < D  $ for all $ i \in V $ as $ k \geq \bar{k}_{2} = \bar{k} $.

\end{proof}

\begin{prop}\label{sum_v_i-x^*_type2a}
Let Assumptions 1-5 hold. Let the functions in Assumption 1 satisfy  $ f^{(i)}({\bf x}^{*}) > f^{(i)}( {\bf x}^{(i)}) $ for at least one $ (i) $. Let the sequences $ \{ {\bf x}_{i}(k) \} $ and $ \{ {\bf v}_{i}(k) \} $, $ i \in V  $ be generated by method \eqref{eq2} with stepsizes and errors as given in Assumptions 3 and 5. Assume that problem \eqref{eq1} has a non-empty optimal solution set $ \mathcal{X}^{*} $ as given in Assumption 1. Then, the sequence $ \{ \sum_{i=1}^{n} \| {\bf v}_{i}(k) - {\bf x}^{*} \|^{2} \} $ converge to a nonnegative value $ D $.
\end{prop}

\begin{proof}
With errors as in Assumption 3 we have $ \|R_{i}(k)\|$ and $\|\epsilon_{i}(k)\|$ as given in Appendix.
Then having $ f^{(i)}({\bf x}^{*}) > f^{(i)}({\bf x}^{(i)}) $ for at least one $ (i) $ also satisfied we have Lemma~\ref{type2_error_bound} satisfied. Then we can use  the resulting inequality \eqref{ineq_type2_error_bound} with the substitution of $ \| {\bf \epsilon}_{j,(i)}(k) \| $ to get 
{ \begin{equation}\label{ineq_type2_error_bound_substa}
\begin{split}
& \sum_{l=1}^{n}\|  {\bf v}_{l}(k+1)-  {\bf x}^{*}\|^{2}  \leq (1-\mu)\sum_{j=1}^{n}\| {\bf v}_{j}(k)- {\bf x}^{*}\|^{2}  \\
& + 2\alpha_{k}n |I| L \max_{(i)} \| {\bf x}^{*}- {\bf x}^{(i)} \| ^{2} \\
& + 4 L \alpha_{k}\sum_{j=1}^{n}\| {\bf A}^{(i)} \| _{\infty} \|{\bf B}^{(i)}\| _{2,\infty}  \max_{k-H \leq \hat{k} \leq k ; q \in V} \| {\bf v}_{q}(\hat{k}) -  {\bf x}^{*}\|^{2} \\
& + 8 L^{2} \alpha_{k}^{2} \frac{\sum_{j=1}^{n}\| {\bf A}^{(i)} \| ^{2} _{\infty} \|{\bf B}^{(i)}\| ^{2} _{2,\infty} \max_{k-H \leq \hat{k} \leq k ; q \in V} \| {\bf v}_{q}(\hat{k}) -  {\bf x}^{*}\|^{2} }{(\sum_{(i)=1}^{p}\gamma_{(i)})^{2}} \\
& - 2 a \alpha_{k} \sum_{j=1}^{n}\sum_{(i)=1}^{p}\gamma_{(i)} (b - \frac{2 \alpha_{k} a}{(\sum_{(i)=1}^{p}\gamma_{(i)})^{2}} ) \| {\bf v}_{j}(k) - {\bf x}^{(i)} \| ^{2}   
\end{split}
\end{equation}}
But in order to be able to use Lemma~\ref{Lemma6a1} on \eqref{ineq_type1_error_bound_substa} we must have
{ \begin{equation}\label{cond2a}
\begin{split}
   -\mu \sum_{j=1}^{n}\| {\bf v}_{j}(k) - {\bf x}^{*}\|^{2}   - 2 a \alpha_{k} \sum_{j=1}^{n}\sum_{(i)=1}^{p}\gamma_{(i)} (b - \frac{2 \alpha_{k} a}{(\sum_{(i)=1}^{p}\gamma_{(i)})^{2}} ) \| {\bf v}_{j}(k) - {\bf x}^{(i)} \| ^{2} \leq 0  
\end{split}
\end{equation}}
That is
{ \begin{equation}
\begin{split}
   & -\mu \sum_{j=1}^{n}\| {\bf v}_{j}(k) - {\bf x}^{*}\|^{2}   +  \sum_{j=1}^{n}\sum_{(i)=1}^{p}\gamma_{(i)} \frac{4 \alpha_{k}^{2} a^{2}}{(\sum_{(i)=1}^{p}\gamma_{(i)})^{2}} ) \| {\bf v}_{j}(k) - {\bf x}^{(i)} \| ^{2} \\
   & \leq 2 a \alpha_{k} \sum_{j=1}^{n}\sum_{(i)=1}^{p}\gamma_{(i)} b \| {\bf v}_{j}(k) - {\bf x}^{(i)} \| ^{2}
\end{split}
\end{equation}}
which reduces to the sufficient condition
{ \begin{equation}\label{suffcond2a}
\begin{split}
  & -\mu \sum_{j=1}^{n}\| {\bf v}_{j}(k) - {\bf x}^{*}\|^{2}   +  \sum_{j=1}^{n}\sum_{(i)=1}^{p}\gamma_{(i)} \frac{8 \alpha_{k}^{2} a^{2}}{(\sum_{(i)=1}^{p}\gamma_{(i)})^{2}} ) \| {\bf v}_{j}(k) - {\bf x}^{*}\|^{2} \\
  & + \sum_{j=1}^{n}\sum_{(i)=1}^{p}\gamma_{(i)} \frac{8 \alpha_{k}^{2} a^{2}}{(\sum_{(i)=1}^{p}\gamma_{(i)})^{2}} ) \| {\bf x}^{*} - {\bf x}^{(i)} \| ^{2} \\
  & \leq 2 a \alpha_{k} \sum_{j=1}^{n}\sum_{(i)=1}^{p}\gamma_{(i)} b \| {\bf v}_{j}(k) - {\bf x}^{(i)} \| ^{2}
\end{split}
\end{equation}}
But $ 0 \geq b = \langle \overrightarrow{u},\overrightarrow{v} \rangle \leq 1 $. And $ \nabla{f}^{(i)}({\bf v}_{j}(k)) - \nabla{f}^{(i)}({\bf x}^{(i)}) = a \| {\bf v}_{j}(k) - {\bf x}^{(i)} \| \overrightarrow{u} $ where $ a \leq L $.
Thus, the right hand side of \eqref{suffcond2a} is nonnegative.
Therefore, the sufficient condition in \eqref{suffcond2a} reduces to
{ \begin{equation}\label{suffcond2aa}
\begin{split}
& + \sum_{j=1}^{n}\sum_{(i)=1}^{p}\gamma_{(i)} \frac{8 \alpha_{k}^{2} a^{2}}{(\sum_{(i)=1}^{p}\gamma_{(i)})^{2}} ) \| {\bf x}^{*} - {\bf x}^{(i)} \| ^{2} \\
& \leq ( \mu - \sum_{(i)=1}^{p}\gamma_{(i)} \frac{8 \alpha_{k}^{2}a^{2}}{(\sum_{(i)=1}^{p}\gamma_{(i)})^{2}} \sum_{j=1}^{n} \| {\bf v}_{j}(k) - {\bf x}^{*}\|^{2}
\end{split}
\end{equation}}
But the sufficient condition in \eqref{suffcond2aa} is satisfied for $ 0 < \mu < 1 $  and $ 0 < a <  L $ since $ \alpha_{k} \rightarrow 0 $. Thus, 
{ \begin{equation}\label{suffcond2aaa}
\begin{split}
& + \sum_{j=1}^{n}\sum_{(i)=1}^{p}\gamma_{(i)} \frac{8 \alpha_{k}^{2} a^{2}}{(\sum_{(i)=1}^{p}\gamma_{(i)})^{2}} ) \| {\bf x}^{*} - {\bf x}^{(i)} \| ^{2} \\
& + \sum_{j=1}^{n}\sum_{(i)=1}^{p}\gamma_{(i)} \frac{8 \alpha_{k}^{2} L^{2}}{(\sum_{(i)=1}^{p}\gamma_{(i)})^{2}} ) \| {\bf x}^{*} - {\bf x}^{(i)} \| ^{2} \\
& \leq ( \mu - \sum_{(i)=1}^{p}\gamma_{(i)} \frac{8 \alpha_{k}^{2}L^{2}}{(\sum_{(i)=1}^{p}\gamma_{(i)})^{2}} \sum_{j=1}^{n} \| {\bf v}_{j}(k) - {\bf x}^{*}\|^{2} \\
& \leq ( \mu - \sum_{(i)=1}^{p}\gamma_{(i)} \frac{8 \alpha_{k}^{2}a^{2}}{(\sum_{(i)=1}^{p}\gamma_{(i)})^{2}} \sum_{j=1}^{n} \| {\bf v}_{j}(k) - {\bf x}^{*}\|^{2}
\end{split}
\end{equation}}
Therefore, for $ k \geq k_{0} $ we have the right hand side of \eqref{suffcond2aaa} to be greater than $ E $ and the left hand side to be less than $ E $ since $ \alpha_{k} \rightarrow 0 $.
Since $ \max_{(i)} \| {\bf x}^{*}- {\bf x}^{(i)} \| ^{2} $ is fixed independent of $ k $ then \eqref{ineq_type2_error_bound_substa} is similar to the martingale inequality \eqref{eqw1} of Lemma~\ref{Lemma6a1}. By the result of Lemma~\ref{Lemma6a1} we have for $ v_{k}= \sum_{i=1}^{n} \| {\bf v}_{i}(k) - {\bf x}^{*} \|^{2} $ that
\begin{equation}
\begin{split}
    \sum_{i=1}^{n} \| {\bf v}_{i}(k) - {\bf x}^{*} \|^{2} \leq \rho ^{k}  V_{0} + \eta
\end{split}    
\end{equation}
for $ k \geq \bar{k}_{2} =\bar{k} = k^{'}_{0} + B $ where $ \rho $, $ V_{0} $, $ \bar{k} $ are as in the lemma and where $ \eta > 0 $ as substituted from the inequality \eqref{ineq_type2_error_bound_substa} by using the lemma. Therefore, as $ k \rightarrow \infty $ we have $ \sum_{i=1}^{n} \| {\bf v}_{i}(k) - {\bf x}^{*} \|^{2} \rightarrow \eta $. That is, $  \sum_{i=1}^{n} \| {\bf v}_{i}(k) - {\bf x}^{*} \|^{2} < D  $ for all $ i \in V $ as $ k \geq \bar{k}_{2} = \bar{k} $.

\end{proof}

\begin{prop}\label{Convergence_type2}
Let Assumptions 1-5 hold. Let the functions in Assumption 1 be strongly convex and satisfy  $ f^{(i)}({\bf x}^{*}) > f^{(i)}({\bf x}^{(i)}) $ for at least one $ (i) $ and let $ p < \frac{1}{\gamma_{min}} $. Let the sequences $ \{ {\bf x}_{i}(k) \} $ and $ \{ {\bf v}_{i}(k) \} $, $ i \in V  $ be generated by method \eqref{eq2} with stepsizes and errors as given in Assumptions 3 and 5. Assume that problem \eqref{eq1} has a non-empty optimal solution set $ \mathcal{X}^{*} $ as given in Assumption 1. Then, the sequences $ \{ {\bf x}_{i}(k) \} $ and $ \{ {\bf v}_{i}(k) \} $, $ i \in V  $ converge to the same random point in $ \mathcal{X}^{*} $ with probability 1.
\end{prop}

\begin{proof}

Having the assumptions of Proposition~\ref{Convergence_type2} holding then Proposition~\ref{sum_v_i-x^*_type2} is satisfied. Then Lemma~\ref{L3.5} premises are satisfied and the lemma follows. Then Lemma~\ref{L3.9} premises are satisfied and the lemma follows. Subsequently, Lemma~\ref{L3.11} premises are satisfied and the lemma follows proving the proposition.  
\end{proof}

\begin{prop}\label{Convergence_type2a}
Let Assumptions 1-5 hold. Let the functions in Assumption 1  satisfy  $ f^{(i)}({\bf x}^{*}) > f^{(i)}({\bf x}^{(i)}) $ for at least one $ (i) $ and let $ p < \frac{1}{\gamma_{min}} $. Let the sequences $ \{ {\bf x}_{i}(k) \} $ and $ \{ {\bf v}_{i}(k) \} $, $ i \in V  $ be generated by method \eqref{eq2} with stepsizes and errors as given in Assumptions 3 and 5. Assume that problem \eqref{eq1} has a non-empty optimal solution set $ \mathcal{X}^{*} $ as given in Assumption 1. Then, the sequences $ \{ {\bf x}_{i}(k) \} $ and $ {\bf v}_{i}(k) \} $, $ i \in V  $ converge to the same random point in $ \mathcal{X}^{*} $ with probability 1.
\end{prop}

\begin{proof}

Having the assumptions of Proposition~\ref{Convergence_type2a} holding then Proposition~\ref{sum_v_i-x^*_type2a} is satisfied. Then Lemma~\ref{L3.5} premises are satisfied and the lemma follows. Then Lemma~\ref{L3.9} premises are satisfied and the lemma follows. Subsequently, Lemma~\ref{L3.11} premises are satisfied and the lemma follows proving the proposition.  
\end{proof}

\begin{lem}\label{L3.5}
Let Assumption I hold, and $ [{\bf W}(k)] $ be row stochastic as in Assumption 2, let Proposition~\ref{sum_v_i-x^*_type2} or ~\ref{sum_v_i-x^*_type2a} be satisfied and $ p < \frac{1}{\gamma_{min}} $ and let $ \sum_{k=0}^{\infty} \alpha_{k} ^{2} < \infty $. Then,
$ \sum_{k=0}^{\infty}dist ^{2} ( {\bf v}_{i}(k), \mathcal{X} ) < \infty $ for all $ i \in V \ a.s. $ 
\end{lem}

\begin{proof}
We are going to prove Lemma~\ref{L3.5} by using the following lemmas~\ref{L3.6} and \ref{L3.7} 
and then exploiting the supermartingale Theorem~1.

\begin{remark}\label{bddX}
Note that the case of Least squares optimization for the set $ \mathcal{X} $ to acquire bounded gradients it must be a bounded set and thus from the representation theorem of closed bounded sets it must be that  $\mathcal{X} \triangleq \cap_{i=1}^{t}\mathcal{X}_{i}$ then let $\hat{{\bf x}} \in \mathcal{X}$. This implies that
$\hat{{\bf x}} \in \mathcal{X}_{i}$ for all $i \in \{1,2,\dots,t\}$.
And hence we only require $f_{i}$ to be Lipschitz on $\mathcal{X}_{i}$ for all $i \in \{1,2,\dots,t\}$ for Lemmas~\ref{L3.6} and \ref{L3.7} to be applied, then $ f_{j}$ is Lipschitz on $ \mathcal{X} $. However, in the analysis we require the estimates and weighted averages gradients to be also Lipschitz without insisting that they must lie in any $ \mathcal{X}_{i} $ (i.e., they can be anywhere in $ \mathbb{R}^{N} $. So to avoid any obstacle that can hinder the flow of the proof we assume Assumption~1(b) and require Lipschitz gradients on all of $ \mathbb{R}^{N} $ for any scenario of Problem~\eqref{eq1}. (In fact in our simulation where we investigate the Least squares problem, we take $ f_{j} $ is Lipschitz on the whole $ \mathbb{R}^{N} $, this for ease of implementation since we pick the matrix $ A $ randomly.
\end{remark}
Take $\overline{{\bf x}_{i}}(k+1) = {\bf x} - \alpha \nabla{ \Phi({\bf x})} \in \mathbb{R}^{N} $, then being in $\mathbb{R}^{N}$ the following inequalities hold,
{ \begin{equation*}\label{3.26}
\begin{split}
dist(\overline{{\bf x}_{i}}(k+1),\mathcal{X}) & =\|\overline{{\bf x}_{i}}(k+1)- \Pi_{\mathcal{X}}(\overline{{\bf x}_{i}}(k+1)) \| \\ 
& \leq \|\overline{{\bf x}_{i}}(k+1)-\hat{{\bf x}} \|.
\end{split}
\end{equation*}}

In fact, we can do better.
For any $\hat{{\bf x}} \in \mathcal{X}$
{ \begin{equation*}\label{3.27}
\begin{split}
\| \Pi_{\mathcal{X}}(&\overline{{\bf x}_{i}}(k+1)) - \hat{{\bf x}} \| ^{2}  \leq \\
 & \|\overline{{\bf x}_{i}}(k+1) - \hat{{\bf x}} \| ^{2} 
- \| \Pi_{\mathcal{X}}(\overline{{\bf x}_{i}}(k+1)) - \overline{{\bf x}_{i}}(k+1) \| ^{2} 
\end{split}
\end{equation*}
\begin{equation*}\label{3.28}
\begin{split}
\implies
 dist ^{2}(&\overline{{\bf x}_{i}}(k+1),\mathcal{X}) =  \| \Pi_{\mathcal{X}}
 (\overline{{\bf x}_{i}}(k+1)) - \overline{{\bf x}_{i}}(k+1) \| ^{2} \\ 
 \leq & \|\overline{{\bf x}_{i}}(k+1) - \hat{{\bf x}} \| ^{2} 
-\| \Pi_{\mathcal{X}}(\overline{{\bf x}_{i}}(k+1)) - \hat{{\bf x}} \| ^{2}. 
\end{split}
\end{equation*}}
\vspace{-0.2cm}
This implies that
{ \begin{equation}\label{3.29}
\begin{split}
 dist ^{2} (&\overline{{\bf x}_{i}}(k+1),\mathcal{X}) =  \| \Pi_{\mathcal{X}}
 (\overline{{\bf x}_{i}}(k+1)) - \overline{{\bf x}_{i}}(k+1) \| ^{2} \\
 \leq & \|\overline{{\bf x}_{i}}(k+1) - \hat{{\bf x}} \| ^{2} 
- a \| \Pi_{\mathcal{X}}(\overline{{\bf x}_{i}}(k+1)) - \hat{{\bf x}} \| ^{2}, 
\end{split}
\end{equation}}
where $a \leq 1$.
\\
We used this form of the inequality with a constant $ a $ that we can adjust in order for the supermartingale theorem hypothesis to hold.

\begin{remark}
Notice from the procedure of the proof that $ \tau, \eta > 0 $ are arbitrary and independent from any quantity, i.e.,they are not related to any bound on any quantity.In other words, the lemmas inequalities can adapt to any values $ \tau, \eta > 0 $.
\end{remark}

Using equation (\ref{3.29}) and substituting the results of Lemma~\ref{L3.6} and Lemma~\ref{L3.7} on the above substitutions, we have for each use of $ f^{(i)} $ the following applies on node $ i $

{ \begin{equation}\label{3.33}
\begin{split}
dist ^{2} ( & \overline{{\bf x}_{i}}(k+1), \mathcal{X}) \leq  \\
& \|\overline{{\bf x}_{i}}(k+1) - \hat{{\bf x}} \| ^{2} 
-  a \| \Pi_{\mathcal{X}}(\overline{{\bf x}_{i}}(k+1)) - \hat{{\bf x}} \| ^{2} \\
\leq & \ (1-a+(A_{\eta}-a A_{\tau})\alpha^{2}) \| {\bf x} -\hat{{\bf x}}\| ^{2} \\
 &-(2-2a) \alpha ( \Phi({\bf z})- \Phi(\hat{{\bf x}}) )+\frac{3a}{4} \|{\bf y} - {\bf x} \| ^{2} \\
 & +(\frac{3}{8 \eta}-\frac{3a}{8 \tau}+ (2-2a) \alpha L ) \|{\bf x}- {\bf z} \|^{2} \\ 
 & +(B_{\eta}-aB_{\tau})\alpha^{2} \| \nabla{ \Phi(\hat{{\bf x}})}\| ^{2},
\end{split}
\end{equation}} 
where $A_{\tau}=8L^{2}+16 \tau L^{2}$, $B_{\tau}=8 \tau + 8 $, $ \tau > 0 $ is arbitrary, $A_{\eta}=2L^{2}+16 \eta L^{2}$, $B_{\eta}=2 \eta + 8 $ and $ \eta > 0 $ is arbitrary.
Here, we note that $\hat{{\bf x}} \in \mathcal{X}$.

Now, we use the above inequality, (\ref{3.33}) for the purpose of making the hypothesis of the supermartingale theorem hold. We thus use the following substitutions:
$  \Phi = f^{(i)} $,
$ \alpha = \alpha_{k} $,
$ \hat{{\bf x}} \in \mathcal{X}  =\cap_{i=1}^{n}\mathcal{X}_{i} $,
$ {\bf y} = \Pi_{\mathcal{X}} (\overline{{\bf x}_{i}}(k+1)) $,
$ \overline{{\bf x}_{i}}(k+1) = {\bf x} - \alpha_{k} \nabla{f({\bf x})} $ and
$ {\bf x} = {\bf v}_{i}(k) $. 
 
In particular, if we take $ \hat{{\bf x}} = \Pi_{\mathcal{X}}({\bf v}_{i}(k)) \in \mathcal{X} $ in the feasibility region and $ {\bf z} = \Pi_{\mathcal{X}}({\bf v}_{i}(k)) = \hat{{\bf x}} $ then the above conclusion (\ref{3.33}) becomes

{ \begin{equation}\label{3.35}
\begin{split}
dist ^{2} & ( \overline{{\bf x}_{i}}(k+1),\mathcal{X}) \leq  (1-a+(A_{\eta}-a A_{\tau})\alpha_{k} ^{2}) dist ^{2}  ( {\bf v}_{i}(k) ,\mathcal{X})\\ 
 & +\frac{3a}{4} \| \Pi_{\mathcal{X}} [{\bf v}_{i}(k) - \alpha_{k} \nabla{f}^{(i)}({\bf v}_{i}(k))]-{\bf v}_{i}(k)\| ^{2} \\
 & +(\frac{3}{8 \eta}-\frac{3a}{8 \tau}+ (2-2a) \alpha_{k} L ) dist ^{2}  ( {\bf v}_{i}(k) ,\mathcal{X}) \\
 & +(B_{\eta}- a B_{\tau})\alpha_{k}^{2} \| \nabla{f^{(i)}(\hat{{\bf x}})}\| ^{2},
\end{split}
\end{equation}} 

But by using the nonexpansiveness property of a projection we have
{ \begin{equation*}\label{3.36}
\begin{split}
 \| \Pi_{\mathcal{X}}& ({\bf v}_{i}(k) - \alpha_{k} \nabla{f^{(i)}({\bf v}_{i}(k))})-{\bf v}_{i}(k)\|^{2} \leq \| \alpha_{k} \nabla{f^{(i)}({\bf v}_{i}(k)}  \|^{2} \\
 & - \| \Pi_{\mathcal{X}} ({\bf v}_{i}(k) - \alpha_{k} \nabla{f^{(i)}({\bf v}_{i}(k))})-{\bf v}_{i}(k) - \alpha_{k} \nabla{f^{(i)}({\bf v}_{i}(k))}\| ^{2} 
\end{split}
\end{equation*}} 
Then for $ 0 \leq a \leq 1 $. Then from $ \hat{{\bf x}}=\Pi_{\mathcal{X}}({\bf v}_{i}(k)) $, $ (a+b)^{2} \leq 2a^{2}+2b^{2} $ and the Lipschitz property of $ f $ from Assumption~\ref{A3.1}~(b), we get

{ \begin{equation*}\label{3.37}
\begin{split}
 \frac{3a}{4}\| \Pi_{\mathcal{X}} ({\bf v}_{i}(k) & - \alpha_{k} \nabla{f^{(i)}({\bf v}_{i}(k))})-{\bf v}_{i}(k)\|  ^{2} 
 \leq \\ 
 & \frac{3a}{4}  \| \alpha_{k} \nabla{f^{(i)}({\bf v}_{i}(k))}  \|^{2}  - \frac{3a}{4} dist ^{2} ( {\bf v}_{i}(k) - \alpha_{k} \nabla{f^{(i)}({\bf v}_{i}(k))} ,\mathcal{X}) \\
& \leq  \frac{3a}{4}  \| \alpha_{k} \nabla{f^{(i)}({\bf v}_{i}(k))} \|^{2}
\end{split}
\end{equation*}}

Then for $ \hat{{\bf x}}=\Pi_{\mathcal{X}}({\bf v}_{i}(k)) $, $ (a+b)^{2} \leq 2a^{2}+2b^{2} $ and the Lipschitz property of $ f $ from Assumption~\ref{A3.1}~(b), we get

{ \begin{equation*}\label{3.39}
\begin{split}
\implies \frac{3a}{4}\| \Pi_{\mathcal{X}} & ({\bf v}_{i}(k)  - \alpha_{k} \nabla{f^{(i)}({\bf v}_{i}(k))})-{\bf v}_{i}(k)\|  ^{2}  \leq \\
& \| \alpha_{k} \nabla{f^{(i)}({\bf v}_{i}(k))} - \nabla{f(\hat{{\bf x}})} + \nabla{f^{(i)}(\hat{{\bf x}})}\|^{2} \\
& \leq  2 \alpha_{k}^{2} L^{2} \| {\bf v}_{i}(k) -\hat{{\bf x}} \| ^{2} + 2 \alpha_{k}^{2}\|\nabla{f^{(i)}(\hat{{\bf x}})}\|^{2} \\
& \leq  2 \alpha_{k}^{2} L^{2} dist ^{2}({\bf v}_{i}(k),\mathcal{X}) + 2 \alpha_{k}^{2}\|\nabla{f^{(i)}(\hat{{\bf x}})}\|^{2}.
\end{split}
\end{equation*}}

Then the result (\ref{3.35}) becomes

{ \begin{equation}\label{3.41}
\begin{split}
dist ^{2} & ( \overline{{\bf x}_{i}}(k+1),\mathcal{X}) \leq \\
 & (1-a+(A_{\eta}- a A_{\tau})\alpha_{k} ^{2}) dist ^{2}  ( {\bf v}_{i}(k) ,\mathcal{X})\\ 
 & +(\frac{3}{8 \eta} - \frac{3a}{8 \tau} +  2 \alpha_{k}^{2} L^{2}  + (2-2a) \alpha_{k} L ) dist ^{2}  ( {\bf v}_{i}(k) ,\mathcal{X}) \\
 & +(B_{\eta} - a B_{\tau} + 2) \alpha_{k} ^{2} \| \nabla{f^{(i)}(\hat{{\bf x}})}\| ^{2} 
\end{split}
\end{equation}} 

From \eqref{3.43} and \eqref{eq5} substitution we can write
{ \begin{equation}\label{3.44a}
\begin{split}
& \hspace{3cm} {\bf x}_{i}(k+1)=\overline{{\bf x}_{i}}(k+1) + R_{i}(k), \\
& \text{where}  \ \ \overline{{\bf x}_{i}}(k+1)={\bf v}_{i}(k)-\alpha_{k}\nabla{f^{(i)}({\bf v}_{i}(k))}
\end{split}
\end{equation}}
(cf. Definition~\ref{Rk} for the definition of $ R_{i}(k) $).

But by using  an equivalent of \eqref{3.44a} for instant $ k $ instead of $ k+1 $, we have
{\small \begin{equation}\label{3.45}
\begin{split}
dist & ({\bf x}_{j}(k),\mathcal{X})  = \| \overline{{\bf x}_{j}}(k) + R_{j}(k-1) - \Pi_{\mathcal{X}} [\overline{{\bf x}_{j}}(k) + R_{j}(k-1)] \| \\
& \leq \| \overline{{\bf x}_{j}}(k) - \Pi_{\mathcal{X}} [\overline{{\bf x}_{j}}(k) ] + R_{j}(k-1) \| \\
& \leq dist(\overline{{\bf x}_{j}}(k),\mathcal{X}) +  \| R_{j}(k-1) \|.
\end{split}
\end{equation}}

But in our algorithm we have,
\vspace{-0.2cm}
{ \begin{equation}\label{3.42}
{\bf v}_{i}(k)=\sum_{j=1}^{n}[{\bf W}(k)]_{ij}(k){\bf x}_{j}(k),
\end{equation}}

Then from the convexity of the norm squared, we have
{ \begin{equation}\label{3.44}
dist ^{2}  ( {\bf v}_{i}(k) ,\mathcal{X})\leq \sum_{j=1}^{n}[{\bf W}(k)]_{ij} dist ^{2} ( {\bf x}_{j}(k) ,\mathcal{X}).
\end{equation}}

But
{ \begin{equation}\label{3.49}
\begin{split}
 (  1  - a  +  (A_{\eta} - a A_{\tau}) & \alpha_{k} ^{2})  \ dist ^{2}  ( {\bf v}_{i}(k) ,\mathcal{X}) = \\
& (1 +  (A_{\eta} - a A_{\tau})\alpha_{k} ^{2})  \ dist ^{2}  ( {\bf v}_{i}(k) ,\mathcal{X}) \\
& +  (1 - a) \ dist ^{2}  ( {\bf v}_{i}(k) ,\mathcal{X}), \end{split}
\end{equation}}

Then by using \eqref{3.49} and \eqref{3.44}, \eqref{3.41} becomes

{ \begin{equation}\label{3.51}
\begin{split}
& \hspace{2cm} dist ^{2}  ( \overline{{\bf x}_{i}}(k+1),\mathcal{X}) \leq \\
& (1+ (A_{\eta} - a A_{\tau})\alpha_{k} ^{2})  \sum_{j=1}^{n}[{\bf W}(k)]_{ij} dist ^{2} ( {\bf x}_{j}(k) ,\mathcal{X}) \\
 & + (1 - a) \ dist ^{2}  ( {\bf v}_{i}(k) ,\mathcal{X}) + (B_{\eta}+ a B_{\tau} + 2) \alpha_{k} ^{2}  G_{f} ^{2}  \\
 & + (\frac{3}{8 \eta} - \frac{3a}{8 \tau} +  2 \alpha_{k}^{2} L^{2}  + (2-2a) \alpha_{k} L ) dist ^{2}  ( {\bf v}_{i}(k) ,\mathcal{X})
\end{split}
\end{equation}} 
where we used $ \| \nabla{f^{(i)}(\hat{{\bf x}})}\|  \leq G_{f} $ (i.e., gradient is bounded on set $\mathcal{X}$). 

Let  $ \mathcal{F}_{k} $ be the $ \sigma$-algebra generated by the entire history of the algorithm up to time $k$ inclusively, that is $ \mathcal{F}_{k} = \{ {\bf x}_{i}(0), i \in V \} \cup \{ \Omega_{i}(l): 0 \leq l \leq k, i \in V \} $.
Therefore, given $ \mathcal{F}_{k} $, the collection $ {\bf x}_{i}(0),...,{\bf x}_{i}(k+1)$ and $ {\bf v}_{i}(0),...,{\bf v}_{i}(k+1)$ generated by the algorithm is fully determined.
So by applying expectation on (\ref{3.51}), and taking consideration the probability $ 0 \leq \gamma_{(i)} \leq 1  $ of accessing a partition per each server, we have for $ k > 0 $, a.s. that

{ \begin{equation}\label{3.51a}
\begin{split}
\mathbb{E} & [ dist ^{2} ( \overline{{\bf x}_{i}}(k+1),\mathcal{X}) / \mathcal{F}_{k} ] \leq \\
 & \sum_{(i)=0}^{p} \gamma_{(i)}(1+ (A_{\eta} - a A_{\tau})\alpha_{k} ^{2})  \sum_{j=1}^{n}[{\bf W}(k)]_{ij} dist ^{2} ( {\bf x}_{j}(k) ,\mathcal{X}) \\
 & + \sum_{(i)=0}^{p} \gamma_{(i)}(1 - a) \ dist ^{2}  ( {\bf v}_{i}(k) ,\mathcal{X}) + \sum_{(i)=1}^{p} \gamma_{(i)}(B_{\eta} - a B_{\tau} + 2) \alpha_{k} ^{2}  G_{f} ^{2} \\
 & + \sum_{(i)=1}^{p} \gamma_{(i)}(\frac{3}{8 \eta} - \frac{3a}{8 \tau} +  2 \alpha_{k}^{2} L^{2}  + (2-2a) \alpha_{k} L ) dist ^{2}  ( {\bf v}_{i}(k) ,\mathcal{X})
\end{split}
\end{equation}} 

Then by using $ 0 \leq  \gamma_{min} \leq \gamma_{(i)} \leq 1 $ for  partitions $ (i) \in \{1,\ldots,p\} $ and choosing $ 0 < a < 1 $ and the fact that $ \sum_{(i)=1}^{p} \gamma_{(i)} \leq \sum_{(i)=0}^{p} \gamma_{(i)} = 1 $, the above reduces to

{ \begin{equation}\label{3.51a}
\begin{split}
\mathbb{E} & [ dist ^{2} ( \overline{{\bf x}_{i}}(k+1),\mathcal{X}) / \mathcal{F}_{k} ] \leq \\
 & (1+  (A_{\eta} - a A_{\tau})\alpha_{k} ^{2})  \sum_{j=1}^{n}[{\bf W}(k)]_{ij} dist ^{2} ( {\bf x}_{j}(k) ,\mathcal{X}) \\
 & + (1 - a) \ dist ^{2}  ( {\bf v}_{i}(k) ,\mathcal{X}) \\
 & +  (\frac{3}{8 \eta} - \sum_{(i)=1}^{p} \frac{3\gamma_{(i)}a}{8 \tau} +  2 \alpha_{k}^{2} L^{2}  + (2-2a) \alpha_{k} L ) dist ^{2}  ( {\bf v}_{i}(k) ,\mathcal{X})\\
 & (B_{\eta} - a B_{\tau} + 2) \alpha_{k} ^{2}  G_{f} ^{2}.
\end{split}
\end{equation}} 

Then summing overall $n$, and having that $[{\bf W}(k)]_{ij}$ doubly stochastic, we have a.s. that

\vspace{-0.5cm}
{ \begin{equation*}\label{3.52a}
\begin{split}
\mathbb{E} & [ \sum_{i=1}^{n} dist ^{2} ( \overline{{\bf x}_{i}}(k+1),\mathcal{X}) / \mathcal{F}_{k} ] \leq \\
 & (1+  (A_{\eta} - a A_{\tau})\alpha_{k} ^{2})  \sum_{i=1}^{n} dist ^{2} ( {\bf x}_{j}(k) ,\mathcal{X}) \\ 
 & +   (\frac{3}{8 \eta} - \sum_{(i)=1}^{p} \frac{3\gamma_{(i)}a}{8 \tau}  +  2 \alpha_{k}^{2} L^{2}  + (2-2a) \alpha_{k} L ) \sum_{i=1}^{n}  dist ^{2}  ( {\bf v}_{i}(k) ,\mathcal{X}) \\
 & + \ n (B_{\eta} - a B_{\tau} + 2) \alpha_{k} ^{2} G_{f} ^{2} +  (1 - a) \  dist ^{2} \sum_{i=1}^{n}  ( {\bf v}_{i}(k) ,\mathcal{X}) 
\end{split}
\end{equation*}}

From \eqref{3.43} and \eqref{eq5} substitution we can write
{ \begin{equation}\label{3.44a}
\begin{split}
& \hspace{3cm} {\bf x}_{i}(k+1)=\overline{{\bf x}_{i}}(k+1) + R_{i}(k), \\
& \text{where}  \ \ \overline{{\bf x}_{i}}(k+1)={\bf v}_{i}(k)-\alpha_{k}\nabla{f^{(i)}({\bf v}_{i}(k))}
\end{split}
\end{equation}}
(cf. Definition~\ref{Rk} for the definition of $ R_{i}(k) $).

But by using  an equivalent of \eqref{3.44a} for instant $ k $ instead of $ k+1 $, we have
{\small \begin{equation}\label{3.45}
\begin{split}
dist & ({\bf x}_{j}(k),\mathcal{X})  = \| \overline{{\bf x}_{j}}(k) + R_{j}(k-1) - \Pi_{\mathcal{X}} [\overline{{\bf x}_{j}}(k) + R_{j}(k-1)] \| \\
& \leq \| \overline{{\bf x}_{j}}(k) - \Pi_{\mathcal{X}} [\overline{{\bf x}_{j}}(k) ] + R_{j}(k-1) \| \\
& \leq dist(\overline{{\bf x}_{j}}(k),\mathcal{X}) +  \| R_{j}(k-1) \|.
\end{split}
\end{equation}}

Using $ 2ab \leq a^{2}+b^{2} $ and squaring both sides, the above becomes
{ \begin{equation}\label{3.47}
\begin{split}
 dist ^{2}  ({\bf x}_{j}(k),  \mathcal{X})
 \leq  & 2 dist ^{2} (\overline{{\bf x}_{j}}(k),\mathcal{X}) + 2 \| R_{j}(k-1) \| ^{2}.
\end{split}
\end{equation}}

{\bf Applying the Supermartingale Theorem in Lemma~\ref{L3.5}:} Then using the result of (\ref{3.47}) and writing the above in the format of the supermartingale theorem, we see that

{ \begin{equation}\label{3.55a}
\begin{split}
\mathbb{E} & [ \sum_{i=1}^{n} dist ^{2} ( \overline{{\bf x}_{i}}(k+1),\mathcal{X}) / \mathcal{F}_{k} ] \leq \\
 &   (1+  (A_{\eta} - a A_{\tau})\alpha_{k} ^{2})  \sum_{i=1}^{n}  dist ^{2} (\overline{{\bf x}_{i}}(k),\mathcal{X})\\
 & + ( 1 - a + \frac{3}{8 \eta} \\
 & - \frac{3\gamma_{min}ap}{8 \tau}  +  2 \alpha_{k}^{2} L^{2}  + (2-2a) \alpha_{k} L )  \sum_{i=1}^{n} dist ^{2}  ( {\bf v}_{i}(k) ,\mathcal{X}) \\
 &  + n (B_{\eta} - a B_{\tau} + 2) \alpha_{k} ^{2} G_{f} ^{2} \\
 & +  (1 +  (A_{\eta} - a A_{\tau})\alpha_{k} ^{2}) \sum_{i=1}^{n}  \| R_{i}(k-1) \| ^{2}  \ \ a.s.
\end{split}
\end{equation}}

Then by using the result of (\ref{3.67}) which implies the boundedness of $ \sum_{i=1}^{n}\|R_{i}(k-1)\|^{2} $ for $ k \geq \bar{k}_{1} $, we can apply the supermartingale convergence theorem on (\ref{3.55a}) with the following substitutions : 
$ v_{k+1}  = \sum_{i=1}^{n} dist ^{2}(\overline{{\bf x}_{i}}(k+1),\mathcal{X}) $, 
i.e., $ v_{k}  = \sum_{i=1}^{n} dist ^{2}(\overline{{\bf x}_{i}}(k),\mathcal{X}) $, 
$ a_{k} =  (A_{\eta} - a A_{\tau}) \alpha _{k} ^{2} $ and $ A_{\eta} - a A_{\tau} > 0 $,
$ b_{k}   =  \ n (B_{\eta} - a B_{\tau} + 2) \alpha_{k} ^{2} G_{f} ^{2} 
 +  (1 + (A_{\eta} -a A_{\tau})\alpha_{k} ^{2}) \sum_{i=1}^{n}\|R_{i}(k-1)\|^{2} $
where $ \sum_{k=0} ^{\infty} a_{k}<\infty $ and $ \sum_{k=0} ^{\infty} b_{k}< \infty $ since $ \sum_{k=0} ^{\infty}\alpha_{k}^{2} < \infty  $ and $ u_{k} = \sum_{i=1}^{n} dist ^{2} ({\bf v}_{i}(k),\mathcal{X}) $.
i.e., for $ b_{k} $ notice that $ \sum_{i=1}^{n}\|R_{i}(k-1)\|^{2} $ is bounded since  $ \sum_{i=1}^{n}\|{\bf v}_{i}(k-1)- x^{*}\|^{2}  < D < \infty $ for every $ k > \bar{k}_{1} $ from the Appendix~\ref{appB} (i.e., Proposition ~\ref{sum_v_i-x^*_type2} or ~\ref{sum_v_i-x^*_type2a}).

But  $ u_{k}= \sum_{i=1}^{n} dist^{2}({\bf v}_{i}(k),{\mathcal{X}})) \leq  \sum_{i=1}^{n}\|{\bf v}_{i}(k)- x^{*}\|^{2}  < D < \infty $ for $ k \geq \bar{k}_{1} $, (i.e., Proposition ~\ref{sum_v_i-x^*_type2} or ~\ref{sum_v_i-x^*_type2a}), and its  coefficient is $ ( 1 - a + \frac{3}{8 \eta} - \frac{3 \gamma_{min}ap}{8 \tau} +  2 \alpha_{k}^{2} L^{2} + (2-2a) \alpha_{k}L) < -1 $ for $k \geq k_{1} $ in order to apply the supermartingale theorem.

\begin{remark}

But $ \alpha_{k} \rightarrow 0 $ (since $ \alpha_{k} $ is chosen such that $ \sum_{k=0}^{\infty} \alpha_{k}^{2} < \infty $). Thus, we can bound $ 2 \alpha_{k}^{2} L^{2} + (2-2a) \alpha_{k}L < \epsilon $ as $ \alpha_{k} \rightarrow 0 $.
Thus, $  2 + \frac{3}{8 \eta} + \epsilon  <   a  + \frac{3 \gamma_{min}ap}{8 \tau} $ 
Thus, taking $ \epsilon < 1 $ and $ \eta > \tau $ where $ \eta = l \tau $ and $ l > 1 $ such that what preceded applies. Then $ a > \frac{ 16 \eta + 3 } { (3  \gamma_{min} p + 8 \tau)l }$ where we can choose $ l $ such that $ a < 1 $.
Therefore, by reducing this inequality with the conditions on the values of the above variables we have $ 8 l \tau < 3(\gamma_{min} p l - 1)$ where we get a sufficient condition for convergence which is that $ \gamma_{min} p l >  1 $, but $ l > \frac{1}{\gamma_{min} p} > 1 $, then a sufficient condition is for $ p < \frac{1}{\gamma_{min}} $.

e.g. for $ l=3 $, $ \gamma_{min}p = 2 $ we have $ \tau < \frac{3}{8} $, so we can choose $ \tau = \frac{1}{4} $, $ \eta = \frac{3}{4} $ and $ 1 > a > \frac{5}{8} $. Thus, choosing an $ a = \frac{3}{4} $ is sufficient for this Lemma~1 to follow.

Thus, the coefficient of $ u_{k} $ is negative for $ k \geq k_{1} $.

\end{remark}

Then from the supermartingale theorem holding for the tail of the sequences (i.e., $ k > \max(k_{1},\bar{k}_{1}))$ we have $\sum_{k=0}^{\infty} u_{k} <\infty $. That is, $ \sum_{k=0}^{\infty}\sum_{i=1}^{n} dist ^{2} ({\bf v}_{i}(k),\mathcal{X}) < \infty $.
We can interchange infinite and finite sums, as an implicit consequence of the linearity of these sums. Thus, we have
$ \sum_{i=1}^{n}( \sum_{k=0}^{\infty} dist ^{2} ({\bf v}_{i}(k),\mathcal{X}) ) < \infty$
$ \implies $ the argument inside the finite sum is bounded, i.e., 
{ 
\begin{equation}
\begin{split}
\sum_{k=0}^{\infty} dist ^{2} ({\bf v}_{i}(k),\mathcal{X}) < \infty,
\end{split}
\end{equation}}
the result we require. And $ \lim_{k \rightarrow \infty}dist ^{2} ({\bf v}_{i}(k),\mathcal{X}) = 0 $.
\end{proof}

\begin{lem}\label{L3.9}
Let Assumptions~\ref{A3.1} hold and Proposition~\ref{sum_v_i-x^*_type2} or ~\ref{sum_v_i-x^*_type2a} satisfied. Also, assume that the stepsize sequence $ \{ \alpha_{k} \} $ is non-increasing such that $ \sum_{k=0}^{\infty} \alpha_{k} ^{2} < \infty $, 
and define ${\bf \epsilon}_{i}(k) = {\bf x}_{i}(k+1)- {\bf v}_{i}(k) $ for all $ i \in V $ and $ k \geq 0 $. Then, we have $ a.s. $ 
{ \begin{equation}\label{3.76}
\begin{split}
\sum_{k=0}^{\infty}||{\bf \epsilon}_{i}(k) || ^{2} < \infty \ for \ all \ i \in V, \\ \sum_{k=0}^{\infty} \alpha_{k} || {\bf v}_{i}(k) - \overline{{\bf v}}(k) || < \infty \ for \ all \ i \in V , \\
where \ \overline{{\bf v}}(k) = \frac{1}{n}\sum_{l=1}^{n}{\bf v}_{l}(k).
\end{split}
\end{equation}}
\end{lem}

\begin{proof}
Let 
\vspace{-0.5cm}
\begin{equation}\label{3.79}
\begin{split}
& \ \ \ \ \ {\bf \epsilon}_{i}(k)  = {\bf x}_{i}(k+1)-{\bf v}_{i}(k) \\
& \ \ \ \ \  = \overline{{\bf x}_{i}}(k+1) + R_{i}(k) -{\bf v}_{i}(k) \\
& \ \ \ \ \  = {\bf v}_{i}(k) - \alpha_{k}\nabla{f^{(i)}({\bf v}_{i}(k))} + R_{i}(k) -{\bf v}_{i}(k) \\
& \ \ \ \ \ = - \alpha_{k}\nabla{f^{(i)}({\bf v}_{i}(k))} + R_{i}(k).
\end{split}
\end{equation}

Notice that in our case, $ f^{(i)} $ is not fixed and can vary on the same node $ i $ and this also complies with $ \epsilon_{i} $ in lemma 6 in \citep{lee2013distributed} which is an arbitrary error at node $ i $.

And, let ${\bf z}_{i}(k)=\Pi_{\mathcal{X}}[{\bf v}_{i}(k)]$ then we have 
{\small \begin{equation}\label{3.80aa}
\begin{split}
& \|{\bf \epsilon}_{i}(k) \|  \leq \|{\bf x}_{i}(k+1)-{\bf z}_{i}(k) \| +\| {\bf z}_{i}(k)-{\bf v}_{i}(k) \| \\
& \leq \| {\bf v}_{i}(k) - \alpha_{k} \nabla{f^{(i)}({\bf v}_{i}(k))} + R_{i}(k) -{\bf z}_{i}(k) \| + \| {\bf z}_{i}(k) -{\bf v}_{i}(k) \| \\
& \leq 2 \| {\bf v}_{i}(k) - {\bf z}_{i} (k) \| + \alpha_{k}\| \nabla{f^{(i)}({\bf v}_{i}(k))} \| +\| R_{i}(k) \| \\
& \leq  2 \| {\bf v}_{i}(k) - {\bf z}_{i} (k) \| + \alpha_{k}\| \nabla{f^{(i)}({\bf z}_{i}(k))} \| \\
& \ \  + \alpha_{k} \| \nabla{f^{(i)}({\bf z}_{i}(k))} - \nabla{f^{(i)}({\bf v}_{i}(k))} \| +\| R_{i}(k) \|.
\end{split}
\end{equation}}
But $ \| \nabla{f^{(i)}({\bf z}_{i}(k))} \| \leq G_{f} $ since $ z = \Pi_{\mathcal{X}}({\bf v}_{i}(k)) \in \mathcal{X}$, \\
and $ \| \nabla{f^{(i)}({\bf z}_{i}(k))} - \nabla{f^{(i)}({\bf v}_{i}(k))} \| \leq L \| {\bf v}_{i}(k) - {\bf z}_{i}(k) \| $.

Therefore, (\ref{3.80aa}) is reduced to
{\small \begin{equation*}\label{3.82}
\begin{split}
\| &{\bf \epsilon}_{i}(k) \|  \leq  \ (2 + \alpha_{k} L) \| {\bf v}_{i}(k) - {\bf z}_{i} (k) \|  + \alpha_{k} G_{f} \\
 & \alpha_{k} ^{2} \|{\bf A}^{(i)} \| ^{2} _{\infty} \|{\bf B}^{(i)}\| _{2,\infty} ^{2} 4 n L^{2} \max_{k-H \leq \hat{k} \leq k} \| {\bf v}_{q}(\hat{k}) -  x^{*}\|^{2}
\end{split}
\end{equation*}}

From $(a+b+c)^{2} \leq 3a^{2}+3b^{2}+3c^{2}$, we get
{\small \begin{equation*}\label{3.83}
\begin{split}
\| &{\bf \epsilon}_{i}(k)  \| ^{2} \leq  3 (2 + \alpha_{k} L)^{2} \| {\bf v}_{i}(k) - {\bf z}_{i} (k) \|^{2} + 3\alpha_{k}^{2} G_{f}^{2} \\ 
& + 48 \alpha_{k} ^{4} \|{\bf A}^{(i)} \| ^{4} _{\infty} \|{\bf B}^{(i)}\| _{2,\infty} ^{4} n^{2} L^{4} \max_{k-1-H \leq \hat{k} \leq k-1} \| {\bf v}_{q}(\hat{k}) -  x^{*}\|^{4}.
\end{split}
\end{equation*}}

{\small \begin{equation}\label{3.83a}
\begin{split}
\sum_{k=0}^{\infty} & \|{\bf \epsilon}_{i}(k)  \| ^{2}  \leq  3 \sum_{k=0}^{\infty}(4 + \alpha_{k}^{2} L^{2}) \| {\bf v}_{i}(k) - {\bf z}_{i} (k) \|^{2} \\
& + 12 \sum_{k=0}^{\infty}\alpha_{k} L \| {\bf v}_{i}(k) - {\bf z}_{i} (k) \|^{2} 
+  3 \sum_{k=0}^{\infty}\alpha_{k}^{2} G_{f}^{2} \\
& + 48 \sum_{k=0}^{\infty} \alpha_{k} ^{4} \|{\bf A}^{(i)} \| ^{4} _{\infty} \|{\bf B}^{(i)}\| _{2,\infty} ^{4} n^{2} L^{4} \max_{k-1-H \leq \hat{k} \leq k-1} \| {\bf v}_{q}(\hat{k}) -  x^{*}\|^{4}.
\end{split}
\end{equation}}

But, we also have from Lemma~\ref{L3.5}
\begin{equation}\label{3.84}
\begin{split}
\sum_{k=0}^{\infty} \| {\bf v}_{i}(k) - {\bf z}_{i}(k) \| ^{2} = & \sum_{k=0}^{\infty} \| {\bf v}_{i}(k) - \Pi_{\mathcal{X}}({\bf v}_{i}(k)) \| ^{2} \\
= & \sum_{k=0}^{\infty} dist ^{2} ({\bf v}_{i}(k), \mathcal{X}) < \infty,
\end{split}
\end{equation}
and 
\begin{equation}\label{3.84a}
\begin{split}
\sum_{k=0}^{\infty}\alpha_{k} \| {\bf v}_{i}(k) & - {\bf z}_{i}(k) \| ^{2} 
=  \sum_{k=0}^{\infty}  \alpha_{k}dist ^{2} ({\bf v}_{i}(k), \mathcal{X}) \\
= & \frac{1}{2}\sum_{k=0}^{\infty}\alpha_{k}^{2} + \frac{1}{2}\sum_{k=0}^{\infty}dist ^{4} ({\bf v}_{i}(k), \mathcal{X}) < \infty,
\end{split}
\end{equation}

and since $\sum_{k=0}^{\infty}\alpha_{k}^{2} < \infty $ which also implies that $\sum_{k=0}^{\infty}\alpha_{k}^{4} < \infty $.
And $ \sum_{i=1}^{n}\| R_{i}(k) \|^{2} $ is bounded since $ \|{\bf v}_{i}(k) - x^{*} \| ^{2} $ is bounded (i.e., Proposition ~\ref{sum_v_i-x^*_type2} or ~\ref{sum_v_i-x^*_type2a}).


From the preceding four assumptions we have that 
in (\ref{3.83a}),
\begin{equation}\label{3.85}
\begin{split}
\sum_{k=0}^{\infty}\|{\bf \epsilon}_{i}(k) \| ^{2} < \infty.
\end{split}
\end{equation} 

By applying  $2ab \leq a^{2}+b^{2}$ on $\sum_{k=0}^{\infty}\alpha_{k}^{2} < \infty$ and $ \sum_{k=0}^{\infty} \|{\bf \epsilon}_{i}(k) \| ^{2} < \infty $, we have 
\begin{equation}\label{3.86}
\begin{split}
\sum_{k=0}^{\infty} \alpha_{k} \|{\bf \epsilon}_{i}(k) \|  \leq \frac{1}{2}\sum_{k=0}^{\infty}\alpha_{k}^{2} + \frac{1}{2} \sum_{k=0}^{\infty} \|{\bf \epsilon}_{i}(k) \| ^{2} < \infty.
\end{split}
\end{equation} 
But ${\bf x}_{i}(k+1) = {\bf v}_{i}(k) +{\bf \epsilon}_{i}(k)$ and $ {\bf v}_{i}(k) = \sum_{j=1}^{n}[{\bf W}(k)]_{ij}{\bf x}_{j}(k) $ where $ \sum_{j=1}^{n}[{\bf W}(k)]_{ij} =1 $  and $ \sum_{k=0}^{\infty} \alpha_{k} \|{\bf \epsilon}_{i}(k) \|  < \infty $ a.s. \\

Therefore, by Lemma~\ref{L3.10} this implies that
$ \sum_{k=0}^{\infty}\alpha_{k} \|{\bf x}_{i}(k) -{\bf x}_{j}(k) \| < \infty $ \\
Now, let $\bar{{\bf v}}(k)=\frac{1}{n}\sum_{i=1}^{n}{\bf v}_{i}(k)$. \\
Since $ {\bf v}_{i}(k) = \sum_{j=1}^{n}[{\bf W}(k)]_{ij}{\bf x}_{j}(k) $ where $ \sum_{j=1}^{n}[{\bf W}(k)]_{ij} = 1 $ implies that 
\begin{equation}\label{3.87}
\begin{split}
\|{\bf v}_{i}(k)-\bar{{\bf v}}(k) \| = & \| \sum_{j=1}^{n}[{\bf W}(k)]_{ij}{\bf x}_{j}(k) - \sum_{j=1}^{n}[{\bf W}(k)]_{ij} \bar{{\bf v}}(k) \| \\ 
\leq & \sum_{j=1}^{n}[{\bf W}(k)]_{ij} \| {\bf x}_{j}(k) - \bar{{\bf v}}(k) \|,
\end{split}
\end{equation}
where the inequality follows by the convexity of the norm.
But, $[{\bf W}(k)]$ is doubly stochastic, so we have 
$\bar{{\bf v}}(k) = \frac{1}{n} \sum_{i=1}^{n}{\bf v}_{i}(k) = \frac{1}{n}\sum_{i=1}^{n} \sum_{j=1}^{n}[{\bf W}(k)]_{ij}{{\bf x}_{j}(k)}=\frac{1}{n} \sum_{i=1}^{n}{\bf x}_{i}(k)$. 
\begin{equation}\label{3.88}
\begin{split}
\|{\bf v}_{i}(k)-\bar{{\bf v}}(k) \| & \leq \sum_{j=1}^{n}[{\bf W}(k)]_{ij} \| {\bf x}_{j}(k) - \bar{{\bf v}}(k) \|  \\
& \leq  \sum_{j=1}^{n}\| {\bf x}_{j}(k) - \bar{{\bf v}}(k) \|  \\ 
& \leq  \sum_{j=1}^{n}\| {\bf x}_{j}(k) - \frac{1}{n}\sum_{i=1}^{n}{\bf x}_{i}(k) \|,
\end{split}
\end{equation}
where in the first inequality we used convexity of the norm, in the second $ 0 \leq [{\bf W}(k)]_{ij} \leq 1$ and in the third we substituted the preceding result on $\bar{{\bf v}}(k)$.
Therefore, 
\begin{equation}\label{3.89}
\begin{split}
\|{\bf v}_{i}(k)-\bar{{\bf v}}(k) \| \leq & \sum_{j=1}^{n} \| {\bf x}_{j}(k) - \frac{1}{n} \sum_{i=1}^{n}{\bf x}_{i}(k) \|  \\ 
\leq & \sum_{j=1}^{n}\| \frac{1}{n} \sum_{i=1}^{n} {\bf x}_{j}(k) - \frac{1}{n} \sum_{i=1}^{n}{\bf x}_{i}(k) \|  \\ 
\leq & \frac{1}{n}  \sum_{j=1}^{n} \sum_{i=1}^{n}  \|  {\bf x}_{j}(k) -  {\bf x}_{i}(k) \|,
\end{split}
\end{equation}
where by the convexity of norm we have the last inequality.
Thus, we have 
\begin{equation*}\label{3.90}
\begin{split}
\alpha_{k}\|{\bf v}_{i}(k)-\bar{{\bf v}}(k) \| \leq & \frac{\alpha_{k}}{n}  \sum_{j=1}^{n} \sum_{i=1}^{n}  \|  {\bf x}_{j}(k) -  {\bf x}_{i}(k) \|.
\end{split}
\end{equation*}
Then
\begin{equation*}\label{3.91}
\begin{split}
\sum_{k=0}^{\infty} \alpha_{k}\|{\bf v}_{i}(k)-\bar{{\bf v}}(k) \| \leq & \sum_{k=0}^{\infty} \frac{\alpha_{k}}{n}  \sum_{j=1}^{n} \sum_{i=1}^{n}  \|  {\bf x}_{j}(k) -  {\bf x}_{i}(k) \| \\ 
\end{split}
\end{equation*}
\vspace{-0.3cm}
\begin{equation*}\label{3.92}
\begin{split}
\implies \sum_{k=0}^{\infty} \alpha_{k}\|{\bf v}_{i}(k)-\bar{{\bf v}}(k) \| \leq & \sum_{k=0}^{\infty} \frac{\alpha_{k}}{n}  \sum_{j=1}^{n} \sum_{i=1}^{n}  \|  {\bf x}_{j}(k) -  {\bf x}_{i}(k) \| \\ 
\leq &  \sum_{j=1}^{n} \sum_{i=1}^{n} \sum_{k=0}^{\infty} \frac{\alpha_{k}}{n}  \|  {\bf x}_{j}(k) -  {\bf x}_{i}(k) \|. 
\end{split}
\end{equation*}
The second inequality is valid since we can interchange infinite sum with finite sum.
But, from a previous result (Lemma~\ref{L3.10}) we have $ \sum_{k=0}^{\infty}\alpha_{k} \|{\bf x}_{i}(k) -{\bf x}_{j}(k) \| < \infty $ which through the preceding, 
$ \implies \sum_{k=0}^{\infty} \alpha_{k}\|{\bf v}_{i}(k)-\bar{{\bf v}}(k) \| < \infty $.
\\ But $\sum_{k=0}^{\infty} \alpha_{k} = \infty $ by choice. 
$\implies \|{\bf v}_{i}(k)-\bar{{\bf v}}(k) \| \to 0$, i.e., 
$\lim_{k \rightarrow \infty}{{\bf v}_{i}(k)} = \lim_{k \rightarrow \infty}{\bar{{\bf v}}(k)} $ if they exist.

Therefore, Lemma~\ref{L3.9} follows.
\end{proof}

\begin{lem}\label{L3.11}
Let Assumptions~\ref{A3.1} hold and Proposition ~\ref{sum_v_i-x^*_type2} or ~\ref{sum_v_i-x^*_type2a} satisfied. Let $ p < \frac{1}{\gamma_{min}} $ and let the step-size be such that $ \sum_{k=0}^{\infty} \alpha_{k} = \infty $ and $ \sum_{k=0}^{\infty} \alpha_{k} ^{2} < \infty $.
Let $ f ^{*} = min_{x \in \mathcal{X}} f({\bf x}) $ and $ \mathcal{X}^{*} = \{ x \in \mathcal{X} | f({\bf x}) = f ^{*} \} $.
Assume then that $ \mathcal{X} ^ {*} \neq  \Phi $. Then, the iterates $ \{ x(k) \} $ generated by SRDO algorithm (\ref{eq2}a)-(\ref{eq2}c) converge almost surely to the solution $ {\bf x} ^{*} \in \mathcal{X} ^{*} $, i.e.,
\begin{equation}\label{3.94}
\lim_{k \to \infty } {\bf x}_{i}(k) = {\bf x} ^{*} \ for \ all \ i \in V \ a.s. \\
\end{equation}
\end{lem}

\begin{proof}
We begin the proof of Lemma~\ref{L3.11} with the following.\\
For any ${\bf x}^{*} \in \mathcal{X}$, we have
\begin{equation}\label{3.95}
\begin{split}
\implies & \| \overline{{\bf x}_{i}}(k+1) - {\bf x}^{*} \| ^{2}\leq \\
 & \|\overline{{\bf x}_{i}}(k+1) - {\bf x}^{*} \| ^{2} + a \| \Pi_{\mathcal{X}}(\overline{{\bf x}_{i}}(k+1)) - {\bf x}^{*} \| ^{2}. 
\end{split}
\end{equation}
where $a > 0$.


And we have a similar inequality as (\ref{3.29})
\begin{equation}\label{3.96}
\begin{split}
\implies \| & \Pi_{\mathcal{X}} (\overline{{\bf x}_{i}}(k+1) - \overline{{\bf x}_{i}}(k+1) \| ^{2}\leq \\
 & \|\overline{{\bf x}_{i}}(k+1) - {\bf x}^{*} \| ^{2} - a \| \Pi_{\mathcal{X}}(\overline{{\bf x}_{i}}(k+1)) - {\bf x}^{*} \| ^{2}, 
\end{split}
\end{equation}
where $a \leq 1$. 

Using Lemma~\ref{L3.6} and \ref{L3.7} with the following  substitutions  $ y= \overline{{\bf x}_{i}}(k+1)$, $ x = {\bf v}_{i}(k) $, and $ {\bf x}^{*}=argmin f(x) \in \mathcal{Y}$ for $\mathcal{Y} =\mathcal{X}$, $\alpha = \alpha_{k} $ and $  \Phi = f $ and $ {\bf z}_{i}(k) =\Pi_{\mathcal{X}}({\bf v}_{i}(k)) \in \mathcal{X} $,equation (\ref{3.96}) where $\mathcal{Y} =\mathcal{X}$, we get the following equivalent of \eqref{3.33}

{ \begin{equation}\label{3.99}
\begin{split}
\| & \overline{{\bf x}_{i}} (k+1)- {\bf x}^{*} \| ^{2} \leq \\
 & (1 - a+(A_{\eta} -a A_{\tau})\alpha_{k} ^{2}) \|{\bf v}_{i}(k) -{\bf x}^{*}\| ^{2} \\
 & -2 (1-a) \alpha_{k} (f^{(i)}({\bf z}_{i}(k))-f^{(i)}({\bf x}^{*})) \\
 & +\frac{3a}{4} \| \Pi_{\mathcal{X}}[{\bf v}_{i}(k) - \alpha_{k} \nabla{f({\bf v}_{i}(k))}]-{\bf v}_{i}(k)\| ^{2} \\
 & +(\frac{3}{8 \eta}-\frac{3a}{8 \tau}+ (2-2a) \alpha_{k} L ) \|{\bf v}_{i}(k)- {\bf z}_{i}(k) \|^{2} \\ & +(B_{\eta}- a B_{\tau})\alpha_{k}^{2} \| \nabla{f^{(i)}({\bf x}^{*})}\| ^{2},
\end{split}
\end{equation}}
where $A_{\tau}=8L^{2}+16 \tau L^{2}$, $B_{\tau}=8 \tau + 8 $, $ \tau > 0 $ is arbitrary, $A_{\eta}=2L^{2}+16 \eta L^{2}$, $B_{\eta}=2 \eta + 8 $ and $ \eta > 0 $ is arbitrary. 

But we know that $ {\bf v}_{i}(k) = \sum_{j=1}^{n}[{\bf W}(k)]_{ij}{\bf x}_{j}(k) $ then by the convexity of the norm squared and double stochasticity of $[{\bf W}(k)]_{ij}$, we have by summing from $ i=1 $ to $ n $,
{ \begin{equation}\label{3.100}
\begin{split}
\sum_{i=1}^{n} \| {\bf v}_{i}(k) - {\bf x}^{*} \| ^{2} &  \leq \sum_{i=1}^{n} \sum_{j=1} ^{n} [{\bf W}(k)]_{ij} \| {\bf x}_{j}(k) - {\bf x}^{*} \| ^{2} \\
& \leq \sum_{i=1}^{n} \| {\bf x}_{j}(k) - {\bf x}^{*} \| ^{2}.
\end{split}
\end{equation}}

Following a similar analysis as in the Appendix and taking $ x^{i}= \min f^{(i)}(x) $ for each $ i \in \{1,\ldots,p\} $. Having for $ i \in I $ that $ f^{(i)}(x^{*})  > f^{(i)}(x^{i}) $ and for $ i \in I^{\complement} $ that $ f^{(i)}(x^{*}) \leq f^{(i)}(x^{i}) $ and for $\bar{{\bf z}}(k)=\frac{1}{n}\sum_{i=1}^{n}{\bf z}_{i}(k)$, where $ {\bf z}_{i}(k) =\Pi_{\mathcal{X}}({\bf v}_{i}(k)) \in \mathcal{X} $, we have


{ \begin{equation}\label{3.101}
\begin{split}
\sum_{(i)=1}^{p} & \gamma_{(i)} ( f^{(i)}({\bf z}_{i}(k))  - f^{(i)}( {\bf x}^{*})) \\
& \geq \sum_{(i)=1}^{p} \gamma_{(i)} ( f({\bf z}_{i}(k))- f^{(i)}( \bar{{\bf z}}(k))) + \gamma_{min} ( f(\bar{{\bf z}}(k))- f( {\bf x}^{*}) ) \\
& + \sum_{(i) \in I } \gamma_{(i)} \langle \nabla{f}^{(i)}({\bf x}^{*}), {\bf x}^{(i)}- {\bf x}^{*} \rangle.
\end{split}
\end{equation}}


But
\vspace{-0.9cm}
{ \begin{equation}\label{3.102}
\begin{split}
 \ \ \  \ \  \ \ \ \ \ \sum_{(i)=1}^{p} \gamma_{(i)} & ( f^{(i)}({\bf z}_{i}(k)) - f^{(i)}(\bar{{\bf z}}(k))) \\
& \geq \sum_{i=1}^{p} \gamma_{(i)} \langle \nabla{f^{(i)}(\bar{{\bf z}}(k))} , {\bf z}_{i}(k) - \bar{{\bf z}}(k) \rangle \\
& \geq - \sum_{(i)=1}^{p} \gamma_{(i)}\| \nabla{f^{(i)}(\bar{{\bf z}}(k))} \| \| {\bf z}_{i}(k) - \bar{{\bf z}}(k) \|.
\end{split}
\end{equation}}

Let $ {\bf z}_{i}(k) =\Pi_{\mathcal{X}}({\bf v}_{i}(k)) \in \mathcal{X} $.
Since $ \bar{{\bf z}}(k) $ is a convex combination of $ {\bf z}_{i} (k) \in \mathcal{X} \implies \bar{{\bf z}}(k) \in \mathcal{X} $ and thus $ \| \nabla{f^{(i)}(\bar{{\bf z}}(k))} \| \leq G_{f} $.
{\small \begin{equation}\label{3.103}
\begin{split}
\implies \sum_{(i)=1}^{p}\gamma_{(i)} ( f({\bf z}_{i}(k))- f( \bar{{\bf z}}(k))) \geq - G_{f} \sum_{(i)=1}^{p} \gamma_{(i)} \| {\bf z}_{i}(k) - \bar{{\bf z}}(k) \|. 
\end{split}
\end{equation}}

But
{ \begin{equation*}\label{3.104}
\begin{split}
&  \ \ \ \ \ \ \ \ \  \| {\bf z}_{i} (k)  - \bar{{\bf z}}(k) \|  =  \| \frac{1}{n} \sum_{j=1}^{n} ( {\bf z}_{i}(k) - {\bf z}_{j}(k) ) \|  \\
& \ \ \ \ \leq   \frac{1}{n} \sum_{j=1}^{n} \| {\bf z}_{i}(k) - {\bf z}_{j}(k)  \|  \leq   \frac{1}{n} \sum_{j=1}^{n} \|  \Pi_{\mathcal{X}}({\bf v}_{i}(k)) -  \Pi_{\mathcal{X}}({\bf v}_{j}(k)) \|   \\
& \ \ \ \ \leq   \frac{1}{n} \sum_{j=1}^{n} \| {\bf v}_{i}(k) - {\bf v}_{j}(k)  \|,
\end{split}
\end{equation*}}
where the first inequality follows from the convexity of the norm and the last inequality follows from the non-expansiveness of the projection $ \Pi $.

But by the triangle inequality, we have 
$ \| {\bf v}_{i}(k) - {\bf v}_{l}(k) \| \leq \| {\bf v}_{i}(k) - \bar{{\bf v}}(k) \| + \| {\bf v}_{l}(k) -\bar{{\bf v}}(k) \| $, \\
Thus,
\vspace{-0.5cm}
{ \begin{equation*}\label{3.105}
\begin{split}
\ \ & \ \ \  \ \ \ \  \ \|  {\bf z}_{i}  (k)  - \bar{{\bf z}}(k) \|  \leq   \frac{1}{n} \sum_{l=1}^{n} \| {\bf v}_{i}(k) - {\bf v}_{l}(k)  \|  \\
 & \leq   \frac{1}{n} \sum_{l=1}^{n} \| {\bf v}_{i}(k) - \bar{{\bf v}}(k) \| + \frac{1}{n} \sum_{l=1}^{n} \| {\bf v}_{l}(k) -\bar{{\bf v}}(k) \|  \\
 & \leq   \frac{n}{n} \| {\bf v}_{i}(k) - \bar{{\bf v}}(k) \| + \frac{1}{n} \sum_{l=1}^{n} \| {\bf v}_{l}(k) -\bar{{\bf v}}(k) \|  \\
& \implies \| {\bf z}_{i}(k)  - \bar{{\bf z}}(k) \| \leq \| {\bf v}_{i}(k) - \bar{{\bf v}}(k) \| + \frac{1}{n} \sum_{l=1}^{n} \| {\bf v}_{l}(k) -\bar{{\bf v}}(k) \|.
\end{split}
\end{equation*}}
Then  by summing over $i$, we get 
{ \begin{equation}\label{3.106}
\begin{split}
\sum_{i=1}^{n} \| {\bf z}_{i}(k)  - \bar{{\bf z}}(k) \| &  \leq \sum_{i=1}^{n} \| {\bf v}_{i}(k) - \bar{{\bf v}}(k) \| + \frac{1}{n} \sum_{i=1}^{n} \sum_{l=1}^{n} \| {\bf v}_{l}(k) -\bar{{\bf v}}(k) \|  \\ 
& \leq \sum_{i=1}^{n} \| {\bf v}_{i}(k) - \bar{{\bf v}}(k) \| + \frac{n}{n} \sum_{l=1}^{n} \| {\bf v}_{l}(k) -\bar{{\bf v}}(k) \|  \\ 
& \leq 2 \sum_{i=1}^{n} \| {\bf v}_{i}(k) - \bar{{\bf v}}(k) \|,
\end{split}
\end{equation}}
which follows since indices $i$ and $l$ are arbitrary indexes.

Then, by using (\ref{3.101}), substituting (\ref{3.103}) and (\ref{3.106}) and the Lipschitz continuity of the gradients accordingly, we have

{ \begin{equation}\label{3.107}
\begin{split}
\sum_{(i)=1}^{p} & \gamma_{(i)} ( f^{(i)}({\bf z}_{i}(k))  - f^{(i)}({\bf x}^{*}))  \geq \sum_{(i)=1}^{p} \gamma_{(i)} ( f^{(i)}({\bf z}_{i}(k))- f^{(i)}(\bar{{\bf z}}(k)))\\
& \ \ \ \ + \gamma_{min} ( f(\bar{{\bf z}}(k))- f( {\bf x}^{*}) ) + \sum_{(i) \in I } \gamma_{(i)} \langle \nabla{f}^{(i)}({\bf x}^{*}), {\bf x}^{(i)}- {\bf x}^{*} \rangle\\
& \geq - G_{f} \sum_{(i)=1}^{p} \gamma_{(i)}  \| {\bf z}_{i}(k) - \bar{{\bf z}}(k) \| + \gamma_{min} ( f(\bar{{\bf z}}(k))-  f( {\bf x}^{*}) ) \\
& - \sum_{(i) \in I} \gamma_{(i)} \| \nabla{f}^{(i)}({\bf x}^{*}) \| \| {\bf x}^{(i)} - {\bf x}^{*} \| \\
& \geq - 2 G_{f} \sum_{(i)=1}^{p} \gamma_{(i)} \| {\bf v}_{i}(k) - \bar{{\bf v}}(k) \| + \gamma_{min} ( f(\bar{{\bf z}}(k))- f( {\bf x}^{*}) ) \\
& - \sum_{(i) \in I} \gamma_{(i)} L \| {\bf x}^{(i)}- {\bf x}^{*} \|^{2}.
\end{split}
\end{equation}}

Then using (\ref{3.99})
and taking the expectation on the history $\mathcal{F}_{k}$ up to $k$, we have after the consideration the probability $ 0 \leq \gamma_{min} \leq \gamma_{(i)} \leq 1  $ of accessing a partition $ (i) \in \{1,\ldots,p\} $per each server, $ 0 < a < 1$  and $ \sum_{(i)=1}^{p}\leq \sum_{(i)=0}^{p}=1 $ that
{ \begin{equation}\label{3.110.2}
\begin{split}
\mathbb{E} [ \| & \overline{{\bf x}_{i}} (k+1)- {\bf x}^{*} \| ^{2} / \mathcal{F}_{k} ] \leq \\
 & \sum_{(i)=0}^{p} \gamma_{(i)} (1+(A_{\eta}-aA_{\tau})\alpha_{k}^{2}\sum_{j=1}^{n}[{\bf W}(k)]_{ij}\| x_{j}(k) - x^{*} \|^{2}  \\
 & + \sum_{(i)=0}^{p} \gamma_{(i)} (1-a) \| {\bf v}_{i}(k) - {\bf x}^{*} \| ^{2} \\
 & -2 \sum_{(i)=1}^{p} \gamma_{(i)} (1-a) \alpha_{k}  ( f^{(i)}({\bf z}_{i}(k))- f^{(i)}( {\bf x}^{*}))  \\
 & + \sum_{(i)=1}^{p} \gamma_{(i)}\frac{3a}{4} \| \Pi_{\mathcal{X}}[{\bf v}_{i}(k) - \alpha_{k} \nabla{f({\bf v}_{i}(k))}]-{\bf v}_{i}(k)\| ^{2} \\
 & + \sum_{(i)=1}^{p} \gamma_{(i)}(\frac{3}{8 \eta}-\frac{3a}{8 \tau}+ (2-2a) \alpha_{k} L ) \|{\bf v}_{i}(k)- {\bf z} \|^{2} \\ & + \sum_{(i)=1}^{p} \gamma_{(i)}(B_{\eta}- a B_{\tau})\alpha_{k}^{2} \| \nabla{f^{(i)}({\bf x}^{*})}\| ^{2} \ \ a.s.
\end{split}
\end{equation}}

But 
{\small \begin{equation}\label{3.111}
\begin{split}
 \| {\bf x}_{j}(k) - {\bf x} ^{*} \| & = \| \overline{{\bf x}_{j}}(k) + R_{j}(k-1) - {\bf x}^{*} \|.
\end{split} 
\end{equation}}
\\

And by squaring the norm and applying $2ab \leq a^{2} + b^{2} $, we have 
{\small \begin{equation}\label{3.112}
\begin{split}
 \| {\bf x}_{j}(k) - {\bf x} ^{*} \| ^{2} & \leq 2 \| \overline{{\bf x}_{j}}(k) - {\bf x}^{*} \| ^{2} + 2 \| R_{j}(k-1) \| ^{2}.
\end{split} 
\end{equation}}

Then through summing over all $i$ and having $ [{\bf W}]_{ij} $ doubly stochastic, the substitution of \eqref{3.107} and having  $ 0 \leq \gamma_{min} \leq \gamma_{(i)} \leq 1  $ for $ (i) \in \{1,\ldots,p\} $ and $ 0 < a < 1 $, \eqref{3.110.2} becomes

{\small \begin{equation}\label{3.113a}
\begin{split}
\mathbb{E} & [ \sum_{i=1}^{n} \|  \overline{{\bf x}_{i}} (k+1)- {\bf x}^{*} \| ^{2} / \mathcal{F}_{k} ] \leq \\
 &  \ (1 + (A_{\eta}-a A_{\tau})\alpha_{k} ^{2}) \sum_{i=1}^{n} \|\overline{{\bf x}_{i}}(k) -{\bf x}^{*}\| ^{2} \\
 + & \ (1 -  a) \sum_{i=1}^{n} \|{\bf v}_{i}(k)- {\bf x}^{*}\|^{2} \\
 & + 2(1-a)n \min(|I|\gamma_{max},1)\alpha_{k} L \max_{(i)} \|x^{i}-x^{*}\|^{2} \\
 & + 4  (1-a) \alpha_{k} G_{f}  \sum_{i=1}^{n}  \| {\bf v}_{i}(k) - \bar{{\bf v}}(k) \| \\
 & - 2 \gamma_{min} n (1-a) \alpha_{k}  ( f(\bar{{\bf z}}(k))- f( {\bf x}^{*})) + n \  (B_{\eta} - a B_{\tau} + 2)\alpha_{k}^{2} G_{f} ^{2} \\
& + \ (\frac{3}{8 \eta} - \frac{3 \gamma_{min} a p }{8 \tau} + 2 \alpha_{k}^{2} L^{2} + (2-2a) \alpha_{k} L ) \sum_{i=1}^{n} dist ^{2} ({\bf v}_{i}(k),\mathcal{X}) \\
& (1 + (A_{\eta}-a A_{\tau})\alpha_{k} ^{2})\| R_{j}(k-1) \| ^{2} \ \ a.s.
\end{split}
\end{equation}}


But we have
{ \begin{equation}\label{3.114a}
\begin{split}
 & (\frac{3}{8 \eta} - \frac{3 \gamma_{min}a p }{8 \tau} +  2 \alpha_{k}^{2} L^{2} + (2-2a) \alpha_{k} L ) \ dist ^{2} ({\bf v}_{i}(k),\mathcal{X}) \\
 & + ( 1 - a ) \sum_{i=1}^{n} \|{\bf v}_{i}(k)- {\bf x}^{*}\|^{2} \leq  \\   & (1-a +\frac{3}{8 \eta} - \frac{3 \gamma_{min}a p }{8 \tau} +  2 \alpha_{k}^{2} L^{2} + (2-2a) \alpha_{k} L ) \sum_{i=1}^{n} \|{\bf v}_{i}(k)- {\bf x}^{*}\|^{2} \leq \\
 & 2(1-a+\frac{3}{8 \eta} - \frac{3 \gamma_{min} a p }{8 \tau} + 2 \alpha_{k}^{2} L^{2} + (2-2a) \alpha_{k} L ) \sum_{i=1}^{n} \|{\bf v}_{i}(k)- {\bf x}^{(i)}\|^{2} \\
 & 2(1-a+\frac{3}{8 \eta} - \frac{3 \gamma_{min} a p }{8 \tau} + 2 \alpha_{k}^{2} L^{2} + (2-2a) \alpha_{k} L ) n \|{\bf x}^{(i)}- {\bf x}^{*} \|^{2}.
\end{split}
\end{equation}}
and

{\bf Applying the Supermartingale Theorem in Lemma~\ref{L3.11}:} Using the above \eqref{3.113a} reduces to

{ \begin{equation}\label{3.113a}
\begin{split}
\mathbb{E} & [ \sum_{i=1}^{n} \|  \overline{{\bf x}_{i}} (k+1)- {\bf x}^{*} \| ^{2} / \mathcal{F}_{k} ] \leq \\
 &  \ (1 + (A_{\eta}-a A_{\tau})\alpha_{k} ^{2}) \sum_{i=1}^{n} \|\overline{{\bf x}_{i}}(k) -{\bf x}^{*}\| ^{2} + 4  (1-a) \alpha_{k} G_{f}  \sum_{i=1}^{n}  \| {\bf v}_{i}(k) - \bar{{\bf v}}(k) \| \\
 & -2 \gamma_{min} (1-a) \alpha_{k}  ( f(\bar{{\bf z}}(k))- f( {\bf x}^{*})) + n \  (B_{\eta} - a B_{\tau} + 2)\alpha_{k}^{2} G_{f} ^{2} \\
& + \ 2(1-a+\frac{3}{8 \eta} - \frac{3 \gamma_{min} a p }{8 \tau} + 2 \alpha_{k}^{2} L^{2} + (2-2a) \alpha_{k} L ) \sum_{i=1}^{n} \min_{(i)} \|{\bf v}_{i}(k)- {\bf x}^{(i)}\|^{2} \\
& + \ 2 n (1-a+\frac{3}{8 \eta} - \frac{3 \gamma_{min} a p }{8 \tau} + 2 \alpha_{k}^{2} L^{2} \\
& \ \ \ \ \ \ \ \ \ \ + (1-a)(2+\min(|I|\gamma_{max},1))\alpha_{k}L)  \max_{(i)}\|{\bf x}^{(i)}- {\bf x}^{*} \|^{2} \\
& (1 + (A_{\eta}-a A_{\tau})\alpha_{k} ^{2})\| R_{j}(k-1) \| ^{2} \ \ a.s.
\end{split}
\end{equation}}

Since $ \| {\bf x}^{i}-{\bf x}^{*}\|^{2}$ is bounded, (i.e., Proposition ~\ref{sum_v_i-x^*_type2} or ~\ref{sum_v_i-x^*_type2a}), and $ \|{\bf v}_{i}(k) -{\bf x}^{*}\|^{2} $ is bounded then $ \|{\bf v}_{i}(k) - {\bf x}^{i}\|^{2}$ is bounded, we need $1 +\frac{3}{8\eta}+\epsilon < \frac{3\gamma_{min}ap}{8\tau}+a $. Then for $ \eta > \tau $ where $ \eta =l \tau $ and $ l > 1 $ we get $ a > \frac{8\eta +3}{l(3 \gamma_{min}p+8 \tau)} $. Introducing the same analysis as in Lemma 4 we have a sufficient condition for $ p < \frac{1}{\gamma_{min}} $, Thus, for e.g. for $ l=8 $, $ \gamma_{min} p =2 $, $ \eta = 4 $ and $ \tau =\frac{1}{2} $, we can take $ 1 > a > \frac{7}{18} $. Thus, by choosing $ a = \frac{1}{2} $ then we have for $ k > k_{2} $ that the terms containing $ \|{\bf v}_{i}(k) -{\bf x}^{i}\|^{2} $ and $ \| {\bf x}^{i} - {\bf x}^{8} \|^{2} $ to be negative.

And using the boundedness of $ R_{i}(k-1) $ from the Appendix (\ref{3.113a}) (i.e., Proposition ~\ref{sum_v_i-x^*_type2} or ~\ref{sum_v_i-x^*_type2a}). becomes

{ \begin{equation}\label{3.115}
\begin{split}
& \mathbb{E} [ \sum_{i=1}^{n}  \|  \overline{{\bf x}_{i}}  (k+1)- {\bf x}^{*} \| ^{2} / \mathcal{F}_{k} ] \leq  \ (1 + (A_{\eta}-a A_{\tau})\alpha_{k} ^{2}) \sum_{i=1}^{n} \|\overline{{\bf x}_{i}}(k) -{\bf x}^{*}\| ^{2} \\
 & + (1 +  (A_{\eta}-a A_{\tau}) \alpha_{k} ^{2})  \alpha_{k-1} ^{2} n F 
+ 4  (1-a) \alpha_{k} G_{f} \sum_{i=1}^{n}  \| {\bf v}_{i}(k) - \bar{{\bf v}}(k) \| \\
 & - 2  \gamma_{min} (1-a) \alpha_{k} ( f(\bar{{\bf z}}(k))- f( {\bf x}^{*})) \\
 & + n \ (B_{\eta} - a B_{\tau} + 2 )\alpha_{k}^{2} G_{f} ^{2} \ \ a.s.
\end{split}
\end{equation}}

But $ \sum_{k=0}^{\infty} ( A_\eta- a A_{\tau}) \alpha_{k}^{2} < \infty $ since $ \sum_{k=0}^{\infty}\alpha_{k}^{2} < \infty $.

Also, $ \sum_{k=0}^{\infty}n(B_{\eta}-a B_{\tau})\alpha_{k}^{2}G_{f}^{2} < \infty $ similarly. \\
And, $ \sum_{k=0}^{\infty} 4  (1-a) \alpha_{k} G_{f} \sum_{i=1}^{n}  \| {\bf v}_{i}(k) - \bar{{\bf v}}(k) \| < \infty $ from Lemma~6 holding.
Therefore, the supermartingale theorem applies. Hence, the sequence $\{ \| \overline{{\bf x}}_{i}(k)-{\bf x}^{*} \|^{2} \}$ is convergent a.s. to a nonnegative random variable for any $ i \in V $ and $ {\bf x} ^{*} \in \mathcal{X}^{*} $ where $\mathcal{X}^{*} = \{ x \in \mathcal{X} | f({\bf x}) = \min_{x \in \mathcal{X}}f({\bf x}) \} $. But $ u_{k} = (1-a) \gamma_{min}\alpha_{k}(f(\bar{{\bf z}}(k))- f( {\bf x}^{*})) >0 $ with negative coefficient $ -2 < - 1$ since $ 0 < a = \frac{1}{2 \gamma_{min}} < 1 $ and $ f({\bf x}^{*}) = \min f(x) $.
And the theorem also implies that $\sum_{k=0}^{\infty} u_{k} = \sum_{k=0}^{\infty} (1-a)\gamma_{min} \alpha_{k}  ( f(\bar{{\bf z}}(k))- f( {\bf x}^{*})) \leq \infty $. This with the condition that, $\sum_{k=0}^{\infty}\alpha_{k} = \infty $ and $ f(\bar{{\bf z}}(k))- f( {\bf x}^{*})) \geq  0 $ imply

{ \begin{equation}\label{3.117}
\begin{split}
\lim_{k \to \infty} \inf( f(\bar{{\bf z}}(k))- f({\bf x}^{*})) = 0 \ \
a.s. 
\end{split}
\end{equation}}
And since $ f(\bar{{\bf z}}(k))- f({\bf x}^{*}) \geq 0 $ for all $ k $ since $ f({\bf x}^{*})=\min f({\bf x}) $ then
{ \begin{equation}\label{3.117a}
\begin{split}
\lim_{k \to \infty} f(\bar{{\bf z}}(k))= f({\bf x}^{*})   \ \
a.s. 
\end{split}
\end{equation}}


By lemma~\ref{L3.5}, we have
$ \sum_{k=0}^{\infty} dist^{2}({\bf v}_{i}(k),\mathcal{X}) < \infty $,
i.e., $ \sum_{k=0}^{\infty} \| {\bf v}_{i}(k) -{\bf z}_{i}(k) \| ^{2} < \infty $ where $ {\bf z}_{i}(k) = \Pi_{\mathcal{X}}({\bf v}_{i}(k)) $.

\begin{equation}\label{3.118}
\begin{split}
\implies \lim_{k \to \infty}\| {\bf v}_{i}(k) -{\bf z}_{i}(k) \| \to 0.
\end{split}
\end{equation}

But we have the sequence $ \{ \|{\bf v}_{i}(k) - {\bf x}^{*} \| \} $ is also convergent a.s. for all $ i \in V $ and $ {\bf x} ^{*} \in \mathcal{X}^{*} $.
By (\ref{3.118}) it follows that $ \{ \| {\bf z}_{i}(k) - {\bf x}^{*} \| \} $ is also convergent a.s for all $ i \in V $ and $ {\bf x} ^{*} \in \mathcal{X}^{*} $. But since $ \| \bar{{\bf v}}(k) - {\bf x} ^{*} \| \leq \frac{1}{n} \sum_{i=1}^{n} \| {\bf v}_{i}(k) - {\bf x} ^{*} \| $ and the sequence $ \{ \| {\bf v}_{i}(k) - {\bf x} ^{*} \| \} $ is convergent a.s for all $ i \in V $ and $ {\bf x} ^{*} \in \mathcal{X}^{*} $, it follows that $ \{ \| \bar{{\bf v}}(k) - {\bf x} ^{*} \| \} $ is convergent a.s for all $ {\bf x} ^{*} \in \mathcal{X}^{*} $. Using a similar argument, we can conclude that $ \{ \| \bar{{\bf z}}(k) - {\bf x} ^{*} \| \} $ is convergent a.s for all $ {\bf x} ^{*} \in \mathcal{X} ^{*} $. As a particular consequence, it follows that the sequences $ \{ \bar{{\bf v}}(k) \} $ and $ \{ \bar{{\bf z}}(k) \} $ are a.s. bounded and, hence they have accumulation points.

From (\ref{3.117a}) and the continuity of $f$, it follows that the sequence $ \bar{{\bf z}}(k) $  must have one accumulation point in the set $ \mathcal{X} ^{*} $ a.s.
This and the fact that $ \{ \| \bar{{\bf z}}(k) - {\bf x} ^{*} \| \} $ is convergent a.s for every $ {\bf x} ^{*} \in \mathcal{X} ^{*} $ imply that for a random point $ {\bf x} ^{*} \in \mathcal{X}^{*} $ (from (\ref{3.117})) \\
\begin{equation}\label{3.120}
\begin{split}
\lim_{k \to \infty} \bar{{\bf z}}(k) = {\bf x} ^{*} \ \ a.s.
\end{split}
\end{equation}

Now, from $ \bar{{\bf z}}(k) = \frac{1}{n} \sum_{l=1}^{n} z_{l}(k) $ and $ \bar{{\bf v}}(k) = \frac{1}{n} \sum_{l=1}^{n} {\bf v}_{l}(k) $ and using (\ref{3.118}) ($ \lim_{k \to \infty} \| {\bf v}_{i}(k) - {\bf z}_{i}(k) \| = 0 $ for all $ i \in V $)  and the convexity of the norm, we obtain that
\begin{equation}\label{3.121}
\begin{split}
\lim_{k \to \infty} \| \bar{{\bf v}}(k) - \bar{{\bf z}}(k) \| \leq  \frac{1}{n}  \sum_{l=1}^{n} \| {\bf v}_{l}(k) - z_{l}(k) \| = 0  \ \ a.s. 
\end{split}
\end{equation}

In view of (\ref{3.120}), it follows that 
\begin{equation}\label{3.122}
\begin{split}
\lim_{k \to \infty} \bar{{\bf v}}(k) = {\bf x} ^{*} \ \  a.s. 
\end{split}
\end{equation}
By $ \sum_{k=0}^{\infty} \alpha_{k} \| {\bf v}_{i}(k) - \bar{{\bf v}}(k) \| < \infty $ for all $ i \in V $ in Lemma~\ref{L3.9}, since $ \sum_{k=0}^{\infty} \alpha_{k} = \infty $, we have $ \lim_{k \to \infty} \inf{\|{\bf v}_{i}(k)-\bar{{\bf v}}(k) \|} = 0 $ for all $ i \in V $ a.s.\\
This fact or the fact that $ \{ \| {\bf v}_{i}(k) - {\bf x}^{*} \| \} $ is convergent a.s. for all $ i \in V $ together with the above limit equality (\ref{3.122}), a consequence of (\ref{3.120}), imply that
\begin{equation}\label{3.123}
\begin{split}
\lim_{k \to \infty} \| {\bf v}_{i}(k) - {\bf x}^{*} \| = 0, 
\end{split}
\end{equation}
for all $ i \in V $ a.s.

Finally, from $ \sum_{k=0}^{\infty} \|\epsilon_{i}(k)\|^{2} < \infty $ for all $ i \in V $ in Lemma 5 (i.e., Lemma 7 in \citep{lee2013distributed}), we thus have
\begin{equation}\label{3.124}
\begin{split}
\lim_{k \to \infty} \| {\bf x}_{i}(k+1) - {\bf v}_{i}(k) \| = 0,\
\end{split}
\end{equation}
for all $ i \in V $ a.s. where ${\bf \epsilon}_{i}(k) = {\bf x}_{i}(k+1) - {\bf v}_{i}(k) $, i.e., $ \sum_{k=0}^{\infty} \|{\bf \epsilon}_{i}(k) \| ^{2} = \sum_{k=0}^{\infty} \| {\bf x}_{i}(k+1) - {\bf v}_{i}(k) \| < \infty $. This implies that $ \lim_{k \to \infty} \| {\bf x}_{i}(k+1) - {\bf v}_{i}(k) \| = 0 $ which together with (\ref{3.123}) (i.e., (\ref{3.124}) and (\ref{3.123})) imply that

{ \begin{equation}\label{3.125}
\begin{split}
\lim_{k \to \infty} {\bf x}_{i}(k) = {\bf x} ^{*}, 
\end{split}
\end{equation}}
for all $ i \in V $ a.s. \\

Hence, the result of Lemma~\ref{L3.11} follows. \qed
\end{proof}



\section{Convergence Rate}

We are going to show more explicitly with details how inequality \eqref{ineq_general_error_bound} in the case where at least for one $(i)$ we have $ f^{(i)}({\bf x}^{(i)}) < f^{(i)}({\bf x}^{*}) $ in the case of strongly convex $ f^{(i)} $ for all $ (i) $.
In this case which corresponds to Lemma~\ref{type2_error_bound} and before reaching the explicit result of \eqref{ineq_type2_error_bound_subst} we are going to use \eqref{type2_error_bound_no_subst}.
Now, 
\begin{equation}\label{result1}
\begin{split}
2 \alpha_{k} \sum_{j=1}^{n} & \sum_{(i) \in I} \gamma_{(i)} (f^{(i)}({\bf x}^{*}) - f^{(i)}({\bf x}^{(i)}))  \\
& \leq 2 \alpha_{k} \sum_{j=1}^{n}\sum_{(i) \in I} \gamma_{(i)}\langle \nabla{f}^{(i)}({\bf x}^{*}) - \nabla{f}^{(i)}({\bf x}^{(i)}), {\bf x}^{*} - {\bf x}^{(i)} \rangle \\
& \leq 2 \alpha_{k} \sum_{j=1}^{n}\sum_{(i) \in I} \gamma_{(i)} L \| {\bf x}^{(i)} - {\bf x}^{*} \|^{2} \\ 
& \leq 2 n \alpha_{k} L \min(|I|\gamma_{max},1) \max_{(i)} \| {\bf x}^{(i)} - {\bf x}^{*} \|^{2} \\
\end{split}
\end{equation}

 And we have from the strong convexity of $ f^{(i)} $ that
\begin{equation}\label{strong_conv1}
\begin{split}
 f^{(i)} & ({\bf v}_{j}(k)) \\
& \geq f^{(i)}({\bf x}^{(i)}) +\langle \nabla{f}^{(i)}({\bf x}^{(i)}),{\bf v}_{j}(k)-{\bf x}^{(i)}\rangle + \frac{\sigma_{(i)}}{2}\|{\bf x}^{(i)} - {\bf v}_{j}(k) \|^{2}
\end{split}
\end{equation}

Thus, \eqref{strong_conv1} becomes
\begin{equation}\label{result2}
\begin{split}
f^{(i)} &({\bf x}^{(i)})  - f^{(i)}({\bf v}_{j}(k)) \\
& \leq - \frac{\sigma_{(i)}}{2}\|{\bf x}^{(i)} - {\bf v}_{j}(k) \|^{2}
\end{split}
\end{equation}

And since
\begin{equation}
\begin{split}
\| {\bf x}^{(i)} - {\bf v}_{j}(k) \| \geq \| {\bf x}^{(i)} - {\bf x}^{*} \| - \| {\bf x}^{*} - {\bf v}_{j}(k) \|
\end{split}
\end{equation}
then 
\begin{equation}
\begin{split}
\| {\bf x}^{(i)} - {\bf v}_{j}(k) \|^{2} & \geq \| {\bf x}^{(i)} - {\bf x}^{*} \|^{2} + \| {\bf x}^{*} - {\bf v}_{j}(k) \|^{2}  \\
& - 2 \| {\bf x}^{(i)} - {\bf x}^{*} \| \| {\bf x}^{*} - {\bf v}_{j}(k) \|
\end{split}
\end{equation}
Thus,
\begin{equation}
\begin{split}
f^{(i)} &({\bf x}^{(i)})  - f^{(i)}({\bf v}_{j}(k)) \\
& \leq - \frac{\sigma_{(i)}}{2} \| {\bf x}^{(i)} - {\bf v}_{j}(k) \|^{2} \\
& \leq - \frac{\sigma_{(i)}}{2} \| {\bf x}^{(i)} - {\bf x}^{*} \|^{2} - \frac{\sigma_{(i)}}{2} \| {\bf x}^{*} - {\bf v}_{j}(k) \|^{2} \\
& + \sigma_{(i)} \| {\bf x}^{(i)} - {\bf x}^{*} \| \| {\bf x}^{*} - {\bf v}_{j}(k) \| 
\end{split}
\end{equation}

Then from convergence lemma (i.e., Lemma~\ref{L3.11}) since $ \| {\bf x}^{*} - {\bf v}_{j}(k) \|^{2} \leq \| {\bf x}^{(i)} - {\bf x}^{*} \|^{2} $ for $ k \geq k_{c} $ we have
\begin{equation}
\begin{split}
f^{(i)} &({\bf x}^{(i)})  - f^{(i)}({\bf v}_{j}(k)) \\
 & \leq  \frac{\sigma_{(i)}}{2} \| {\bf x}^{(i)} - {\bf x}^{*} \|^{2} - \frac{\sigma_{(i)}}{2} \| {\bf x}^{*} - {\bf v}_{j}(k) \|^{2} 
\end{split}
\end{equation}
Therefore,
\begin{equation}\label{result3}
\begin{split}
 & \sum_{j=1}^{n} \sum_{(i)=1}^{p} \gamma_{(i)} (f^{(i)} ({\bf x}^{(i)})  - f^{(i)}({\bf v}_{j}(k))) \\
 & \leq  p \gamma_{min} n (\frac{\sigma_{max}}{2} \max_{(i)} \| {\bf x}^{(i)} - {\bf x}^{*} \|^{2} - \frac{\sigma_{min}}{2} \| {\bf x}^{*} - {\bf v}_{j}(k) \|^{2}) 
\end{split}
\end{equation}
But we have from quadratic inequality due to convexity of $ f^{(i)} $ that
{ \begin{equation}\label{strong_conv1}
\begin{split}
 f^{(i)} ({\bf v}_{j}(k)) \leq f^{(i)}({\bf x}^{(i)}) +\langle \nabla{f}^{(i)}({\bf x}^{(i)}),{\bf v}_{j}(k)-{\bf x}^{(i)}\rangle + \frac{L}{2}\|{\bf x}^{(i)} - {\bf v}_{j}(k) \|^{2}
\end{split}
\end{equation}}
then
\begin{equation}
\begin{split}
f^{(i)} & ({\bf v}_{j}(k)) - f^{(i)} ({\bf x}^{(i)}) \\
 & \leq \frac{L}{2} \| {\bf v}_{j}(k) - {\bf x}^{(i)} \| ^{2} \\
& \leq L \| {\bf x}^{(i)} - {\bf x}^{*} \|^{2} +  L \| {\bf x}^{*} - {\bf v}_{j}(k) \|^{2} 
\end{split}
\end{equation}

then
\begin{equation}\label{result4}
\begin{split}
& \sum_{j=1}^{n} \sum_{(i)=1}^{p} \gamma_{(i)} ( f^{(i)}({\bf v}_{j}(k)) - f^{(i)} ({\bf x}^{(i)}) ) \\
 & \leq   L n \min (p \gamma_{max}, 1) ( \| {\bf x}^{(i)} - {\bf x}^{*} \|^{2} +  \| {\bf x}^{*} - {\bf v}_{j}(k) \|^{2})
\end{split}
\end{equation}

Using \eqref{ineq_type2_error_most_explicit_strong_convex} and evaluating the fourth term by \eqref{result1} and using \eqref{result3} to evaluate the first part of the fifth term and \eqref{result4} to evaluate the second part we get

{\small \begin{equation}\label{ineq_type2_error_most_explicit_strong_convex}
\begin{split}
    & \sum_{l=1}^{n}\|  {\bf v}_{l}(k+1) - {\bf x}^{*}\|^{2}  \leq ( 1 - \mu - \alpha_{k} p \gamma_{min} n \sigma_{min} \\
    & +\frac{2L n \min (p \gamma_{max}, 1) \alpha_{k}^{2}}{(\sum_{(i)=1}^{p}\gamma_{(i)})^{2}} ) \sum_{j=1}^{n}\|{\bf v}_{j}(k)- {\bf x}^{*}\|^{2}  \\
   & + 2 \alpha_{k}\sum_{j=1}^{n}\sum_{(i)=1}^{p}\gamma_{(i)} \langle {\bf \epsilon}_{j,(i)}(k), {\bf v}_{j}(k) - {\bf x}^{*} \rangle  \\
   & + 2 \alpha_{k}^{2} \frac{\sum_{j=1}^{n}\sum_{(i)=1}^{p}\gamma_{(i)}}{(\sum_{(i)=1}^{p}\gamma_{(i)})^{2}}\|{\bf \epsilon}_{j,(i)}(k)\|^{2}\\
   &  + (2 n \alpha_{k} L \min(|I|\gamma_{max},1) + \alpha_{k} p \gamma_{min} n  \sigma_{max} \\ 
   & + \frac{2L n \min (p \gamma_{max}, 1) \alpha_{k}^{2}}{(\sum_{(i)=1}^{p}\gamma_{(i)})^{2}} ) \max_{(i)} \| {\bf x}^{(i)} - {\bf x}^{*} \|^{2}
\end{split}
\end{equation}}

evaluate $ {\bf \epsilon}_{j,(i)}(k) $ using Definition 1 and \eqref{3.67a} to get 

{ \begin{equation}\label{ineq_type2_error_most_explicit_strong_convex}
\begin{split}
    & \sum_{l=1}^{n}\|  {\bf v}_{l}(k+1) - {\bf x}^{*}\|^{2}  \leq (1 - \mu - \alpha_{k} p \gamma_{min} n \sigma_{min} \\
    & + \frac{2L n \min (p \gamma_{max}, 1) \alpha_{k}^{2}}{(\sum_{(i)=1}^{p}\gamma_{(i)})^{2}} )  \sum_{j=1}^{n}\|{\bf v}_{j}(k)- {\bf x}^{*}\|^{2}  \\
   & + 4 L \alpha_{k}\sum_{j=1}^{n} \| {\bf A}^{(i)} \| _{\infty} \|{\bf B}^{(i)}\| _{2,\infty} \max_{k-H \leq \hat{k} \leq k ; q \in V} \| {\bf v}_{q}(\hat{k}) -  x^{*}\| ^{2}   \\
   & + 8 L ^{2} \alpha_{k}^{2} \frac{\sum_{j=1}^{n} \| {\bf A}^{(i)} \| ^{2} _{\infty} \|{\bf B}^{(i)}\| ^{2} _{2,\infty} \max_{k-H \leq \hat{k} \leq k ; q \in V} \| {\bf v}_{q}(\hat{k}) -  x^{*}\| ^{2} }{(\sum_{(i)=1}^{p}\gamma_{(i)})^{2}} \\
   &  + (2 n \alpha_{k} L \min(|I|\gamma_{max},1) + \alpha_{k} p \gamma_{min} n  \sigma_{max} \\ 
   & +\frac{2L n \min (p \gamma_{max}, 1) \alpha_{k}^{2}}{(\sum_{(i)=1}^{p}\gamma_{(i)})^{2}} ) \max_{(i)} \| {\bf x}^{(i)} - {\bf x}^{*} \|^{2}
\end{split}
\end{equation}}
Then in order to use Lemma Martingale 2 at least one of these three inequalities must be satisfied:
\begin{itemize}
    \item $\frac{2L n \min (p \gamma_{max}, 1) \alpha_{k}^{2}}{(\sum_{(i)=1}^{p}\gamma_{(i)})^{2}} \leq \alpha_{k} p \gamma_{min} n \sigma_{min} $
    \item $\frac{2L n \min (p \gamma_{max}, 1) \alpha_{k}^{2}}{(\sum_{(i)=1}^{p}\gamma_{(i)})^{2}} \leq \mu $
    \item  $\frac{2L n \min (p \gamma_{max}, 1) \alpha_{k}^{2}}{(\sum_{(i)=1}^{p}\gamma_{(i)})^{2}} \leq \mu + \alpha_{k} p \gamma_{min} n \sigma_{min}$ 
\end{itemize}
Thus, one way to tackle this problem is to adjust a lower bound on $ \mu $ by adequately choosing the described row stochastic matrices $ {\bf W}(k) $.
But fortunately enough due to the strong convexity of the function and the existence of a lower bound on the monotonicity of the gradient of $ f^{(i)} $ which is $ \sigma_{(i)} $ the first inequality is satisfied for any decreasing sequence $ \alpha_{k} $. Thus, eventually for $ k \geq k_{0} $ an inequality in the form \eqref{eqw1} is valid. And having it satisfied until the at least the first $ B $ iterations after $ k_{0} $ will result in applying the Martingale 2 result and thus bounding $ \sum_{j=1}^{n} \| {\bf v}_{j}(k) - {\bf x}^{*} \| ^{2} \leq \eta $ where

{ 
\begin{equation}\label{eta_explicit_strongly_convex_type2}
\begin{split}
\eta = \frac{ ( 2 n \alpha_{k} L \min(|I|\gamma_{max},1) + \alpha_{k} p \gamma_{min} n \sigma_{max} + \frac{2L n \min (p \gamma_{max}, 1) \alpha_{k}^{2}}{(\sum_{(i)=1}^{p}\gamma_{(i)})^{2}} ) \max_{(i)} \| {\bf x}^{(i)} - {\bf x}^{*} \|^{2} }{ \mu - \frac{4 L \alpha_{k} \| {\bf A}^{(i)} \| _{\infty} \|{\bf B}^{(i)}\| _{2,\infty} ( 1 + 2 L \alpha_{k}  \| {\bf A}^{(i)} \|  _{\infty} \|{\bf B}^{(i)}\|  _{2,\infty} )} {(\sum_{(i)=1}^{p}\gamma_{(i)})^{2}}}  
\end{split}
\end{equation}}

Then using \eqref{3.115} from the third part of the convergence analysis and the strong convexity of $ f $ resulting from the strong convexity of $ f^{(i)} $ we get
{ \begin{equation}
\begin{split}
& \mathbb{E} [ \sum_{i=1}^{n}  \|  \overline{{\bf x}_{i}}  (k+1)- {\bf x}^{*} \| ^{2} / \mathcal{F}_{k} ] \leq  \ (1 + (A_{\eta}-a A_{\tau})\alpha_{k} ^{2}) \sum_{i=1}^{n} \|\overline{{\bf x}_{i}}(k) -{\bf x}^{*}\| ^{2} \\
 & + (1 +  (A_{\eta}-a A_{\tau}) \alpha_{k} ^{2})  \alpha_{k-1} ^{2} n F 
+ 4  (1-a) \alpha_{k} G_{f} \sum_{i=1}^{n}  \| {\bf v}_{i}(k) - \bar{{\bf v}}(k) \| \\
 & - \sigma  \gamma_{min} (1-a) \alpha_{k} \| \bar{{\bf z}(k)} -  {\bf x}^{*} \| ^{2} \\
 & + n \ (B_{\eta} - a B_{\tau} + 2 )\alpha_{k}^{2} G_{f} ^{2} 
\end{split}
\end{equation}}
But since $ \|\Pi_{\mathcal{X}}({\bf v}_{i}(k))- x^{*}\|  \geq \| {\bf v}_{i}(k) - {\bf x}^{*} \| - \|\Pi_{\mathcal{X}}({\bf v}_{i}(k))- {\bf v}_{i}(k) \| $, $ \|\Pi_{\mathcal{X}}({\bf v}_{i}(k))- {\bf v}_{i}(k) \| \leq \| {\bf v}_{i}(k) - {\bf x}^{*} \| $ and the convexity of the norm we get
{ \begin{equation}
\begin{split}
& \mathbb{E} [ \sum_{i=1}^{n}  \|  \overline{{\bf x}_{i}}  (k+1)- {\bf x}^{*} \| ^{2} / \mathcal{F}_{k} ] \leq  \ (1 + (A_{\eta}-a A_{\tau})\alpha_{k} ^{2}) \sum_{i=1}^{n} \|\overline{{\bf x}_{i}}(k) -{\bf x}^{*}\| ^{2} \\
 & + (1 +  (A_{\eta}-a A_{\tau}) \alpha_{k} ^{2})  \alpha_{k-1} ^{2} n F 
+ 16  (1-a) \alpha_{k} G_{f} \sum_{i=1}^{n}  \| {\bf v}_{i}(k) - {\bf x}^{*} \| ^{2} \\
 & - \frac{\sigma}{n}  \gamma_{min} (1-a) \alpha_{k} \sum_{i=1}^{n} \|\Pi_{\mathcal{X}}({\bf v}_{i}(k))- {\bf v}_{i}(k) \|  ^{2} \\
 & + \frac{\sigma}{n}  \gamma_{min} (1-a) \alpha_{k} \sum_{i=1}^{n} \| {\bf v}_{i}(k) - {\bf x}^{*} \|  ^{2} \\
 & + n \ (B_{\eta} - a B_{\tau} + 2 )\alpha_{k}^{2} G_{f} ^{2} 
\end{split}
\end{equation}}

Using $ \| \bar{{\bf x}_{i}(k)} - {\bf x}^{*} \|^{2} = \| {\bf v}_{i}(k-1) - {\bf x}^{*} \|^{2} + 2\alpha_{k} \langle \nabla{f}^{(i)}({\bf v}_{i}(k-1)), {\bf x}^{*} - {\bf v}_{i}(k-1) \rangle $ and $ f({\bf x}^{*})- f({\bf v}_{i}(k-1)) \leq - \sigma \alpha_{k} \| {\bf v}_{i}(k-1) - {\bf x}^{*} \|^{2} $, we get
{ \begin{equation}
\begin{split}
 & \mathbb{E} [ \sum_{i=1}^{n}  \|  \overline{{\bf x}_{i}}  (k+1)- {\bf x}^{*} \| ^{2} / \mathcal{F}_{k} ] \leq  \\
 & (1+( 2 -\sigma)\alpha_{k} )(1 +  (A_{\eta}-a A_{\tau}) \alpha_{k} ^{2}) \sum_{i=1}^{n} \| {\bf v}_{i}(k-1) - {\bf x}^{*} \|^{2} \\
 & + (1 +  (A_{\eta}-a A_{\tau}) \alpha_{k} ^{2})  \alpha_{k-1} ^{2} n F 
+ 16  (1-a) \alpha_{k} G_{f} \sum_{i=1}^{n}  \| {\bf v}_{i}(k) - {\bf x}^{*} \| ^{2} \\
 & - \frac{\sigma}{n}  \gamma_{min} (1-a) \alpha_{k} \sum_{i=1}^{n} \|\Pi_{\mathcal{X}}({\bf v}_{i}(k))- {\bf v}_{i}(k) \|  ^{2} \\
 & + \frac{\sigma}{n}  \gamma_{min} (1-a) \alpha_{k} \sum_{i=1}^{n} \| {\bf v}_{i}(k) - {\bf x}^{*} \|  ^{2} \\
 & + n \ (B_{\eta} - a B_{\tau} + 2 )\alpha_{k}^{2} G_{f} ^{2} 
\end{split}
\end{equation}}
But having $ \sum_{i=1}^{n} \| {\bf v}_{i}(k) - {\bf x}^{*} \|^{2} \leq \eta $, we get
{ \begin{equation}\label{convergence_rate_strongly_convex_Martingale_2}
\begin{split}
 & \mathbb{E} [ \sum_{i=1}^{n}  \|  \overline{{\bf x}_{i}}  (k)- {\bf x}^{*} \| ^{2}  ] \leq  \\
 & [(1 +  (A_{\eta}-a A_{\tau}) \alpha_{k} ^{2}) (1+( 2 -\sigma)\alpha_{k} + 4 \alpha_{k}^{2} n  L^{2} \max_{(i)} \|{\bf A}^{(i)} \| ^{2} _{\infty} \|{\bf B}^{(i)}\| _{2,\infty} ^{2} ) \\
& + 16  (1-a) \alpha_{k} G_{f} + \frac{\sigma}{n}  \gamma_{min} (1-a) \alpha_{k}]\eta + n \ (B_{\eta} - a B_{\tau} + 2 )\alpha_{k}^{2} G_{f} ^{2} 
\end{split}
\end{equation}}
which is the expected convergence rate of a strongly convex function formed of $ p $ strongly convex functions with $ f^{(i)}({\bf x}^{*})>f^{(i)}({\bf x}^{i}) $ for at least one $ (i) $
\subsection{Convergence Rate for Strongly Convex Function with $ f^{(i)}({\bf x}^{*})=f^{(i)}({\bf x}^{i}) $ for all $ (i) $}

The convergence rate of a strongly convex function formed of $ p $ strongly convex functions  with $ f^{(i)}({\bf x}^{*})=f^{(i)}({\bf x}^{i}) $ for all $ (i) $ can be deduced by applying Lemma~\ref{type1_error_bound} and is 
{ \begin{equation}\label{convergence_rate_strongly_convex_Martingale_1}
\begin{split}
 & \mathbb{E} [ \sum_{i=1}^{n}  \| {\bf v}_{i}  (k)- {\bf x}^{*} \| ^{2}  ] \leq \\
 & ( 1 - \mu + \frac{4 L \alpha_{k} \| {\bf A}^{(i)} \| _{\infty} \|{\bf B}^{(i)}\| _{2,\infty} ( 1 + 2 L \alpha_{k}  \| {\bf A}^{(i)} \|  _{\infty} \|{\bf B}^{(i)}\|  _{2,\infty} )} {(\sum_{(i)=1}^{p}\gamma_{(i)})^{2}})^{k}V_{0}  
\end{split}
\end{equation}}
for $ k \geq k_{0} $.

ection{Numerical Simulation}

Our aim in the simulation is to verify the convergence of the proposed algorithm, while showing its convergence rate for different network topologies.

In this section, we restrict the optimization problem to the following unconstrained convex optimization problem on a network  
\begin{equation}\label{objective.simulation}
 \arg\min_{x \in \mathbf{R}^{N}} \|  {\bf G}x- y \|_{2}^{2},\end{equation}
where the network contains $ n $ nodes, $ G $ is a random  matrix of size $M\times N$ whose entries are independent and identically distributed standard normal random variables, and
\begin{equation} y= {\bf G} x_o\in \mathbb{R}^{M}\end{equation}
has entries of ${\bf x}_o$ that are identically independent random variables sampled from the uniform bounded random distribution  between $ -1 $ and $ 1 $.
The solution ${\bf x}^{*}$ of the optimization problem above  is the least squares solution of the overdetermined system  
$ y  = {\bf G} x_o, \ x_o\in \mathbb{R}^N$. 
We demonstrate the  performance of SRDO to solve the convex optimization problem \eqref{eq1}, and match it with the calculated convergence rates.

Assume that the network has $s$ as the maximum number of allowed stragglers. Then we can repartition the network
around those $ m $ worker nodes, and accordingly,
the random measurement matrix $ {\bf G} $,  the measurement data $ y$,  and the objective function  $ f({\bf x}):= \|  {\bf G} x- y  \|_{2}^{2} $ in \eqref{eq1} as follows:
 $$f({\bf x})= \sum_{i=1}^{n} f^{(i)}({\bf x}):= \sum_{i=1}^{n} \| {\bf G}_i x-y_i\|_2^2.  $$
In the simulations, without a loss of generality, we assume the number of worker partitions equals the number of server nodes; i.e., $ p = n $. We also assume that the repartitioned parts have the same size; i.e., the number of rows in $ {\bf G}_{i} $
and the lengths of vectors $y_i, 1\le i\le n$ are the same.

Each worker node finds its local coded gradient through a combination of uncoded local gradients computed through local optimization problems of  overdetermined linear systems of equations. The stepsizes $\alpha_{k}$ are chosen such that
\vspace{-0.1cm}
\begin{equation}
    \alpha_{k} = \frac{1}{(k+a)^{-\theta}}.
\end{equation}  
We use the  absolute error
\vspace{-0.8cm}
\begin{equation*}
\begin{split}
 \hspace{4.2cm} {\rm AE}:= \max_{1\le i\le n} \frac{\|{\bf x}_i(k)-{\bf x}_o\|_2}{\|{\bf x}_0\|_2}      
\end{split}
\end{equation*}
and  consensus error
\vspace{-0.5cm}
\begin{equation*}
\begin{split}
 \hspace{3.8cm}  {\rm CE}:= \max_{1\le i\le n} \frac{\|{\bf x}_i(k)-\bar{\bf x}(k)\|_2}{\|{\bf x}_o\|_2}     
\end{split}
\end{equation*}

%
to measure  the performance of SRDO, where $n$ is the number of server nodes (i.e., servers) in the network.

We simulate the algorithm for the fixed network topologies of the allowed number of stragglers or varying connection topologies according to gradient computation scenarios $ 1 $ through $ 3 $.


We present the simulation for $ 100 $ samples of parameter server networks of $ p =5 $ equal sized partitions $ i $, where $ n_{i} =3 $, $ s_{i} = 1 $ and $ n_{i} = 5 $, $ s_{i} = 2 $ for $ \mu = 100 $, $ N =100 $, i.e., $ m_{i} = 300$ and $ 500 $, and $ M= 1500 $ and $ 2500 $, respectively. Here, $ \mu $ stands for the number of rows in a partition sub-partition, which is assumed equal all over the network. That is, $ \mu $ corresponds to the functions $ f^{(i)}_{l} $, where $ l $ corresponds to worker node $ l $ in partition $ i $.
We simulate different samples from variant connectivity levels:

\begin{itemize}
\item{\emph{Strong Connectivity} Condition:}
i) the allowed-number-of-stragglers connection topology (i.e., strong connectivity condition under gradient computation scenario $ 1 $), Fig. 
 \ref{Fig.5} in comparison with the full connection with the no failures scenario of uncoded centralized gradient descent algorithm.
\item{\emph{Weak Connectivity} Condition:}
ii) the more-than-the-allowed number of stragglers connection with some instants of allowed number of stragglers connectivity under the weak straggler tolerance condition compared with the uncoded centralized gradient descent. For the two used gradient computation scenarios, we have:
\begin{itemize}
\item{a) Scenario $ 2 $ where only iteration available instant gradients are used, Fig. \ref{Fig.7}.}
\item{b) Scenario $ 3 $ with delayed gradients used, Fig. \ref{Fig.9} and Fig. \ref{Fig.10}.}
\end{itemize}
\end{itemize}

\subsection{Discussion of the Figures}


In Fig. \ref{Fig.5},  SRDO algorithm converges in a rate close to the uncoded centralized gradient descent algorithm with full connection. The convergence can be faster or slower depending on the condition number of the coded matrices in relation to the uncoded matrices at each node; i.e., this is related to the respective Lipschitz constants $ L $.
For Figs. \ref{Fig.7}, 
\ref{Fig.9}, 
and \ref{Fig.10}, the behavior matches the described convergence rates in Section VII. The fluctuation of the average consensus  error for the SRDO has a larger variation, which is dependent on $ T = \max_{i \in [1:n] }T_{i} $, while the proposed algorithm corresponding absolute
error behaves more smoothly.

\vspace{-0.5cm}
\begin{figure}[H]
\hspace*{-0.5cm}
\includegraphics[bb=0 0 800 800,scale=0.5]{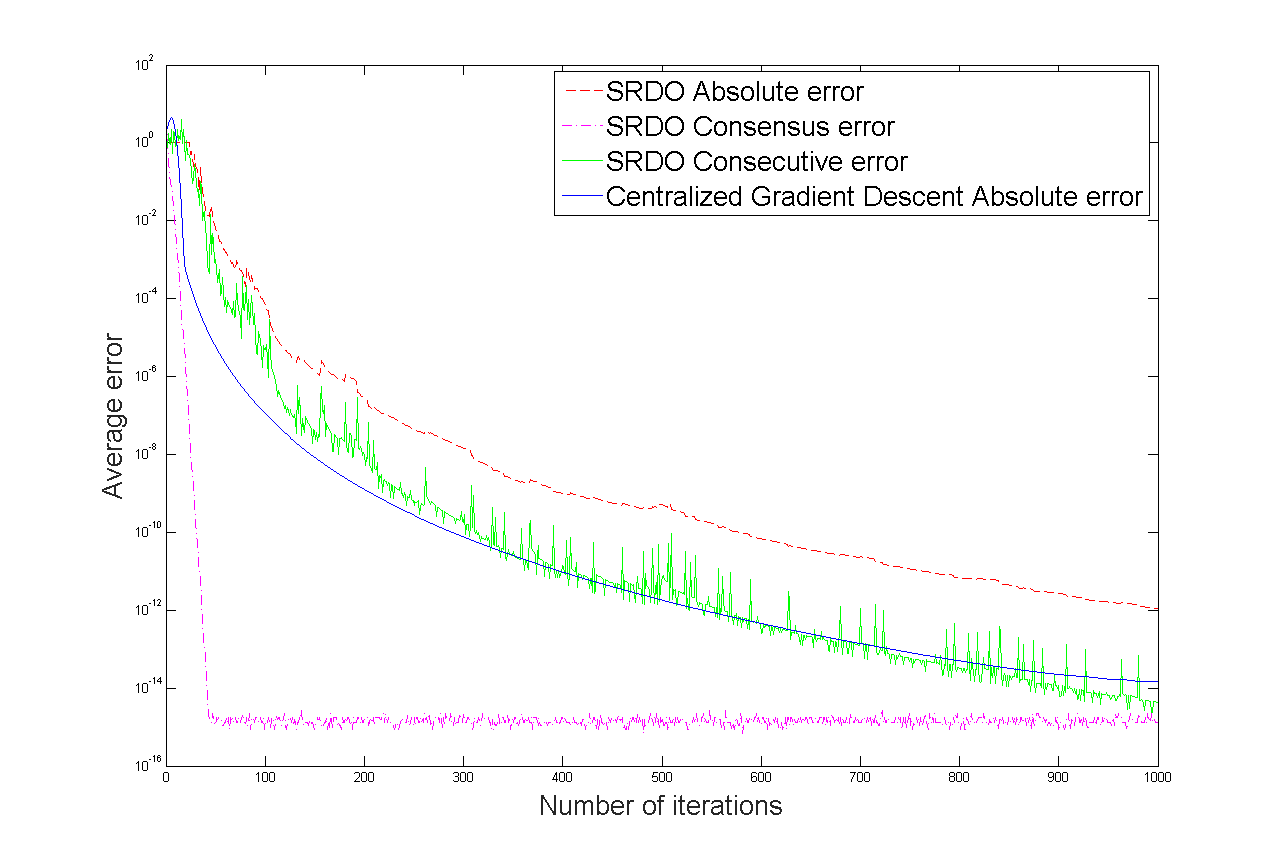}
\caption{Allowed number of stragglers connection for SRDO algorithm network (n=5, s=2) with $ \frac{1}{(k+300)^{0.55}} $ and centralized uncoded gradient descent algorithm for a full connection with no stragglers (n=5, s=0), for $ M= 2500, N =100 $.}
\label{Fig.5}
\end{figure}

We can conceive from the simulations that SRDO has faster convergence for a smaller exponent $\theta \in (0 , 1] $, which confirms its convergence rate estimate in Section VII.
However, the simulations also indicate that decreasing exponent $\theta$ moves SRDO into the divergence phase, which could directly be related to the complexity of the network.
It is  worth mentioning that we can adequately calibrate this divergence by increasing the value of $ a $ in our illustrated examples for a fixed exponent $ \theta $.
We can see that for a fixed value of $ T $ and a fixed condition number, i.e., fixed Lipschitz constant $ L $, (more specifically for a fixed matrix $ G_{i} $), the decrease in the exponent $ \theta $ allows the algorithm to enter the divergence instability region. Then an increase in the stepsize will make it converge the fastest where then any increase will ultimately degrade the algorithm to a slower convergence.
Similarly, if we fix $ T $ and the stepsize, the behavior of the convergence of the algorithm relative to the change in the condition number is the same as that relative to the stepsize in the previous scenario.
Moreover, we see that when $ T $ increases, and the allowed number of stragglers connection becomes less frequent, then the convergence is replaced by an anticipated divergence. Then, for that $ T $, we can reenter the convergence region of the algorithm by increasing the exponent $ \theta $ for a fixed optimization problem.  Convergence is also achieved for problems with matrices of higher condition numbers when the stepsize is fixed.
As for scenario $ 3 $, we see that its performance is better than that of scenario $ 2 $, because scheme $ 3 $ exploits stale gradients to form the overall gradient at each instant, although both schemes $ 2 $ and $ 3 $ have slower convergence rates for the fixed stepsize and fixed condition number.

\begin{figure}
\hspace*{-0.5cm}
\includegraphics[bb=0 0 800 800,scale=0.5]{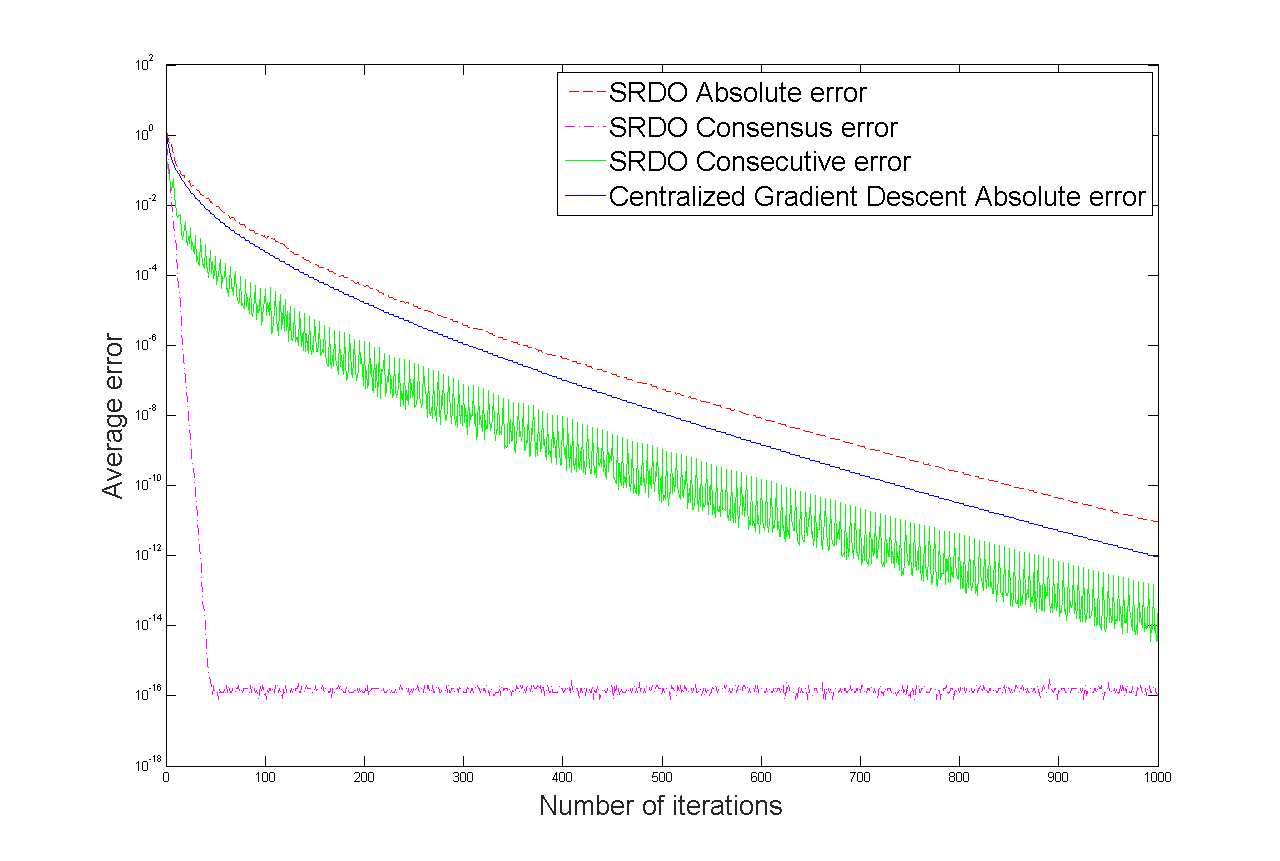}
\caption{Varying number of stragglers connection for SRDO algorithm network (n=3, s=1) and $ T=5 $ for $ \alpha_{k}= \frac{1}{(k+300)^{0.35}} $ using gradient computation scenario $ 2 $ and centralized uncoded gradient descent algorithm for a full connection with no stragglers (n=3, s=0), for $ M= 1500, N =100 $.}
\label{Fig.7}
\end{figure}

\begin{figure}
\hspace*{-0.5cm}
\includegraphics[bb=0 0 800 800,scale=0.5]{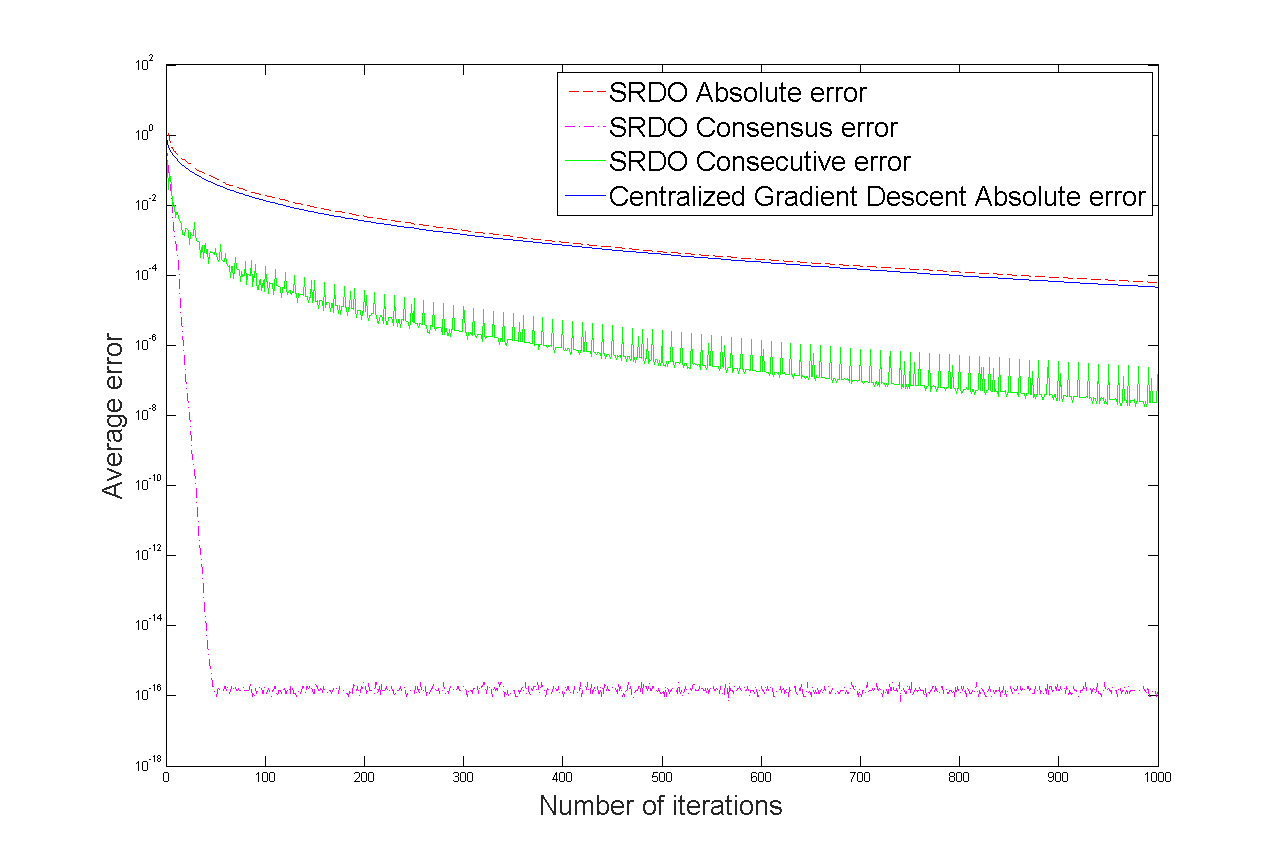}
\caption{Varying number of stragglers connection for SRDO algorithm network (n=3, s=1) and $ T=10 $ for $ \alpha_{k}= \frac{1}{(k+300)^{0.55}} $ using gradient computation scenario $ 3 $  and centralized uncoded gradient descent algorithm for a full connection with no stragglers (n=3, s=0), for $ M= 1500, N =100 $.}
\label{Fig.9}
\end{figure}

Thus, we can anticipate in Fig. \ref{Fig.7} that the value of $ \theta = 0.35 $ allowed a comparable convergence rate of the SRDO for $ T = 5 $ as that of the centralized gradient descent algorithm.
In Figs. \ref{Fig.9} and \ref{Fig.10}, we realize that the lower value of $ \theta=0.35 $ is not permissible, because the SRDO algorithm will considerably enter the instability region, while a higher value of $ \theta=0.55 $ favors a better convergence rate for $ T=10 $, and the highest value of $ \theta=0.75 $ a better convergence rate for $ T=20 $.
Moreover, in the simulation we have used a definite coding scheme introduced in \citep{tandon17a}. We could have used other schemes. As it is proven in Section VII, we anticipate a behavior dependent on the coding scheme as shown in \eqref{convergence_rate}. Meanwhile, we can adequately propose a coding scheme that will better enhance the convergence rate.

\begin{figure}
\hspace*{-0.5cm}
\includegraphics[bb=0 0 800 800,scale=0.5]{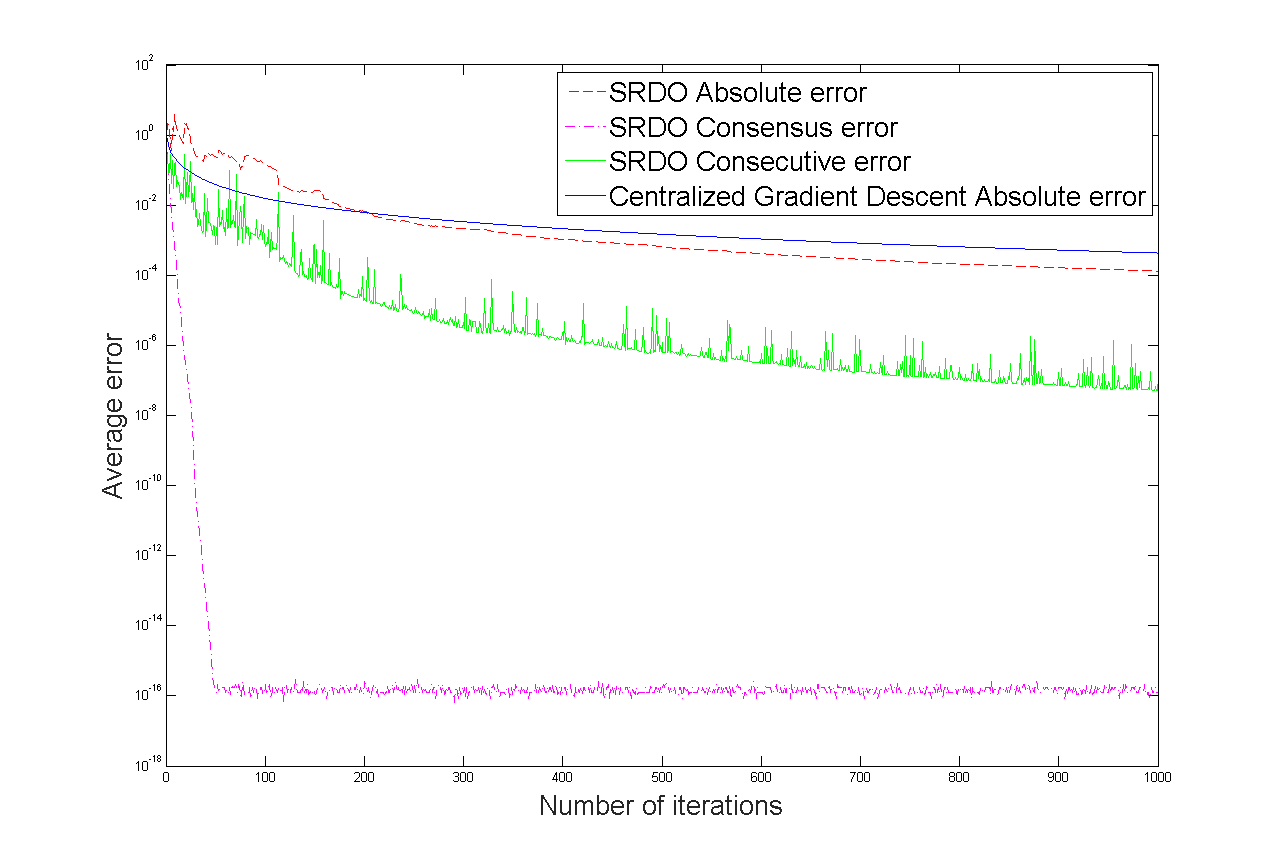}
\caption{Varying number of stragglers connection for SRDO algorithm network (n=3, s=1) and $ T=20 $ for $ \alpha_{k}= \frac{1}{(k+300)^{0.75}} $ using gradient computation scenario $ 3 $  and centralized uncoded gradient descent algorithm for a full connection with no stragglers (n=3, s=0), for $ M= 1500, N =100 $.}
\label{Fig.10}
\end{figure}



\section{Conclusion}

We have considered in this paper a parameter server network algorithm, SRDO, for minimizing a convex function that consists of a number of component functions. We restricted the simulation for the case of a quadratic function which corresponds to solving an overdetermined system of linear equations. A convergence proof for this algorithm in its general form (not necessarily a quadratic function) was provided in the case of network topologies where the number of stragglers is under the allowed threshold (cf. \cite{tandon17a}). The convergence of the algorithm variants, in the case of the number of stragglers exceeding the quantity allowed, is formalized in Section VII through the described convergence rates. Furthermore, the simulation showed optimal results for the algorithm convergence rate, not only for the allowed number of stragglers scenarios but also for the number of stragglers exceeding the allowed threshold. These metrics matched the centralized gradient descent method with the bonus of robustness to an allowed number of stragglers, and to the case of stragglers exceeding the allowed number. We further analytically showed that the convergence rate can be considerably enhanced through applying an adequate coding scheme.

\begin{appendix}

\subsection{Proof of Lemma~\ref{general_error_bound}}

\begin{proof}
{\footnotesize \begin{equation}
\begin{split}
    & \|  {\bf v}_{l}(k+1)-  {\bf x}^{*}\|^{2} = \|\sum_{j=1}^{n} [{\bf W}(k)]_{l,j} {\bf x}_{j}(k+1)- \sum_{j=1}^{n} [{\bf W}(k)]_{l,j} {\bf x}^{*}\|^{2}  \\
     & \leq \sum_{j=1}^{n} [{\bf W}(k)]_{l,j} \| {\bf x}_{j}(k+1)- {\bf x}^{*}\|^{2} \\
     & \leq \sum_{j=1}^{n}{[\bf W}(k)]_{l,j} \| \sum_{(i)=1}^{p} \gamma_{(i)} [{\bf v}_{j}(k) -\alpha_{k}\widehat{\nabla{f}^{(i)}}({\bf v}_{j}(k))]+\gamma_{(0)} {\bf v}_{j}(k) - {\bf x}^{*} \|^{2} \\
     & \leq \sum_{j=1}^{n}[{\bf W}(k)]_{l,j} \| {\bf v}_{j}(k) -\alpha_{k} \sum_{(i)=1}^{p}\widehat{\nabla{f}^{(i)}}({\bf v}_{j}(k))- {\bf x}^{*} \|^{2} \\
     & \leq \sum_{j=1}^{n}[{\bf W}(k)]_{l,j} [\| {\bf v}_{j}(k) - {\bf x}^{*} \|^{2} - 2 \alpha_{k}\sum_{(i)=1}^{p} \gamma_{(i)}\langle \widehat{\nabla{f}^{(i)}}({\bf v}_{j}(k)), {\bf v}_{j}(k) - {\bf x}^{*} \rangle \\
     & + \alpha_{k}^{2} \|\sum_{(i)=1}^{p}\gamma_{(i)}\widehat{\nabla{f}^{(i)}}({\bf v}_{j}(k))\|^{2}] \\
     & \leq \sum_{j=1}^{n}[{\bf W}(k)]_{l,j} [\|  {\bf v}_{j}(k) - {\bf x}^{*} \|^{2} + 2 \alpha_{k} \sum_{(i)=1}^{p} \gamma_{(i)} \langle \nabla{f}^{(i)}({\bf v}_{j}(k)), {\bf x}^{*} - {\bf v}_{j}(k) \rangle \\
     & +2 \alpha_{k} \sum_{(i)=1}^{p} \gamma_{(i)}\langle {\bf \epsilon}_{j,(i)}(k), v_{j}(k) - {\bf x}^{*} \rangle 
     + \alpha_{k}^{2} \|\sum_{(i)=1}^{p}\gamma_{(i)}(\nabla{f}^{(i)}({\bf v}_{j}(k)) - {\bf \epsilon}_{j,(i)}(k))\|^{2}] \\
     & \leq \sum_{j=1}^{n}[{\bf W}(k)]_{l,j} [ \|  {\bf v}_{j}(k) - {\bf x}^{*} \|^{2} + 2 \alpha_{k} \sum_{(i)=1}^{p} \gamma_{(i)} \langle \nabla{f}^{(i)}({\bf v}_{j}(k)), {\bf x}^{*} - {\bf v}_{j}(k) \rangle \\
     & +2 \alpha_{k} \sum_{(i)=1}^{p} \gamma_{(i)}\langle {\bf \epsilon}_{j,(i)}(k), {\bf v}_{j}(k) - {\bf x}^{*} \rangle 
      + \alpha_{k}^{2} \|\sum_{(i)=1}^{p}\gamma_{(i)}\frac{(\nabla{f}^{(i)}({\bf v}_{j}(k)) - {\bf \epsilon}_{j,(i)}(k))}{(\sum_{(i)=1}^{p}\gamma_{(i)})^{2}}\|^{2}]
     \end{split}
     \end{equation}}
     
\end{proof}
Then summing from $i=1$ to $n$ and knowing that the sum of each column is less than or equal to $ 1 -\mu $ then the lemma follows.

\subsection{ Proof of Lemma~\ref{type1_error_bound} }

Having Lemma~\ref{general_error_bound} then 

we have
{ \begin{equation}\label{type1_exp}
\begin{split}
    & ( 1- \mu) \sum_{l=1}^{n}\|  {\bf v}_{l}(k+1) - {\bf x}^{*}\|^{2}  \leq (1-\mu) \sum_{j=1}^{n}\|{\bf v}_{j}(k)- {\bf x}^{*}\|^{2}  \\
   & + 2 \alpha_{k}\sum_{j=1}^{n}\sum_{(i)=1}^{p}\gamma_{(i)} \langle {\bf \epsilon}_{j,(i)}(k), {\bf v}_{j} - {\bf x}^{*} \rangle  + 2 \alpha_{k}^{2} \frac{\sum_{j=1}^{n}\sum_{(i)=1}^{p}\gamma_{(i)}}{(\sum_{(i)=1}^{p}\gamma_{(i)})^{2}}\|{\bf \epsilon}_{j,(i)}(k)\|^{2}\\
    & + 2 \alpha_{k} \sum_{j=1}^{n}\sum_{(i)=1}^{p}\gamma_{(i)} \langle \nabla{f}^{(i)}({\bf v}_{j}(k)), {\bf x}^{*} - {\bf v}_{j}(k) \rangle  \\
    &  + \frac{2 \alpha_{k}^{2}}{(\sum_{(i)=1}^{p}\gamma_{(i)})^{2}})\sum_{j=1}^{n}\sum_{(i)=1}^{p}\gamma_{(i)} \| \nabla{f}^{(i)}({\bf v}_{j}(k)) \| ^{2}
\end{split}
\end{equation}}

But
{ \begin{equation}\label{des_exp}
\begin{split}
& \langle \nabla{f}^{(i)}({\bf v}_{j}(k)), {\bf x}^{*} - {\bf v}_{j}(k) \rangle  \\
& = - \langle \nabla{f}^{(i)}({\bf v}_{j}(k)), {\bf v}_{j}(k) -{\bf x}^{*} \rangle \\ & = - \langle \nabla{f}^{(i)}({\bf v}_{j}(k)), {\bf v}_{j}(k) -{\bf x}^{(i)} \rangle \\
& = - \langle \nabla{f}^{(i)}({\bf v}_{j}(k)) - \nabla{f}^{(i)}({\bf x}^{(i)}), {\bf v}_{j}(k) -{\bf x}^{(i)} \rangle 
\end{split}
\end{equation}}
While 
{ \begin{equation}
\begin{split}
 \nabla{f}^{(i)}({\bf v}_{j}(k)) - \nabla{f}^{(i)}({\bf x}^{(i)}) = a \| {\bf v}_{j}(k) - {\bf x}^{(i)} \| \overrightarrow{u}
\end{split}
\end{equation}}
where $ \| \overrightarrow{u} \| = 1 $ and $  0 \leq a \leq L $.
and 
{ \begin{equation}
\begin{split}
& {\bf v}_{j}(k) -{\bf x}^{(i)} = \| {\bf v}_{j}(k) -{\bf x}^{(i)} \| \overrightarrow{v}
\end{split}
\end{equation}}
where $ \| \overrightarrow{v} \| = 1 $.
Using what preceded we have the expression in \eqref{des_exp} equal to
{ \begin{equation}
\begin{split}
& \langle \nabla{f}^{(i)}({\bf v}_{j}(k)), {\bf x}^{*} - {\bf v}_{j}(k) \rangle  \\
& = - \langle \nabla{f}^{(i)}({\bf v}_{j}(k)) - \nabla{f}^{(i)}({\bf x}^{(i)}), {\bf v}_{j}(k) -{\bf x}^{(i)} \rangle \\
& = -a \| {\bf v}_{j}(k) - {\bf x}^{(i)} \| ^{2}  \langle \overrightarrow{u},\overrightarrow{v} \rangle
\end{split}
\end{equation}}
But since $  \langle \overrightarrow{u},\overrightarrow{v} \rangle \geq 0 $ due to the monotonicity of the gradient we have $  0 \leq \langle \overrightarrow{u},\overrightarrow{v} \rangle \leq  1 $.
Then
{ \begin{equation}\label{res1}
\begin{split}
& \langle \nabla{f}^{(i)}({\bf v}_{j}(k)), {\bf x}^{*} - {\bf v}_{j}(k) \rangle  \\
& = -a \| {\bf v}_{j}(k) - {\bf x}^{(i)} \| ^{2}  b
\end{split}
\end{equation}}
where $ 0 \leq b \leq 1 $.


Similarly, 
{\small \begin{equation}\label{res2}
\begin{split}
\| \nabla{f}^{(i)} & ({\bf v}_{j}(k)) \|^{2} =  \\
& \langle \nabla{f}^{(i)}({\bf v}_{j}(k)) - \nabla{f}^{(i)}({\bf x}^{(i)}),\nabla{f}^{(i)}({\bf v}_{j}(k)) - \nabla{f}^{(i)}({\bf x}^{(i)}) \rangle \\
& = a^{2} \| {\bf v}_{j}(k) -{\bf x}^{(i)} \| ^{2} \rangle \overrightarrow{u},\rightarrow{u} \rangle \\
& = a^{2} \| {\bf v}_{j}(k) -{\bf x}^{(i)} \| ^{2}
\end{split}
\end{equation}}
Then substituting \eqref{res1} and \eqref{res2} in \eqref{type1_exp} we have
{ \begin{equation}\label{type1_exp}
\begin{split}
    & ( 1 -\mu ) \sum_{l=1}^{n}\|  {\bf v}_{l}(k+1) - {\bf x}^{*}\|^{2}  \leq (1-\mu) \sum_{j=1}^{n}\|{\bf v}_{j}(k)- {\bf x}^{*}\|^{2}  \\
   & + 2 \alpha_{k}\sum_{j=1}^{n}\sum_{(i)=1}^{p}\gamma_{(i)} \langle {\bf \epsilon}_{j,(i)}(k), {\bf v}_{j} - {\bf x}^{*} \rangle  + 2 \alpha_{k}^{2} \frac{\sum_{j=1}^{n}\sum_{(i)=1}^{p}\gamma_{(i)}}{(\sum_{(i)=1}^{p}\gamma_{(i)})^{2}}\|{\bf \epsilon}_{j,(i)}(k)\|^{2}\\
    & - 2 a \alpha_{k} \sum_{j=1}^{n}\sum_{(i)=1}^{p}\gamma_{(i)} (b - \frac{2 \alpha_{k} a}{(\sum_{(i)=1}^{p}\gamma_{(i)})^{2}} ) \| {\bf v}_{j}(k) - {\bf x}^{(i)} \| ^{2}   
\end{split}
\end{equation}}

\subsection{ Proof of Lemma~\ref{type2_error_bound}}

From Lemma~\ref{general_error_bound} we have
we have
{ \begin{equation}
\begin{split}
    & ( 1- \mu) \sum_{l=1}^{n}\|  {\bf v}_{l}(k+1) - {\bf x}^{*}\|^{2}  \leq (1-\mu) \sum_{j=1}^{n}\|{\bf v}_{j}(k)- {\bf x}^{*}\|^{2}  \\
   & + 2 \alpha_{k}\sum_{j=1}^{n}\sum_{(i)=1}^{p}\gamma_{(i)} \langle {\bf \epsilon}_{j,(i)}(k), {\bf v}_{j} - {\bf x}^{*} \rangle  + 2 \alpha_{k}^{2} \frac{\sum_{j=1}^{n}\sum_{(i)=1}^{p}\gamma_{(i)}}{(\sum_{(i)=1}^{p}\gamma_{(i)})^{2}}\|{\bf \epsilon}_{j,(i)}(k)\|^{2}\\
    & + 2 \alpha_{k} \sum_{j=1}^{n}\sum_{(i)\in I }\gamma_{(i)} \langle \nabla{f}^{(i)}({\bf v}_{j}(k)), {\bf x}^{*} - {\bf v}_{j}(k) \rangle  \\
     & + 2 \alpha_{k} \sum_{j=1}^{n}\sum_{(i)\in I^{\complement} }\gamma_{(i)} \langle \nabla{f}^{(i)}({\bf v}_{j}(k)), {\bf x}^{*} - {\bf v}_{j}(k) \rangle  \\
    &  + \frac{2 \alpha_{k}^{2}}{(\sum_{(i)=1}^{p}\gamma_{(i)})^{2}})\sum_{j=1}^{n}\sum_{(i)=1}^{p}\gamma_{(i)} \| \nabla{f}^{(i)}({\bf v}_{j}(k)) \| ^{2}
\end{split}
\end{equation}}

But for $(i) \in I$ we have $ f^{(i)}({\bf x}^{(i)}) < f^{(i)}({\bf x}^{*})$ and for $(i) \in I^{\complement}$ we have $ f^{(i)}({\bf x}^{*})= f^{(i)}({\bf x}^{(i)})$, then the above inequality becomes
{ \begin{equation}
\begin{split}
    & ( 1- \mu) \sum_{l=1}^{n}\|  {\bf v}_{l}(k+1) - {\bf x}^{*}\|^{2}  \leq (1-\mu) \sum_{j=1}^{n}\|{\bf v}_{j}(k)- {\bf x}^{*}\|^{2}  \\
   & + 2 \alpha_{k}\sum_{j=1}^{n}\sum_{(i)=1}^{p}\gamma_{(i)} \langle {\bf \epsilon}_{j,(i)}(k), {\bf v}_{j} - {\bf x}^{*} \rangle  + 2 \alpha_{k}^{2} \frac{\sum_{j=1}^{n}\sum_{(i)=1}^{p}\gamma_{(i)}}{(\sum_{(i)=1}^{p}\gamma_{(i)})^{2}}\|{\bf \epsilon}_{j,(i)}(k)\|^{2}\\
    & + 2 \alpha_{k} \sum_{j=1}^{n}\sum_{(i)\in I }\gamma_{(i)} \langle \nabla{f}^{(i)}({\bf v}_{j}(k)), {\bf x}^{*} - {\bf x}^{(i)} - {\bf x}^{(i)} - {\bf v}_{j}(k) \rangle  \\
     & + 2 \alpha_{k} \sum_{j=1}^{n}\sum_{(i)\in I^{\complement} }\gamma_{(i)} \langle \nabla{f}^{(i)}({\bf v}_{j}(k)), {\bf x}^{(i)} - {\bf v}_{j}(k) \rangle  \\
    &  + \frac{2 \alpha_{k}^{2}}{(\sum_{(i)=1}^{p}\gamma_{(i)})^{2}})\sum_{j=1}^{n}\sum_{(i)=1}^{p}\gamma_{(i)} \| \nabla{f}^{(i)}({\bf v}_{j}(k)) \| ^{2}
\end{split}
\end{equation}}

which becomes
{ \begin{equation}\label{type2_error_bound_no_subst}
\begin{split}
    & ( 1- \mu) \sum_{l=1}^{n}\|  {\bf v}_{l}(k+1) - {\bf x}^{*}\|^{2}  \leq (1-\mu) \sum_{j=1}^{n}\|{\bf v}_{j}(k)- {\bf x}^{*}\|^{2}  \\
   & + 2 \alpha_{k}\sum_{j=1}^{n}\sum_{(i)=1}^{p}\gamma_{(i)} \langle {\bf \epsilon}_{j,(i)}(k), {\bf v}_{j} - {\bf x}^{*} \rangle  + 2 \alpha_{k}^{2} \frac{\sum_{j=1}^{n}\sum_{(i)=1}^{p}\gamma_{(i)}}{(\sum_{(i)=1}^{p}\gamma_{(i)})^{2}}\|{\bf \epsilon}_{j,(i)}(k)\|^{2}\\
    & + 2 \alpha_{k} \sum_{j=1}^{n}\sum_{(i)\in I }\gamma_{(i)} \langle \nabla{f}^{(i)}({\bf v}_{j}(k)), {\bf x}^{*} - {\bf x}^{(i)} \rangle  \\
     & + 2 \alpha_{k} \sum_{j=1}^{n}\sum_{(i)=1}^{p}\gamma_{(i)} \langle \nabla{f}^{(i)}({\bf v}_{j}(k)), {\bf x}^{(i)} - {\bf v}_{j}(k) \rangle  \\
    &  + \frac{2 \alpha_{k}^{2}}{(\sum_{(i)=1}^{p}\gamma_{(i)})^{2}})\sum_{j=1}^{n}\sum_{(i)=1}^{p}\gamma_{(i)} \| \nabla{f}^{(i)}({\bf v}_{j}(k)) \| ^{2}
\end{split}
\end{equation}}

which by using \eqref{res1} and \eqref{res2} we get the result.

\subsection{Martingale 1}

\begin{lem}\label{Lemma6a}
Assume the following inequality holds a.s. for all $ k \geq k^{*} $,

\vspace{-0.3cm}

{ \begin{equation}\label{eqw}
    v_{k+1} \leq a_{1}v_{k}+ a_{2,k}\max_{k-B \leq \hat{k} \leq k} v_{\hat{k}}
\end{equation}}
$ v_{k}$. $u_{k}$, $ b_{k} $, $ c_{k} $, $ a_{1} $ and $ a_{2,k} $ are non-negative random variables where $ a_{1} + a_{2,k} \leq 1 $ and $ \{ a_{2,k} \} $ is a decreasing sequences. Then if for $ \rho = (a_{1}+a_{2,1})^{\frac{1}{B+1}} $ and
{ \begin{equation}\label{base}
    v_{k_{0}} \leq \rho^{\Phi(k_{0})}V_{0}^{'} \ \ \ a.s.
\end{equation}}
for base case $ k = k_{0} = \bar{k} - B  $. (i.e., notice $ V_{0} $ is not necessary the initial value $ v_{0} $). And $ \Phi $ is a random variable from $ \mathbb{N} $ to $ \mathbb{N} $ where $ \Phi([n,m])=[n,m] $.
 
And assume that this also holds for all $ k \geq k_{0} $ up to $ k = \bar{k} $ in an arbitrary manner (i.e., notice the power of $ \rho $ is independent of k ). i.e., $ k \in \{ k_{0}=\bar{k}-B, \ldots , \bar{k} \} $ and $ \bar{k} - B  \geq \max(k^{*},\tilde{k}) $. That is
{ \begin{equation}\label{ind1}
   v_{k} \leq \rho^{\Phi(k)}V_{0}^{'} \ \ \ a.s.
\end{equation}}
for $ k= \{ k_{0},\ldots,\bar{k} \} $. \\
Then we have

{ \begin{equation}\label{rec1}
     v_{k} \leq \rho^{k}V_{0} \ \ \ a.s.
 \end{equation}}
for all $ k \geq \bar{k} $ where $ V_{0} > 0 $ for all sequences patterns and $ \rho $ as before.
\end{lem}

{\bf Proof:} First since $ a_{1} + a_{2,k} \leq 1 $ then 
{ \begin{equation}
\begin{split}
 & 1  \leq (a_{1}+a_{2,k})^{-\frac{B}{B+1}} \implies
  1  \leq (a_{1}+a_{2,1})^{-\frac{B}{B+1}} \\
 \implies & (a_{1}+a_{2,k}) \leq (a_{1}+a_{2,1}) 
\end{split}
\end{equation}}
which implies that
{ \begin{equation*}
\begin{split}
a_{1} + & a_{2,k}  \rho ^ {-B}  = a_{1} + a_{2,k}(a_{1}+a_{2,1})^{-\frac{B}{B+1}} \\
& \leq a_{1,k}(a_{1}+a_{2,1})^{-\frac{B}{B+1}}+ a_{2,k}(a_{1}+a_{2,1})^{-\frac{B}{B+1}} \\
& = (a_{1}+a_{2,k})(a_{1}+a_{2,1})^{-\frac{B}{B+1}} \\
& \leq (a_{1}+a_{2,k})^{\frac{1}{B+1}} = \rho
\end{split}
\end{equation*}}
That is 
\vspace{-0.8cm}
{ \begin{equation}
\begin{split}
a_{1} + & a_{2,k}  \rho ^ {-B}  \leq \rho
\end{split}
\end{equation}}
Now, by induction we show that \eqref{rec1} for all $ k \geq k_{0} $.
Assume \eqref{base} is true for $ k = k_{0} $ and that the induction hypothesis holds for all $ k \geq k_{0} $ up to $ \bar{k} $ where $ k_{0}= k-B \leq k \leq \bar{k} $. Then  we have for any arbitrary behavior for $ k $ where $ k_{0}= k-B \leq k \leq \bar{k} $ that we can write the sequences $ v_{k} $ in a decreasing sequence. Without a loss of generality assume we will have for $ 0 \leq l \leq B $
{ \begin{equation}
\begin{split}
v_{\bar{k}} & \leq \rho ^ {\bar{k}-l} V_{0}^{'} \\
v_{\bar{k}-B} & \leq \rho ^{\Phi(\bar{k}-B)} V_{0}^{'} 
\end{split}    
\end{equation}}
Then from \eqref{eqw} we have

{\begin{equation}\label{induction}
\begin{split}
 v_{\bar{k}+1} & \leq a_{1}v_{\bar{k}}+ a_{2,k}\max_{\bar{k}-B \leq \hat{k} \leq \bar{k}} v_{\hat{k}} \\
     & \leq a_{1}\rho ^ {\bar{k}-l} V_{0}^{'} + a_{2,\bar{k}}\rho ^{\bar{k}-B} V_{0}^{'}  \\
      \leq & \ a_{1}\rho ^ {\bar{k} -l} V_{0}^{'} + a_{2,\bar{k}}\rho ^{\bar{k} - l - B } V_{0}^{'} 
      \\
     & = (a_{1} + a_{2,\bar{k}} \rho ^ {-B}) \rho ^ {\bar{k}-l } V_{0}^{'} 
       \\
      & \leq \rho^{\bar{k} - l + 1} V_{0}^{'}\ \ \ a.s.
\end{split}    
\end{equation}

But without a loss of generality, we can find $ V_{0} > 0 $ such that $ \rho^{\bar{k} - l + 1} V_{0}^{'} \leq \rho^{\bar{k}+1} V_{0}$ to keep indexing tractable.
And thus $ \eqref{induction} $ is true for all $ k \geq \bar{k} + 1 $. i.e., notice that for $ k + 1 = \bar{k} + 2 $, we already have for $ k  = \bar{k} + 1 $ that the power of $ \rho $ in the recursive inequality after the coefficient $ a_{1,\bar{k}} $ is $ \bar{k} + 1 $. Thus, no matter what the arbitrary behavior for the prior $ B $ terms is, we will have 

\vspace{-0.8cm}
{ \begin{equation}
\begin{split}
    v_{\bar{k}+2} & \leq \rho^{\bar{k} + 2} V_{0}^{'}  \ \ \ a.s.
\end{split}    
\end{equation}}

Thus, \eqref{induction} follows for all $ k \geq \bar{k} $. \hspace{3cm}
$ \blacksquare $

\begin{remark}
i.e., notice that it is true for $ k = \bar{k} $ since
{ \begin{equation}
\begin{split}
    v_{\bar{k}+1}  \leq  \rho^{\bar{k} - l} V_{0}^{'} \ \text{and}\
    v_{\bar{k}+1}  \leq  \rho^{\bar{k}} V_{0}
\end{split}    
\end{equation}}
\end{remark}

\subsection{Martingale 2}

\begin{lem}\label{Lemma6a1}
Assume the following inequality holds a.s. for all $ k \geq k^{*} $,

\vspace{-0.3cm}

{ \begin{equation}\label{eqw1}
    v_{k+1} \leq a_{1}v_{k}+ a_{2,k}\max_{k-B \leq \hat{k} \leq k} v_{\hat{k}}+a_{3}
\end{equation}}
$ v_{k}$. $u_{k}$, $a_{3} $. $ a_{1} $ and $ a_{2,k} $ are non-negative random variables where $ a_{1} + a_{2,k} \leq 1 $ and $ \{ a_{2,k} \} $ is a decreasing sequences. Then if for $ \rho = (a_{1}+a_{2,1})^{\frac{1}{B+1}} $ and $\eta=\frac{a_{3}}{1-a_{1}-a_{2,1}}$
{ \begin{equation}\label{base1}
    v_{k_{0}} \leq \rho^{\Phi(k_{0})}V_{0}^{'}+\eta \ \ \ a.s.
\end{equation}}
for base case $ k = k_{0} = \bar{k} - B  $. (i.e., notice $ V_{0} $ is not necessary the initial value $ v_{0} $). And $ \Phi $ is a random variable from $ \mathbb{N} $ to $ \mathbb{N} $ where $ \Phi([n,m])=[n,m] $.
 
And assume that this also holds for all $ k \geq k_{0} $ up to $ k = \bar{k} $ in an arbitrary manner (i.e., notice the power of $ \rho $ is independent of k ). i.e., $ k \in \{ k_{0}=\bar{k}-B, \ldots , \bar{k} \} $ and $ \bar{k} - B  \geq \max(k^{*},\tilde{k}) $. That is
{ \begin{equation}\label{ind11}
   v_{k} \leq \rho^{\Phi(k)}V_{0}^{'}+\eta \ \ \ a.s.
\end{equation}}
for $ k= \{ k_{0},\ldots,\bar{k} \} $. \\
Then we have

\vspace{-0.5cm}

{ \begin{equation}\label{rec11}
     v_{k} \leq \rho^{k}V_{0}+\eta \ \ \ a.s.
 \end{equation}}
for all $ k \geq \bar{k} $ where $ V_{0} > 0 $ for all sequences patterns and $ \rho $ as before.
\end{lem}

{\bf Proof:} First since $ a_{1,k} + a_{2,k} \leq 1 $ then 
{ \begin{equation}
\begin{split}
 & 1  \leq (a_{1}+a_{2,k})^{-\frac{B}{B+1}} \implies
  1  \leq (a_{1}+a_{2,1})^{-\frac{B}{B+1}} \\
 \implies & (a_{1}+a_{2,k}) \leq (a_{1}+a_{2,1}) 
\end{split}
\end{equation}}
which implies that
{ \begin{equation*}
\begin{split}
a_{1} + & a_{2,k}  \rho ^ {-B}  = a_{1} + a_{2,k}(a_{1}+a_{2,1})^{-\frac{B}{B+1}} \\
& \leq a_{1}(a_{1}+a_{2,1})^{-\frac{B}{B+1}}+ a_{2,k}(a_{1}+a_{2,1})^{-\frac{B}{B+1}} \\
& = (a_{1}+a_{2,k})(a_{1}+a_{2,1})^{-\frac{B}{B+1}} \\
& \leq (a_{1}+a_{2,k})^{\frac{1}{B+1}} = \rho
\end{split}
\end{equation*}}
That is 
\vspace{-0.8cm}
{ \begin{equation}
\begin{split}
a_{1} + & a_{2,k}  \rho ^ {-B}  \leq \rho
\end{split}
\end{equation}}
Now, by induction we show that \eqref{rec1} for all $ k \geq k_{0} $.
Assume \eqref{base1} is true for $ k = k_{0} $ and that the induction hypothesis holds for all $ k \geq k_{0} $ up to $ \bar{k} $ where $ k_{0}= k-B \leq k \leq \bar{k} $. Then  we have for any arbitrary behavior for $ k $ where $ k_{0}= k-B \leq k \leq \bar{k} $ that we can write the sequences $ v_{k} $ in a decreasing sequence. Without a loss of generality assume we will have for $ 0 \leq l \leq B $
{ \begin{equation}
\begin{split}
v_{\bar{k}} & \leq \rho ^ {\bar{k}-l} V_{0}^{'} + \eta \\
v_{\bar{k}-B} & \leq \rho ^{\Phi(\bar{k}-B)} V_{0}^{'} +\eta
\end{split}    
\end{equation}}
Then from \eqref{eqw1} we have
{ \begin{equation}\label{induction1}
\begin{split}
 v_{\bar{k}+1} & \leq a_{1}v_{\bar{k}}+ a_{2,k}\max_{\bar{k}-B \leq \hat{k} \leq \bar{k}} v_{\hat{k}}+ a_{3} \\
     & \leq a_{1}\rho ^ {\bar{k}-l} V_{0}^{'} + a_{2,\bar{k}}\rho ^{\bar{k}-B} V_{0}^{'} +  a_{1} \eta + a_{2,\bar{k}} \eta + a_{3} \\
     & \leq  \ a_{1}\rho ^ {\bar{k} -l} V_{0}^{'} + a_{2,\bar{k}}\rho ^{\bar{k} - l - B } V_{0}^{'} +  a_{1} \eta + a_{2,\bar{k}} \eta + a_{3} \\
      & = (a_{1} + a_{2,\bar{k}} \rho ^ {-B}) \rho ^ {\bar{k}-l } V_{0}^{'} 
      + \eta \\
      & \leq \rho^{\bar{k} - l + 1} V_{0}^{'} + \eta \ \ \ a.s.
\end{split}    
\end{equation}}

But without a loss of generality, we can find $ V_{0} > 0 $ such that $ \rho^{\bar{k} - l + 1} V_{0}^{'} \leq \rho^{\bar{k}+1} V_{0}$ to keep indexing tractable.
And thus $ \eqref{induction1} $ is true for all $ k \geq \bar{k} + 1 $. i.e., notice that for $ k + 1 = \bar{k} + 2 $, we already have for $ k  = \bar{k} + 1 $ that the power of $ \rho $ in the recursive inequality after the coefficient $ a_{1} $ is $ \bar{k} + 1 $. Thus, no matter what the arbitrary behavior for the prior $ B $ terms is, we will have 

\vspace{-0.8cm}
{ \begin{equation}
\begin{split}
    v_{\bar{k}+2} & \leq \rho^{\bar{k} + 2} V_{0}^{'} + \eta \ \ \ a.s.
\end{split}    
\end{equation}}
Thus, \eqref{induction1} follows for all $ k \geq \bar{k} $. \hspace{3cm} $ \blacksquare $
\begin{remark}
i.e., notice that it is true for $ k = \bar{k} $ since
{ \begin{equation}
\begin{split}
    v_{\bar{k}+1}  \leq  \rho^{\bar{k} - l} V_{0}^{'} +\eta \ \text{and} \
    v_{\bar{k}+1}  \leq  \rho^{\bar{k}} V_{0} + \eta
\end{split}    
\end{equation}}
\end{remark}}

\subsection{Evaluation of $ \sum_{i=1}^{n} \| R_{i}(k-1) \| ^{2} $}\label{app:evaluateR}

Then by Cauchy-Schwartz inequality, we have \\
$ \| R_{i}(k-1) \| ^{2} \leq \alpha_{k-1} ^{2} \|{\bf A}^{(i)} \| ^{2} _{\infty} \|{\bf B}^{(i)}\| _{2,\infty} ^{2}  \| \nabla{f}^{(i)}_{l}({\bf v}_{M}(k-1))-\nabla{f}^{(i)}_{l}({\bf v}_{m}(k-1)) \| ^{2} $ where   $ \|{\bf A}^{(i)} \| ^{2}_{2,\infty} $ is the norm of the row of ${\bf A}^{(i)}$ with maximum $ l_2 $ norm 
which is bounded by the  Euclidean norm of the vector formed by the support of a row of ${\bf A}^{(i)}$ of length equals $n_{i}-s_{i}$ of maximum Euclidean norm, since not all coefficients are nonzero.
And $ {\bf v}_{M}(k-1) $ and $ {\bf v}_{m}(k-1) $ are the instants at iteration $ k' $  where $ k-1-H \leq k' \leq k - 1 $ with 
$  \| \nabla{f}^{(i)}_{l}({\bf v}_{M}(k-1))-\nabla{f}^{(i)}_{l}({\bf v}_{m}(k-1)) \| = \max_{i,j} \| \nabla{f}^{(i)}_{l}({\bf v}_{j}(k'))-\nabla{f}^{(i)}_{l}({\bf v}_{i}(k')) \| $. Then

{ \begin{equation*}\label{3.57}
\begin{split}
\ \ &   \| \nabla{f}^{(i)}_{l}({\bf v}_{M}(k-1))-\nabla{f}^{(i)}_{l}({\bf v}_{m}(k-1)) \| \\
\leq & \| \nabla{f}^{(i)}_{l}({\bf v}_{M}(k-1))-\nabla{f}^{(i)}_{l}(x^{*}) \| \\
& + \| \nabla{f}^{(i)}_{l}({\bf v}_{m}(k-1))-\nabla{f}^{(i)}_{l}(x^{*}) \| \\
& +  \| \nabla{f}^{(i)}_{l}(x^{*}) \| + \| \nabla{f}^{(i)}_{l}(x^{*}) \|.
\end{split}
\end{equation*}}

Thus, squaring both sides and using $ 2ab \leq a^{2} + b^{2} $ with the Lipschitz assumption on the gradients along with the nonexpansiveness property and the boundedness of gradients in the set $ \mathcal{X} $ we have

{ \begin{equation*}\label{3.61}
\begin{split}
  \| \nabla{f}^{(i)}_{l}({\bf v}_{M}(k-1))- & \nabla{f}^{(i)}_{l}({\bf v}_{m}(k-1)) \|^{2} \\
 \leq \ & 2 L^{2}\| {\bf v}_{M}(k-1)- {\bf x}^{*} \| ^{2} + 2 L^{2}\| {\bf v}_{m}(k-1)- {\bf x}^{*} \| ^{2}
\end{split}
\end{equation*}}
i.e., we used $ \nabla{f(x^ {*})} = 0 $.
Therefore,

{ \begin{equation}\label{3.63}
\begin{split}
 \| & R_{i}  (k-1) \| ^{2} \\
  & \leq  \alpha_{k-1} ^{2} \|{\bf A}^{(i)} \| ^{2} _{\infty} \|{\bf B}^{(i)}\| _{2,\infty} ^{2}  \| \nabla{f}^{(i)}_{l}({\bf v}_{M}(k-1))-\nabla{f}^{(i)}_{l}({\bf v}_{m}(k-1)) \| ^{2} \\
 & \leq \alpha_{k-1} ^{2} \|{\bf A}^{(i)} \| ^{2} _{\infty} \| {\bf B}^{(i)}\| _{2,\infty} ^{2}( 2 L^{2}\| {\bf v}_{M}(k-1)- {\bf x}^{*} \| ^{2}\\
 & \ \ \ \ \  \ \ \ \ \ + 2 L^{2}\| {\bf v}_{m}(k-1)- {\bf x}^{*} \| ^{2} ). 
\end{split}
\end{equation}}

But we have

{ \begin{equation*}\label{3.64}
\begin{split}
 \ \ \ \ \ \ \ \  \| {\bf v}_{m}(k-1) - x^{*} \|^{2}  \leq \max_{k-1-H \leq \hat{k} \leq k-1} \| {\bf v}_{q}(\hat{k}) -  x^{*}\|^{2} 
\end{split}
\end{equation*}}
where $ q \in \{1,\ldots,q\} $.

{ \begin{equation*}\label{3.64}
\begin{split}
 \ \ \ \ \ \ \ \  \| {\bf v}_{M}(k-1) - x^{*} \|^{2}  \leq \max_{k-1-H \leq \hat{k} \leq k-1} \| {\bf v}_{q}(\hat{k}) -  x^{*}\|^{2} 
\end{split}
\end{equation*}}

where $ q \in \{1,\ldots,q\} $.

Then
{ \begin{equation}\label{3.67a}
\begin{split}
 & \| R_{i} (k-1) \|   \leq \\
 & \alpha_{k-1}  \| {\bf A}^{(i)} \| _{\infty} \|{\bf B}^{(i)}\| _{2,\infty} ( 2 L \max_{k-1-H \leq \hat{k} \leq k-1} \| {\bf v}_{q}(\hat{k}) -  x^{*}\| ) 
\end{split}
\end{equation}}

where $ q \in \{1,\ldots,n\} $.

and
{ \begin{equation}\label{3.67a}
\begin{split}
 & \| R_{i} (k-1) \| ^{2}  \leq \\
 & \alpha_{k-1} ^{2} \| {\bf A}^{(i)} \| ^{2} _{\infty} \|{\bf B}^{(i)}\| _{2,\infty} ^{2}( 4 L^{2}  \max_{k-1-H \leq \hat{k} \leq k-1, q \in V} \| {\bf v}_{q}(\hat{k}) -  x^{*}\|^{2} ) 
\end{split}
\end{equation}}

and consequently

{ \begin{equation}\label{3.67}
\begin{split}
& \sum_{i=1}^{n}  \| R_{i}  (k-1) \| ^{2}  \leq \\ 
& \alpha_{k-1} ^{2} \|{\bf A}^{(i)} \| ^{2} _{\infty} \|{\bf B}^{(i)}\| _{2,\infty} ^{2}( 4 L^{2} \sum_{i=1}^{n} \max_{k-1-H \leq \hat{k} \leq k-1 ; q \in V} \| {\bf v}_{q}(\hat{k}) -  x^{*}\|^{2} 
\end{split}
\end{equation}}

\subsection{Bounding $ \sum_{i=1}^{n} \| R_{i}(k-1) \| ^{2} $}\label{appB}

From Proposition ~\ref{sum_v_i-x^*_type2} or ~\ref{sum_v_i-x^*_type2a} we have

{ \begin{equation*}\label{3.65}
\begin{split}
\sum_{i=1}^{n}  \| {\bf v}_{i}(\hat{k}) -  x^{*}\|^{2} \leq D
\end{split}
\end{equation*}}

Then

\vspace{-0.3cm}
{ \begin{equation}\label{3.67}
\begin{split}
\sum_{i=1}^{n}  \| R_{i}  (k-1) \| ^{2}  \leq & \alpha_{k-1} ^{2} \|{\bf A}^{(i)} \| ^{2} _{\infty} \|{\bf B}^{(i)}\| _{2,\infty} ^{2}( 4  L^{2} D) \\
&  leq  \alpha_{k-1} ^{2}  F
\end{split}
\end{equation}}

where
\vspace{-0.5cm}
{ \begin{equation}
\begin{split}
 F =  4 L^{2} D \|{\bf A}^{(i)} \| ^{2} _{\infty} \|{\bf B}^{(i)}\| _{2,\infty} ^{2} < \infty     
\end{split}
\end{equation}}

\subsection{Lemmas}

\begin{lem}\label{L3.6}
[A variant of Lemma 2 in \cite{nedic2011random}] Let $\mathcal{Y} \subset \mathbb{R}^{N}$ be a closed convex set. Let the function $ \Phi : \mathbb{R}^{N} \rightarrow \mathbb{R}$ be convex and differentiable over $\mathbb{R}^{N}$ with Lipschitz continuous gradients with a constant $L$.\\
Let $ {\bf y} $ be given by 
${\bf y}={\bf x}-\alpha \nabla{ \Phi}({\bf x})$ for some ${\bf x} \in \mathbb{R}^{N}$, $\alpha > 0$.\\
Then, we have for any $\hat{{\bf x}} \in \mathcal{Y}$ and ${\bf z} \in \mathbb{R}^{N}$,
{ \begin{equation}\label{3.24}
\begin{split}
\|{\bf y}-\hat{{\bf x}}\| ^{2} \leq & \ (1+A_{\eta}\alpha^{2}) \|{\bf x}-\hat{{\bf x}}\| ^{2} - 2 \alpha ( \Phi({\bf z})- \Phi(\hat{{\bf x}})) \\
 & +(\frac{3}{8 \eta}+2 \alpha L ) \|{\bf x}- {\bf z} \|^{2}+B_{\eta}\alpha^{2} \| \nabla{ \Phi(\hat{{\bf x}})}\| ^{2},
\end{split}
\end{equation}}
where $A_{\eta}=2L^{2}+16 \eta L^{2}$, $B_{\eta}=2\eta + 8 $ and $ \eta > 0 $ is arbitrary.
\end{lem}

A variant of the above lemma (Lemma 2 in \citep{nedic2011random} or Lemma 4 in \cite{lee2013distributed}). Next we invoke Lemma 7(a) in \cite{lee2015asynchronous}. which is given by
\begin{lem}\label{L3.7}
Let $\mathcal{Y} \subset \mathbb{R}^{N}$ be a closed convex set. Let the function $ \Phi : \mathbb{R}^{N} \rightarrow \mathbb{R}$ be convex and differentiable over $\mathbb{R}^{N}$ with Lipschitz continuous gradients with a constant $L$.\\
Let $y$ be given by,\\
$y=\Pi_{\mathcal{Y}}({\bf x}-\alpha \nabla{ \Phi}({\bf x}))$ for some $x \in \mathbb{R}^{N}$, $\alpha > 0$.\\
Then, we have for any $\hat{{\bf x}} \in \mathcal{Y}$ and $z \in \mathbb{R}^{N}$,
{ \begin{equation*}
\begin{split}
\|{\bf y} & -\hat{{\bf x}}\| ^{2} \leq  \ (1+A_{\tau}\alpha^{2}) \| {\bf x} -\hat{{\bf x}}\| ^{2} - 2 \alpha ( \Phi({\bf z})- \Phi(\hat{{\bf x}}) ) \\ 
& -\frac{3}{4} \|{\bf y} - {\bf x} \| ^{2}  +(\frac{3}{8 \tau}+2 \alpha L ) \|{\bf x}- {\bf z} \|^{2}+B_{\tau}\alpha^{2} \| \nabla{ \Phi(\hat{{\bf x}})}\| ^{2},
\end{split}
\end{equation*}} 
where $A_{\tau}=8L^{2}+16 \tau L^{2}$, $B_{\tau}=8 \tau + 8 $ and $ \tau > 0 $ is arbitrary.
The proof is similar to Lemma 2 with the imposing of the non-expansive property to obtain the extra term.
\end{lem}

We begin with the following result, Lemma 6 in \cite{lee2013distributed}.

\begin{lem}\label{L3.10}
Let Assumptions~\ref{A3.1} and \ref{A3.3} hold. Consider the iterates generated by
\begin{equation}\label{3.77}
\begin{split}
\theta_{i}(k+1) = \sum_{j=1}^{n}[{\bf W}(k)]_{ij}\theta_{j}(k)+ e_{i}(k) \ for \ all \ i \in V. 
\end{split}
\end{equation}
Suppose there exists a non-negative non-increasing sequence $ \{ \alpha_{k} \} $ such that $ \sum_{k=0}^{\infty} \alpha_{k}|| e_{i}(k) || < \infty $ for all $ i \in V $, then for all $ i,j \in V $ we have 
\begin{equation}\label{3.78}
\begin{split}
\sum_{k=0}^{\infty} \alpha_{k}|| \theta_{i}(k) - \theta_{j}(k) || < \infty.
\end{split}
\end{equation}
\end{lem}

\end{appendix}

\bibliographystyle{unsrt}
\bibliography{ref}
\end{document}